\documentclass[12pt,reqno]{amsart} 
\usepackage{amssymb}
\usepackage{euscript}
\usepackage{versions}
\usepackage[usenames,dvipsnames] {color}
\let\cal\mathcal
\newtheorem{theorem}{Theorem}[section]
\newtheorem{remark}[theorem]{Remark}
\newtheorem{lemma}[theorem]{Lemma}
\newtheorem{corollary}[theorem]{Corollary}
\newtheorem{example}[theorem]{Example}
\newtheorem{definition}[theorem]{Definition}
\newtheorem{proposition}[theorem]{Proposition}
\numberwithin{equation}{section}

\def\span{\operatorname{span}}
\def\supp{\operatorname{supp}}

\def\qed{\hfill $\square$}
\def\D{\mathbb{D}}
\def\C{\mathbb C}
\def\O{\Omega}
\newcommand\al{\alpha}
\renewcommand\a{\alpha}
\newcommand\gam{\gamma}
\def\norm#1{\| #1 \|}
\newcommand\T{\mathbb{T}}

\newcommand\be{\begin{equation}}
\newcommand\ee{\end{equation}}
\newcommand\ov{\overline}

\newcommand\half{{\tfrac 12}}
\newcommand\calh{\mathcal{H}}
\newcommand\calk{\mathcal{K}}
\newcommand\calm{\mathcal{M}}
\newcommand\calp{\mathcal{P}}
\newcommand\cala{\mathcal{A}}
\newcommand\calb{\mathcal{B}}
\newcommand\cals{\mathcal{S}}
\newcommand\calv{\mathcal{V}}
\newcommand\calw{\mathcal{W}}

\newcommand\lam{\lambda}
\newcommand\inv{^{-1}}
\newcommand{\ip}[2]{\left\langle #1, #2 \right\rangle}

\newcommand\black{\color{black}}

\renewcommand\phi{\varphi}
\newcommand\fa{\mbox{ for all }}
\newcommand\bbm{\begin{bmatrix}}
\newcommand\ebm{\end{bmatrix}}
\newcommand\bpm{\begin{pmatrix}}
\newcommand\epm{\end{pmatrix}}
\newcommand{\threepartdef}[6]
{
	\left\{
	\begin{array}{lll}
		#1 & \mbox{ if } #2 \\
		#3 & \mbox{ } #4 \\
		#5 & \mbox{ if } #6
	\end{array}
	\right.
}
\renewcommand\vec[2]{\begin{pmatrix} #1 \\#2 \end{pmatrix} }

\begin{document}

\title[Singularities of functions]{A Hilbert space approach to singularities of
	functions}

\author{Jim Agler}
\address{Department of Mathematics, University of California at San Diego, CA \textup{92103}, USA}
\email{jagler@ucsd.edu}

\author{Zinaida A. Lykova}
\address{School of Mathematics,  Statistics and Physics, Newcastle University, Newcastle upon Tyne
	NE\textup{1} \textup{7}RU, U.K.}
\email{Zinaida.Lykova@ncl.ac.uk}

\author{N. J. Young}
\address{School of Mathematics, Statistics and Physics, Newcastle University, Newcastle upon Tyne NE1 7RU, U.K.
	{\em and} School of Mathematics, Leeds University,  Leeds LS2 9JT, U.K.}
\email{Nicholas.Young@ncl.ac.uk}

\date{20th December 2022}



\thanks{Partially supported by National Science Foundation Grants
	DMS 1361720 and 1665260, a Newcastle URC Visiting Professorship, the Engineering and Physical Sciences Research Council grant EP/N03242X/1 and London Mathematical Society Grant 42013}

\begin{abstract} 
	We introduce the notion of a {\em pseudomultiplier} of a Hilbert space $\cal H$ of functions on a set $\Omega$.  Roughly, a pseudomultiplier of $\calh$ is a function which multiplies a finite-codimensional subspace of $\cal H$ into $\cal H$, where we allow the possibility that a pseudomultiplier is not defined on all of $\Omega$.  A pseudomultiplier of $\calh$ has {\em singularities}, which comprise a subspace of $\mathcal H$, and generalize the concept of singularities of an analytic function, even though the elements of $\mathcal H$ need not enjoy any sort of analyticity.
We analyse the natures of these singularities, and obtain a broad classification of them in function-theoretic terms.
\end{abstract} 

\maketitle
\tableofcontents   
\black

\section{Introduction}\label{intro} 

Many different mathematical structures have been brought
to bear on the study of functions.  In the last half-century
notable theories which have yielded insights into
function-theoretic problems include those of locally convex
spaces, ordered vector spaces, commutative Banach
algebras and the geometry of Banach spaces.  Perhaps the
richest structure of all is that of Hilbert spaces; when a
function-theoretic problem can be formulated in terms of
Hilbert space geometry one can have high hopes of making
progress on it \cite{amy}.  Accordingly, there has been a substantial
development of the theory of Hilbert function spaces (also
called reproducing kernel Hilbert spaces and functional
Hilbert spaces) \cite{Ar}, \cite{Sai}.   In this paper we approach singularities of functions from
a Hilbert space perspective.

By a {\em Hilbert function space on a set $\Omega$} we mean a Hilbert space $\calh$ whose elements are
complex-valued functions on $\Omega$ such that, for each $\lam\in\Omega$, the linear functional $f \mapsto f(\lam)$ is continuous on $\calh$.  It follows that, for each $\lam\in\Omega$ there exists an element of $\calh$, which we shall denote by $k_\lam$, such that $f(\lam)=\ip{f}{k_\lam}$ for every $f\in\calh$.  We call $k_\lam$ the {\em kernel} corresponding to $\lam$, and we call the function $k:\Omega\times\Omega \to \C$ given by
\[
k(\mu,\lam)= k_\lam(\mu) \quad \fa \lam,\mu \in \Omega
\]
the {\em reproducing kernel} of $\calh$.

Even when we study problems which are not framed in Hilbert space terms we may still be able to use Hilbert space methods.  This seminal idea is
the basis of a substantial body of work stemming from
papers by Sarason \cite{S} and Adamyan, Arov and Krein
\cite{AAK} which has had a strong influence not only in
analysis but also in control theory \cite{G}.  In these papers
bounded analytic functions are analysed and constructed in
terms of their action as multipliers on the Hardy space
$H^2$. 
See the beginning of Section 2 for a definition of $H^2$.

Both the elements of $H^2$ and the multipliers of $H^2$
 are functions possessing a high degree of
regularity, and one might expect that the same
assertion will be true of most Hilbert function spaces one is
likely to encounter, and that therefore operator-theoretic
methods should be well adapted to problems involving
highly regular functions, but of less use for the study of
singularities of functions.  However, the cited paper of
Adamyan, Arov and Krein and its forerunner by Akhiezer
\cite{Ak} shows that this is not entirely true.  One can
construct meromorphic functions in the disc with a
prescribed number of poles by the spectral analysis of
operators on $H^2$.  In this approach a function with
singularities determines a multiplier not on the whole of
$H^2$ but on a closed subspace of finite codimension; such an object is called a {\em pseudomultiplier} of the Hilbert space in question.  The
power of operator theory in analysing functions with
singularities is thus far from vacuous, and we are led to
wonder about its potential and its limitations.

In \cite{AY1}  the domain of definition
of a pseudomultiplier, its
associated ``Pick kernel", and the particular case of Hilbert
spaces of analytic functions of one variable were studied.
The paper \cite{AY2} is
devoted to a key result of Adamyan, Arov and Krein and 
the extent to which it can have analogues for other function
spaces.  Pseudomultipliers also
appear naturally in connection with nearly invariant subspaces for backward shift operators on vector-valued Hardy spaces (see \cite{CCP}). 

In this paper we will develop the theory of pseudomultipliers and local subspaces.
The main notions and results are presented in the overview of the paper and illustrated by examples.

We are very grateful to a helpful referee who made numerous discerning comments and suggestions.

\section{Overview of the paper}\label{overview}

Before giving a formal definition of a pseudomultiplier, let us consider some archetypal examples.
The Hardy space $H^2$ is the space of analytic functions $f$ on the open unit disc $\D$ such that
\[
\sup_{0<r<1} \int_0^{2\pi} |f(re^{i\theta})|^2 \ d\theta < \infty.
\]
When endowed with pointwise addition and scalar multiplication and the inner product
\[
\ip{f}{g} = \lim_{r\to 1-} \frac{1}{2\pi} \int_0^{2\pi} f(re^{i\theta}) \overline{g(re^{i\theta})} \ d\theta
\]
$H^2$ is a Hilbert function space, whose reproducing kernel is the {\em Szeg\H{o} kernel $k$}, defined by
\[
k(\lam,\mu)= \frac{1}{1-\bar\mu\lam} \quad \fa \lam,\mu\in\D.
\]
An elementary introduction to $H^2$ can be found in \cite[Chapter 13]{y1988}. More thorough and far-reaching treatments are in \cite{hof62,koosis}.\\

\begin{center}{\sc 2.1. Pseudomultipliers}\label{pseudo1} \end{center}

\begin{definition} \label{mul1.1}
A {\em multiplier} 
of a Hilbert function space $\calh$ on a set $\Omega$ is defined to be a function $\phi:\Omega \to\C$ such that $\phi f \in \calh$ for every $f\in \calh$.  Here $\phi f$ denotes the pointwise product of $\phi$ and $f$.

\end{definition}
One can show directly from the definitions that every function in the space $H^\infty$ of bounded analytic functions on $\D$ is a multiplier of $H^2$.
  On the other hand, if $\phi$ is a multiplier of $H^2$, then since the constant function $\mathbf{1} \in H^2$, we have $\phi=\phi\mathbf{1} \in H^2$, and so $\phi$ is analytic on $\D$.  Moreover, the operator $X_\phi$ on $H^2$ defined,  for all $f\in H^2$, by
  \begin{equation}\label{1}
(X_\varphi f) (\lambda) = \varphi(\lambda) f(\lambda), \quad \fa
\lambda \in \Omega
\end{equation}
  is linear and is easily seen to have a closed graph, and hence $X_\phi$ is a bounded linear operator on $H^2$.  The calculation
\begin{eqnarray*}
\ip{X_\phi^* k_\lam }{ f }  & = &\ip{k_\lam}{X_\phi f} = \ip{k_\lam}{\phi f} \\
\;     &=& \overline{(\phi f)(\lam)} = \overline{\phi(\lam)}\ip{k_\lam}{f} \fa \lam\in\D\; \text{ and all }\; f\in H^2
\end{eqnarray*}
shows that $X_\phi^*k_\lam = \overline{\phi(\lam)} k_\lam$, so that $ \overline{\phi(\lam)}$ is an eigenvalue of $X_\phi^*$, and therefore $|\phi(\lam)| \leq \|X_\phi\|$ for all $\lam \in\D$, which is to say that $\phi$ is bounded on $\D$.  Thus the multipliers of $H^2$ are precisely the elements of $H^\infty$.

\begin{example}\label{arch1} { A pseudomultiplier of $H^2$.}  \rm
Consider the function $\phi(z)=1/z$, defined on the set $\D\setminus\{0\}$.  This function is clearly not a multiplier of $H^2$, since it is unbounded on $\D$, nor even defined on the whole of $\D$.  On the other hand it is close to being a multiplier in the sense that there is a finite-codimensional closed subspace of $H^2$, namely $zH^2$, which is a closed subspace of codimension $1$ in $H^2$ with the property that, for every $f \in zH^2$, $\phi f$ is the restriction to $\D\setminus \{0\}$ of a function in $H^2$. 
\end{example}
\begin{example}\label{arch2}  { Another pseudomultiplier of $H^2$.}  \rm
 Choose distinct complex numbers $a,b$ and define
 \[
 \psi(z) = \threepartdef{a}{z\neq 0}{}{}{b}{z=0.}
 \]
For any $f\in H^2$, it is easy to see that $\psi f \in H^2$ if and only if $f(0)=0$, that is, if and only if $f$ is in the closed $1$-codimensional subspace $zH^2$ of $H^2$.  Hence $\psi$ is a pseudomultiplier of $H^2$.
 \end{example}

In the light of these two examples we introduce the following notion.

\begin{definition}\label{1.1}
	{\rm Let $\cal H$ be a Hilbert function space on a set
		$\Omega$.  We say that a function
		
		$$
		\varphi: D_\varphi \subset \Omega \to \mathbb C
		$$
		is a {\it pseudomultiplier} of $\cal H$ if
		
		\begin{enumerate}
			\item [(1)] $D_\varphi$ is a set of uniqueness for $\cal H$;
			
			\item [(2)] the subspace $E_\varphi$ of $\calh$, defined to be
			\[
			\left\{h\in \calh: {\rm
				there\ exists}\ g\in \calh \ {\rm such\ that}\ g(\lambda)=
			\varphi(\lambda) h(\lambda)\ {\rm for\ all}\ \lambda \in
			D_\varphi\right\}
			\]
			is closed in $\cal H$, and
			
			\item [(3)]$E_\varphi$ has finite codimension in $\calh$.
		\end{enumerate}
		
	If $\calh$ is a Hilbert function space on a set $\Omega$ then we say that a subset $D$ of $\Omega$ is a {\em set of uniqueness for $\calh$} if, for any functions $f,g\in \calh$, if $f(\lam)=g(\lam)$ for all $\lam \in D$ then $f=g$.

		When conditions (1) to (3) hold we define the operator
		$X_\varphi:E_\varphi \to \cal{ H}$ by $X_\varphi h=g$ where
		$g$ is the element of $\cal H$ (necessarily unique, by condition (1))
		such that $g(\lambda) = \varphi(\lambda) h(\lambda)$ for all
		$\lambda \in D_\varphi$. 
		The space $E_\phi$ will be called the {\em regular space of $\phi$},
	    while the space $E_\phi^\perp$ will be called the {\em singular space} of $\phi$ and will be denoted by $\cals_\phi$.
	 }
\end{definition}
 
For the pseudomultiplier $\phi(z)=1/z$ on $H^2$ of Example \ref{arch1}, $D_\phi=\D\setminus\{0\}$ and $E_\phi$ is the closed subspace $zH^2$ of $H^2$, which has codimension $1$ in $H^2$. Thus $\cals_\phi= H^2\ominus zH^2$, the subspace of constant functions in $H^2$.  The operator $X_\phi:zH^2 \to H^2$ is given by
$X_\phi zf=f$ for all $f\in H^2$. 
For the  pseudomultiplier $\psi$  of $H^2$ of Example \ref{arch2},  $D_\psi =\D$,  $E_\psi=zH^2$, $\cals_\psi$ is the subspace of constant functions in $H^2$ and $X_\psi$ is the scalar operator $X_\psi f= af$ on $zH^2$.

The Hilbert function spaces on which pseudomultipliers act need not be spaces of analytic functions.

\begin{example}\label{intro2.3}{ A pseudomultiplier on a Sobolev space.}
	{\rm Let $\cal{ H}$ be the space $W^{1,2}[0,1]$ of absolutely continuous
		functions on $[0, 1]$ whose weak derivatives are in $L^2$ and
		with inner product given by 
		$$
		\ip{f}{g} = \int^1_0 f(t) \overline {g(t)} dt + \int^1_0
		f^\prime (t) \overline{ g^\prime(t)} dt.
		$$ 
It is well known that $W^{1,2}[0,1]$ is a reproducing kernel Hilbert space, with reproducing kernel $k$ given by 
	\begin{eqnarray}\label{kerW}
k(\lambda, \mu) = \left\{ \begin{array}{ll}
	{\rm cosech} 1 \cosh(1 -\mu) \cosh \lambda & \mbox{if $0 \leq \lambda \leq \mu \leq 1,$} \\
	{\rm cosech} 1 \cosh(1 -\lambda) \cosh \mu & \mbox{if $0 \leq \mu \leq  \lambda \leq 1$.}
\end{array} \right.
	\end{eqnarray}
 See, for example, \cite[Problem 6.10]{y1988}, where verification of the reproducing formula is left as an exercise, or \cite[Section 3]{aak} for a full calculation.

Let us check that $\chi(t)=\sqrt{t}$ is a pseudomultiplier of
	$\cal{ H}$ with $E_\chi = k^\perp_0$ and $\cals_\chi = \C k_0$.
	Consider a function $f\in k_0^\perp$, so that $f\in W^{1,2}[0,1]$ and $f(0)=0$. 
 Let us prove that the function  $\chi f$ is absolutely continuous and  $(\chi f)'$ is in $L^2(0,1)$.
Note that 
\[
(\chi f)'(t) = \frac{1}{2\sqrt{t}}f(t) + \sqrt{t}f'(t).
\]
To prove that  $(\chi f)'$ is in $L^2(0,1)$, it is enough to show that $\frac{1}{\sqrt{t}}f(t)$ is in 
 $L^2(0,1)$. Since  $f(0)=0$, for $x\in (0,1)$,
 \begin{eqnarray}\label{W11.1}
 x^{-\tfrac 12}f(x) &= &x^{-\tfrac 12}\int_0^x f'(t)\, dt \nonumber \\
 	&=&\sqrt{x} \left( x^{-1}\int_0^x f'(t)\, dt \right).
 \end{eqnarray}
 By assumption  $f'$ is in $L^2(0,1)$, and so, by Hardy's inequality,
 $\left( x^{-1}\int_0^x f'(t)\, dt \right)$ is in $L^2(0,1)$. Thus
  $\frac{1}{\sqrt{t}}f(t)$ is in 
 $L^2(0,1)$. Hence  $(\chi f)'$ is in $L^2(0,1)$. 
Therefore $(\chi f)' \in L^1(0,1)$, and so 
$$
(\chi f)(x) = \int_0^x (\chi f)'(t)\ dt
$$ 
is absolutely continuous on $[0,1]$.
Therefore, for $f\in k_0^\perp$, $\chi f \in W^{1,2}$.
Conversely, if $f(0)\neq 0$ then $(\chi f)'\notin L^2[0,1]$ and so $\chi f \notin W^{1,2}[0,1]$.  Thus, $\chi$ is a pseudomultiplier of $W^{1,2}$,  $E_\chi = k_0^\perp$ and $\cals_\chi=\C k_0$.	\qed
}\end{example}

Nor need it be true that the kernel vectors $\{k_\lambda:\lambda\in\Omega\}$ are linearly independent in $\calh$.
\begin{example}\label{ups} { Linearly dependent kernels}. \rm
Let $a,b$ be distinct points in $\D$ and let $\calh$ comprise the functions $\{f\in H^2:f(a)=f(b)\}$, with the inner product inherited from $H^2$. Thus, if $k$ is the reproducing kernel of $\calh$, we have $k_a=k_b$.  Consider an element $\upsilon\in H^\infty$ such that $\upsilon(a)\neq\upsilon(b)$.  It is clear that, for any $f\in\calh$, $\upsilon f\in\calh $ if and only if $f(a)=f(b)=0$.  Thus $E_\upsilon=k_a^\perp=k_b^\perp$, a closed subspace of $\calh$ of codimension $1$, and $\upsilon$ is a pseudomultiplier of $\calh$.\\
\end{example}

\begin{center}{\sc 2.2. The singular space} \label{singularspace}  \end{center}

 Informally speaking, the singular space $\cals_\phi$ is what stops the pseudomultiplier $\phi$ from being a multiplier.
  Pseudomultipliers that are not multipliers always have {\em singularities} in the sense that $\cals_\phi$ is non-zero. 
The reader will have noticed that in the three examples \ref{arch1}, \ref{arch2} and \ref{intro2.3} the pseudomultipliers $\phi,\psi$, and $\chi$ all fail to be multipliers of $\calh$, in each case on account of a lack of regularity at $0$; $\phi$ is unbounded in a neighborhood of $0$, $\psi$ is discontinuous at $0$ and 
 $\chi$ is not locally in $L^2$ at $0$.  These examples motivated us to call $\cals_\phi$ the {\em singular space} of $\phi$.

 In this paper we develop a structure theory for the singular spaces of pseudomultipliers. 
 The principal result is that $\cals_\phi$ can be orthogonally decomposed into the direct sum of
 the {\em polar space} of $\phi$, which we denote by $\calp_\phi$, and the set of {\em ambiguous vectors of $\phi$}, which we denote by $\cala_\phi$.
 The definitions of  $\calp_\phi$ and $\cala_\phi$, which we shall give shortly, depend on the notion of {\em $\phi$ seeing a vector}.\\

\begin{center}{\sc 2.3 Seeing a vector}\label{seeing} \end{center}

  Let $\cal{H}$ be a Hilbert function space on $\Omega$ and fix $v\in
\cal{ H}$.  
We shall formalise the idea of enlarging the set $\Omega$ by adjoining an extra point $p\notin\O$
in such a way that the vector $v$ is the reproducing kernel associated with $p$.
Let $\tilde\Omega = \Omega \cup \{p\}$, and for $f\in \cal{ H}$
define $ f_v:\tilde \Omega \to \mathbb C$ by the formula
$$
 f_v(\lambda)=\left\{ \begin{array}{ll} f(\lambda)\qquad  &
{\rm if}\qquad \lambda \in \Omega\\ \\
\ip{f}{v}\qquad & {\rm if} \ \lambda =
p\ .\end{array}\right.
$$ 
Let $\calh_v = \{ f_v: f\in \cal{H}\}$.  Since $f=0$ in
$\calh$ implies that $f_v=0$ in $\calh_v$, we see that
the formula
$$
\ip{f_v}{ g_v}_{\calh_v}\ =\ \ip{f}{g}_{\calh}
$$
defines an inner product on $\calh_v$.  Furthermore, with
this inner product, $\calh_v$ is a Hilbert function space on
$\tilde \Omega$, and the restriction map
$$
\calh_v \ni f_v \mapsto f_v|\Omega= f \in \calh
$$
is a Hilbert space isomorphism. 

Now let $\cal{ H}$ be a Hilbert function space, fix $v\in \calh$, and let $\varphi$ be a pseudomultiplier of $\cal{ H}$. 
Since $D_\varphi$ is a set of uniqueness for $\cal{ H}$,
$D_\varphi$ is also a set of uniqueness for $\calh_v$ and
thus $\varphi$ is also a pseudomultiplier of $\cal{ H}_v$.  
Its domain when so regarded is still just $D_\varphi$ and its
regular space is unchanged.  One can wonder, however, whether
 $\varphi$ can be extended to a pseudomultiplier of $\calh_v$ with domain containing $D_\phi\cup\{p\}$.

\begin{definition}\label{1.7}
	{\rm Let $\cal{ H}$ be a Hilbert function space on $\Omega$,
		let $v\in \cal{ H}$,  let $c\in \mathbb C$, and let $\varphi$ be a
		pseudomultiplier of $\calh$.  We say that $\varphi$ {\it sees $v$
			with value $c$} if there exists a pseudomultiplier $\psi$ of
		$\cal{ H}_v$ such that $\psi$ extends $\varphi$ (that is,
		$D_\varphi \subseteq D_\psi$ and $\varphi =
		\psi|D_\varphi$), $p \in D_\psi, E_\varphi = E_\psi
		|\O$, and $\psi(p) =c$.  We say $\varphi$ {\it sees} $v$
		if there exists $c\in \mathbb C$ such that $\varphi$ sees $v$
		with value $c$; in this case we also say that $v$ is {\it visible}
		to $\varphi$.}
\end{definition}

\begin{example}\label{1.75}
	\rm Let $\cal{H}$ be a Hilbert function space on $\Omega$ and let $\varphi$ be a
	pseudomultiplier of $\cal{H}$.  For any $c\in \C, \phi$ sees the zero vector in $\calh$ with value $c$.  To prove this fact,
	in Definition \ref{1.7} choose $v=0$ and $\psi$ to be given by $\psi(\lam)=\phi(\lam)$ for $\lam\in\O$ and $\psi(p)=c$. One finds that $f_v\in E_\psi$ if and only if $f\in E_\phi$ and that $\psi$ is a pseudomultiplier of $\calh_v$.  Thus $\phi$ sees $0$ with value $c$.	
\end{example}

\begin{example}\label{1.8}
\rm Let $\cal{ H}$ be a Hilbert function space on $\Omega$
		and let
		$\varphi$ be a pseudomultiplier of $\cal{ H}$.  If $\lambda \in
		D_\varphi$, then $\varphi$ sees $k_\lambda$ with value
		$\varphi(\lambda)$.
				
To prove this statement, let $\tilde\O = \O\cup\{p\}$, where $p \notin\O$, and let $D_\psi=D_\phi\cup\{p\}$.  Define a function 
$\psi:D_\psi \to \C$ by $\psi(z)=\phi(z)$ if $z\in D_\phi$ and $\psi(p)=\phi(\lam)$.  Let $v=k_\lam$, so that $\calh_v$ is a Hilbert space of functions on $\tilde\O$.  We claim that $\psi$ is a pseudomultiplier of $\calh_v$ and that $E_\phi=E_\psi|\O$. Indeed, consider any function $f\in\calh$.  We have
\begin{eqnarray*}
\hspace{-1cm} f_v \in E_\psi & \Leftrightarrow & \psi f_v \in \calh_v|_{D_\psi}        \nonumber\\
& \Leftrightarrow & \text{there exists } \; g\in \calh \; \; \text{such that}\; (\psi f)_v=g_v|_{D_\psi} 
\end{eqnarray*} 
Now, for any $g\in\calh$, $ (\psi f)_v=g_v|_{D_\psi}$ if and only if $\phi(z)f(z)=g(z)$ for $z\in D_\phi$ and $\psi(p)f_v(p)=\ip{g}{v}$, which is to say that $\phi f=g$ on $D_\phi$ and $\phi(\lam)\ip{f}{v}=\ip{g}{v}$.
Thus $f_v \in E_\psi$ if and only if $\phi f$ extends to an element $g\in\calh$ and $\phi(\lam)f(\lam)=g(\lam)$.
Thus $f_v\in E_\psi$ if and only if $f\in E_\phi$.  Hence, by Definition \ref{1.7}, $\phi$ sees $k_\lam$ with value $\phi(\lam)$, as claimed. \qed
\end{example}

\begin{example}\label{1.9int} What vectors does $1/z$ see in $H^2$? \rm
Consider the pseudomultiplier $\phi$ of Example \ref{arch1}: $\phi(z)=1/z$ on $(\calh,\Omega)=(H^2,\D)$.  Here, $D_\phi=\D\setminus\{0\}$ and $E_\phi= zH^2$. 
We claim that the vectors in $\calh=H^2$ visible to $\phi$ are precisely the scalar multiples of the kernels $k_\lam$, for $\lam\in \D\setminus\{0\}.$ 

To prove this claim, we first note that, by Example \ref{1.8}, $\phi$ does see $k_\lam$ with value $\phi(\lam)=1/\lam$ for every $\lam\in D_\phi$. Conversely, suppose $\phi$ sees a vector $v\neq 0$ with value $c$.  By definition, there exists a pseudomultiplier $\psi$ of $\calh_v$ such that $\psi$ extends $\phi$, $p \in D_\psi, E_\varphi = E_\psi|\D$, and $\psi(p) =c$.  Consider any function $h\in H^2$. The function $f$ defined by $f(z)=zh(z)$ is in $E_\phi=zH^2$.  Hence there exists $\tilde f\in E_\psi$ such that $\tilde f|\D\setminus\{0\}=f$.  We have $\psi\tilde f \in \calh_v$, and therefore 
\[
(\psi\tilde f)(p) = \ip{(\psi\tilde f)|\D\setminus\{0\}}{v}=\ip{\tfrac 1z f}{v}= \ip{\tfrac 1z zh}{v}= \ip{h}{v}.
\]
Now $(\psi\tilde f)(p) =\psi(p) \tilde f(p) = c\ip{f}{v}= c\ip{zh}{v}$.
Therefore, for every $h\in H^2$, $c\ip{zh}{v}=\ip{h}{v}$, and so, if $S$ denotes the forward shift operator on $H^2$, given by $(Sh)(z)=zh(z)$ for all $h\in H^2$, then 
\[
\ip{h}{v}= \ip{czh}{ v} =  \ip{cSh}{ v}= \ip{h}{\bar c S^*v}.
\]
Hence $\bar c S^* v= v$.  Since, by assumption, $v\neq 0$, $v$ is an eigenvector of $S^*$ corresponding to the eigenvalue $1/\bar c$.  It is a simple calculation to show (see for example \cite[Lemma 1.67]{amy})  that the eigenvectors of $S^*$ are the non-zero scalar multiples of the Szeg\H{o} kernels $k_\lam$ for $\lam\in\D$, that is, $S^*k_\lam = \bar\lam k_\lam$.  Hence $\bar c \bar\lam k_\lam= k_\lam$, and so $c\lam =1$, and we conclude that $\lam\in\D\setminus\{0\}$, and that $v$ is a scalar multiple of the kernel $k_\lam$, for some $\lam\in \D\setminus\{0\}$, as claimed. \qed 
\end{example}

\begin{center}{\sc 2.4. Ambiguous points and the ambiguous space}\label{ambiguity} \end{center}

In order to define the ambiguous space $\cala_\phi$ of a pseudomultiplier $\phi$ we first consider the points in $\Omega$ at which the value of $\phi$ is not determined by the operator $X_\phi$. We call such points {\em ambiguous points} of $\phi$.
For $c\in \mathbb C$ and $\alpha\in D_\phi$, define the function
$\varphi_{c, \alpha}$ on $D_\varphi$ by
\begin{equation}\label{modify}
\varphi_{c, \alpha} (\lambda) = \left\{ \begin{array}{cl}
\varphi(\lambda) & \ \  {\rm if}\quad \lambda \in
D_\varphi\setminus\{\alpha\}\\ \\
c& \ \ {\rm if}\quad \lambda = \alpha.\end{array}\right.
\end{equation}
Typically, modification of $\varphi$ in
this manner will affect the operator $X_\varphi:E_\varphi \to
\cal{ H}$.  This modification could happen in two possible ways: $E_\varphi$
could be changed, or $E_\varphi$ could remain unaltered but
the action of $X_\varphi$ could be changed.

\begin{definition}\label{2.1}
	{\rm Let $\varphi$ be a pseudomultiplier and let $\alpha \in	D_\varphi$.  We say $\alpha $ is an {\it ambiguous point for }
		$\varphi$, or an {\em ambiguity of $\phi$}, if there exists $c\in \mathbb C$ with the properties
		that $c\not= \varphi(\alpha),\; E_\varphi = E_{\varphi_{c, \alpha}}$, and
		$X_{\varphi_{c, \alpha}} = X_\varphi$.  We say $\alpha$ is an {\it
			unambiguous point for} $\varphi$ if  $\alpha$ is not an
		ambiguous point for $\varphi$.}
\end{definition}

The pseudomultiplier $\psi$ of $H^2$ given in Example \ref{arch2} has a discontinuity at $0$ which is an ambiguous point for $\psi$.  Indeed, for any choice of $c \in \C\setminus\{a,b\}$, the one-point modification $\psi_{c, 0}$ of $\psi$ is still an instance of Example \ref{arch2}, so that $E_{\psi_{c, 0}}=k_0^\perp= E_\psi$.  Moreover, for any $f\in k_0^\perp$,
$X_\psi f=af =X_{\psi_{c, 0}} f$.  Another example of an ambiguous point is furnished by the pseudomultiplier $\chi$ in Example \ref{intro2.3}.  It has an ambiguous point at $0$, as we show in Example \ref{2.3} below.

 In Section \ref{ambiguous}, Theorem \ref{2.14}, we shall prove the following connection between visibility of vectors and ambiguous points.
\begin{theorem}\label{2.14intro}
 Let $\phi$ be a pseudomultiplier of a Hilbert function space $\calh$ on a set $\Omega$ and let  $\alpha \in\Omega$. Then $\alpha$ is an ambiguous point for 
$\varphi$ if and only if $\varphi$ sees $k_\alpha$ with
arbitrary value.
\end{theorem}

Theorem \ref{2.14intro} leads naturally to a more general notion, an
{\em ambiguous vector} for a pseudomultiplier. 

\begin{definition}\label{3.1}
	{\rm Let $\varphi$ be a pseudomultiplier of $\cal{ H}$ and
		let
		$v\in \cal{ H}$.  We say that {\it $v$ is an ambiguous vector
			for }
		$\varphi$ if $\varphi$ sees $v$ with arbitrary value, and
	we define  $\cal{ A}_\varphi$ to be the set of ambiguous vectors for $\varphi$. }
\end{definition}

It follows from Theorem \ref{2.14intro} that if $\alpha$ is a ambiguous point for $\phi$ then $k_\alpha$ is an ambiguous vector for $\phi$. 
Observe also that Example \ref{1.75} immediately implies that, for any Hilbert function space $\calh$, the zero vector of $\calh$ is an ambiguous vector of every pseudomultiplier of $\calh$.

\begin{example}\label{arch1-A}  $\phi(z)=1/z$ on $H^2$ has no non-zero ambiguities. \rm By Example \ref{1.9int},  the pseudomultiplier $\phi(z)=1/z$ of $H^2$ sees only the scalar multiples of kernels $k_\lam$ for some 
$\lam\in\D\setminus\{0\}$. Moreover, for such $\lam$, $\phi$ sees $k_\lam$ only with value $\phi(\lam)$.  Thus $\phi$ does not see any non-zero vector with arbitrary value, which is to say that $\phi$ has no ambiguous vectors other than $0$.
\end{example}

\begin{example}\label{ambpsi-int} \rm 
	Consider the pseudomultiplier $\chi(t)=\sqrt{t}$ on $W^{1,2}[0,1] $  of Example \ref{intro2.3}. 
	We will show in Example \ref{ambchi}, that  $\cala_\chi=\C k_0$.	
\end{example}

\begin{center}{\sc 2.5. Definable vectors and the polar space}\label{definable} \end{center}

In addition to singularities that arise from ambiguous vectors, pseudomultipliers can possess another type of singular vector, which we call a {\em polar vector}.
  We say that a vector
$d\in \cal{ H}$ is {\it definable for} $\varphi$ if $d\not= 0,
\varphi$ sees $d$, and $d\perp \cal{ A}_\varphi$.  We denote by
$\mathbf{D}_\varphi$ the set of vectors that are
definable for $\varphi$.
For any definable vector $v$ for a pseudomultiplier $\phi$, there is a unique $c\in\C$ such that $\phi$ sees $v$ with value $c$.  We denote this number $c$ by $\phi(v)$.  Thus we can regard $\phi$ as a function on $\mathbf{D}_\phi$.

\begin{example} \label{definableex} The definable vectors for $1/z$ on $H^2$.  \rm  By Example  \ref{arch1-A}, for the pseudomultiplier $ \phi(z)=1/z$ of $H^2$, $ \cala_\phi= \{0\}$, and so the definable vectors for $ \phi$ coincide with the non-zero vectors visible to $ \phi$.  Hence, by Example  \ref{1.9int}, 
$\mathbf{D}_\phi$  comprises the non-zero scalar multiples of the
kernels $k_\lam $, for $\lam \in  \D \setminus \{0\}$.	
\end{example}

\begin{definition}\label{3.11}
\rm Let $\varphi$ be a pseudomultiplier and let $p\in \cal{H}$.  We say $p$ is a {\it polar vector for} $\varphi$ if $p\not= 0$ and there exists a sequence of definable vectors $\{d_n\} \subseteq
		 \mathbf{D}_\varphi$ such that $d_n\to p$ (with respect to 
		 $\|\cdot \|_{\cal{H}}$) and	 $\varphi(d_n) \to \infty$   as $n\to \infty$.
We define $\cal{ P}_\varphi$, the
	{\it polar space of
		$\varphi$}, to be the set of vectors $v\in \cal{ H}$ such that
	either $v$ is a polar vector of $\varphi$ or $v=0$.
\end{definition}

\begin{remark} \label{pseudopole}  \rm In \cite[Section 1]{AY1} we introduced a notion more directly modelled on the familiar notion of a pole from Complex Analysis.  If $\phi$ is a pseudomultiplier of a Hilbert function space $\calh$ on a set $\O$ we said in \cite{AY1} that a point $\al\in\O$ is a {\em pseudopole} of $\phi$ if there exist functions $g \in\calh$ and $h\in E_\phi$ such that $h(\al)=0$, $g$ extends $\phi h$ and $g(\al) \neq 0$. 

From the perspective of this paper, pseudopoles are superseded by the more general notion of polar vectors. Indeed, we shall prove in Proposition \ref{thm-pseudopole} that  if $\al$ is a pseudopole of a pseudomultiplier $\phi$ and $\al$ is in the closure of $D_\phi$ in $\O$  then $ P_{\cala_\phi^\perp} k_\al$ is a polar vector of $\phi$.  Here, by the closure of $D_\phi$ we mean the closure in $\Omega$  with respect to the natural metric on $\O$, induced by the norm on $\calh$ and the identification of a point of $\O$ with its reproducing kernel in $\calh$.  \qed
\end{remark}

\begin{example}\label{polar-of-1/z} The polar space of $1/z$. \rm 
In Example \ref{definableex} we proved that the set of definable vectors $\mathbf{D}_\phi$ for $\phi(z)=1/z$ on $H^2$  comprises the scalar multiples of the kernels $k_\lam $, for $\lam \in  \D \setminus \{0\}$. Recall that the singular space $\cals_\phi=E_\phi^\perp = H^2\ominus zH^2$. Let us show  that
  $$\calp_\phi=\cals_\phi=E_\phi^\perp = H^2\ominus zH^2 =\{  \text{constant functions on} \; \D \}. $$ 
  It is enough  to show that $1= k_0 $ is in $\calp_\phi$. 
Choose $\lambda_n \in \D$, $n \ge 1$, such that $\lambda_n \to 0$. By Example \ref{1.8}, $\phi $ sees $k_{\lambda_n} $ with value $\frac{1}{\lambda_n} $. Since $k_{\lambda_n} \to k_{\mu}$ in $H^2$ if and only if  $\lambda_n \to \mu$ in the usual topology of $ \D$,
   $\{k_{\lambda_n}\}$ is a sequence of definable vectors such that $k_{\lambda_n} \to k_{0}$ in $H^2$ and 
$\phi(k_{\lambda_n}) =\phi(\lambda_n) = \frac{1}{\lambda_n} \to \infty$. This implies that $ k_0 \in \calp_\phi $. \qed
\end{example}

In Section \ref{polar-vectors}, Theorem \ref{3.14},
we shall prove the following fact.

\begin{theorem}\label{3.14intro}
	Let $\varphi$ be a pseudomultiplier of $\calh$.  Then $\varphi$ is locally
	bounded on $\mathbf{ D}_\varphi^-$ if and only if $\varphi$ has no
	polar vectors.
\end{theorem}
	
\noindent In Section \ref{polar-vectors},  Theorem \ref{4.3}, we shall also establish an alternative description of the polar space in terms of definable vectors.
\begin{theorem}\label{4.3intro}
If $\calh$ is infinite-dimensional and $\phi$ is a pseudomultiplier of $\calh$  then $\calp_\phi=\ov{\mathbf{D}_\phi }\setminus \mathbf{D}_\phi$.
\end{theorem}

\begin{center}{\sc 2.6. Decomposition of the singular space}\label{decomposition}  \end{center}

In Theorem \ref{4.2} we shall prove that every singular vector of a pseudomultiplier can be  represented  in a unique way as the sum of a polar vector and an ambiguous vector.

\begin{theorem}\label{4.2intro}
 Let $\cal H$ be a Hilbert function space on a set
$\Omega$.
	Let $\varphi$ be a pseudomultiplier of $\calh$.  $\cal{ P}_\varphi$ is a closed
	subspace of $\cal{ H}$ and 
	$$\cal{ S}_\varphi=\cal{ P}_\varphi \oplus \cal{A}_\varphi.$$
	\end{theorem}
		
This theorem contains the very  interesting fact that the polar space $\cal{ P}_\varphi$, which is the set of polar vectors of a pseudomultiplier $\varphi$, together with the zero vector, constitutes  a linear subspace  of $\cal H$. 

\begin{example}\label{2dim-polarspace} \rm 
The pseudomultiplier	  $\phi(z)=\frac{1}{z(z-\half)}$ on $H^2$ has no ambiguities
$\cala_\phi=\{0\}$, and the polar space of $\varphi$
$$\cal{ P}_\varphi =  \span\{k_{0}, k_{\half}\} = H^2\ominus z (z -\half)H^2 = E_\phi^\perp.$$
See Example \ref{2kernels-phi} for proofs.
\end{example}

\begin{example}\label{polarspaces} \rm If 
$\chi(t)=\sqrt{t}$ on 
$W^{1,2}[0,1]$ is the pseudomultiplier of
 Example \ref{intro2.3},  then, as Example \ref{ambpsi-int} implies,
  the ambiguous space $\cala_\chi =  \C k_0 = \cals_\chi$, and so we deduce from Theorem \ref{4.2} that $\calp_\chi=\{0\}$.  Thus $\chi$ has no polar vectors.
  This example illustrates Theorem \ref{3.14intro}.\\
\end{example}

\begin{center}{\sc 2.7. Locality}\label{locality}  \end{center}

In Section \ref{localness} we shall introduce the notion of a visible subspace for a pseudomultiplier.
We generalise the notion of $\phi$ seeing a vector to $\phi$ seeing a closed subspace of $\calh$.
We say that $\phi$ sees a closed subspace $\cal V$ if there is a bounded linear operator $C:\cal{ V}
		\to \cal{ V}$ such that, for all $f\in E_\varphi$,
		\begin{equation}\label{22int}
		P_\cal{V} X_\varphi f = C P_\cal{ V} f
		\end{equation}
where $P_\cal{ V}$ is the orthogonal projection of $\calh$ onto $\cal{V}$.

\begin{example}\label{5.4.int}
	{\rm Let $\varphi \in H^\infty$, so that $\varphi$ is a
		multiplier of $H^2$. 		
		 The only vectors in $H^2$ seen by
		$\varphi$ are the eigenvectors of $X^*_\varphi$.  In
		particular,  if $\varphi$ is such that $\varphi'(0) \neq 0$, then $\varphi$ does not see the  vector		$v(\lambda) = \lambda, \lambda
		\in \mathbb D$.  However, $\varphi$ does see the space
		$\cal{ V} \stackrel{\rm def}{=} \span \{1, v\}$: for $f\in
		H^2, P_\cal{ V} X_\varphi f$ depends only the first two
		Taylor coefficients of $f$, that is, on $P_\cal{ V} f$.
		Thus $\phi$ sees $\cal V$, and there is a {\it unique} $C$
		such that equation (\ref{22}) holds for  all $f\in E_{\varphi}$.		\qed}
\end{example}
We observe that in the case when $\varphi$ is a
multiplier of $\cal{ H}$, $\phi$ sees a closed subspace $\cal V$ of $\calh$
if and only if  $\calv$ is an invariant subspace for $X^*_\varphi$. 
In Proposition \ref{5.5} we show that  a closed subspace $\calv$ is visible to a psudomultiplier $\phi$ if and only if $X_\phi^*\calv \subseteq \calv + E_\phi^\perp$.

In Section \ref{singularities}
we introduce the idea of a {\em local subspace} of a
Hilbert function space over a subset $D$ of $\Omega$. 
In Section \ref{local_spaces} we shall clarify the sense in which a subspace
which is local over a set $D$ can genuinely be attached to a
set of points in $D$. The following key example shows the naturalness of the notion of a local subspace.

\begin{example}\label{loc-sub-intro} \rm
 Consider the subspace
$\cal{ M}=\span\{1, z\}$ of $H^2$.  The elements of $\cal{M}^\perp$ are the functions which vanish together with their
first derivatives at $0$, and we accordingly regard $\cal{ M}$
as being local at $0$.  Notice that
$\cal{ M}$ is not spanned by kernel functions -- rather, it is a
limit of two-dimensional spaces spanned by kernels.  For
$f\in H^2$ we have
$$f'(0) = \lim\limits_{\alpha \to 0} \frac{f(\alpha) -
	f(0)}{\alpha} = \lim\limits_{ \alpha\to 0} \ip{f} {\frac{k_\alpha
	-k_0}{\overline \alpha}}.
$$  
In a sense to be made
precise,
$\cal{ M}$ is the limit as $\alpha \to 0$ of the space
$\span\{1,
\overline \alpha^{-1}(k_\alpha-1)\}$ 
(which equals $\span\{1, k_\alpha\}$).  Likewise $\span\{1, z, z^2\}$ is
the limit  as
$\alpha, \beta \to 0$ of \newline $\span\{1, k_\alpha, k_\beta\}$. 
\end{example}
This example provides the model for the notion of a
local subspace of a general Hilbert function space.

Let $D \subset \Omega$. A closed subspace $\cal M$ of
$\cal H$ is {\it local over} $D$ if $\calm$ is a cluster point of a collection of spans of uniformly boundedly many kernels of points in $D$ (see Definition \ref{6.8}).
 We say that $\alpha \in
\Omega$ is a {\it support point } of $\cal{ M}$ if there exist a positive integer $n$, a sequence $(\alpha_j)$ in $\Omega$, a sequence $(c_j)$ in
$\mathbb C$ and a sequence $(E_j)$ of subsets of $\Omega$
such that $\#E_j = n, \alpha_j \in E_j, \span\{k_\lam:\lam\in E_j\}\to \cal{M}$ and $c_jk_{\alpha_j}\to k_\alpha$.  The {\it support} of
$\cal{ M}$ is the set of support points of $\cal{ M}$.  It will
be denoted by $\supp \cal{ M}$.  We say that $\calm$ is {\em punctual} if $\supp \calm$ is a singleton set. Example \ref{loc-sub-intro} shows that the local subspace
$\cal{ M}=\span\{1, z\}$ of $H^2$ is punctual with $\supp \cal{ M} = \{0\}$. 

The main result of Section \ref{singularities} is Theorem \ref{6.11}, which asserts that all local spaces are seen by all pseudomultipliers.  This fact confirms the importance of the notion of local subspaces.

\begin{theorem}\label{6.11intro}
	Let $\varphi$ be a pseudomultiplier of $\cal H$.  If $\cal M$ is a
	subspace of $\cal H$ which is local over $D_\varphi$,
	then $\phi$ sees $\cal M$.
\end{theorem}

In light of Theorem \ref{6.11intro} it appears that the structure of local subspaces is an important component of the theory of pseudomultipliers.  Therefore, in the final section of the paper we study such local subspaces. In Theorem \ref{6.18}, we establish  the structure
of local subspaces $\cal M$ of projectively complete Hilbert function spaces $\cal H$, as follows.
  
\begin{theorem}\label{6.18intro}
	Let $\cal{ H}$ be a projectively complete Hilbert function
	space on $\Omega$, let $D\subset \Omega$ and let $\cal{
		M}$ be a local space over $D$.  There exist subspaces $\cal{
		M}_1, \dots, \cal{ M}_N$ of $\cal{ H}$ such that
	\begin{itemize}
		\item[\rm (i)] $\cal{ M}_i$ is punctual for $i=1, \dots, N$;
		
		\item[\rm (ii)] $\supp \cal{ M}_i= \{\alpha_i\} \subset {\rm pc}(D), 1 \le
		i \le N$;
		
		\item [(iii)]$\cal{ M} = \cal{ M}_1+\cal{ M}_2+\dots + \cal{
			M}_N$;
		
		\item[\rm (iv)] $\sum\limits^N_{i=1} \dim \cal{ M}_i = \dim
		\cal{ M}$, and
		
		\item[\rm (v)] $\supp \cal{ M}=\{\alpha_1, \dots, \alpha_N\}$.		
	\end{itemize}
\end{theorem}
In this theorem,  $\mathrm{pc}(D)$ denotes the projective closure of a subset $D$ of $\O$, which is defined immediately after Definition \ref{6.17}.

\section{Examples of pseudomultipliers}\label{pseudomultipliers}
 
In Definition \ref{1.1} we introduced pseudomultipliers $\phi$ and their associated multiplication operators $X_\phi$, as a generalization of the notion of a multiplier (as in Definition \ref{mul1.1}) of a Hilbert function space. 
Firstly, we
allowed the possibility that $\varphi$ only be defined on a
proper subset of $\Omega$ and secondly we allowed
$X_\varphi$ only to be defined on a proper subspace of $\cal{
H}$ of finite codimension.

Let $\cal H$ be a Hilbert function space on a set
		$\Omega$.
We will start with examples of functions
			$$
		\varphi: D_\varphi \subset \Omega \to \mathbb C
		$$
		which are not pseudomultipliers.
Firstly, we assume that $D_\varphi$ is a set of uniqueness for $\cal H$.
Let the subset $E_\varphi$ of $\calh$ be defined by
			\[
			\left\{h\in \calh: {\rm
				there\ exists}\ g\in \calh \ {\rm such\ that}\ g(\lambda)=
			\varphi(\lambda) h(\lambda)\ {\rm for\ all}\ \lambda \in
			D_\varphi\right\}
			\]	
There are two ways that a function $\varphi$
defined on a set of uniqueness can fail to be a
pseudomultiplier:  $E_\varphi$ might not be closed or
$E_\varphi$ might be closed but not of finite codimension.
	
	Let		$X_\varphi$ be the operator
		$X_\varphi:E_\varphi \to \cal{ H}$ given by $X_\varphi h=g$ where
		$g$ is the element of $\cal H$
		such that $g(\lambda) = \varphi(\lambda) h(\lambda)$ for all
		$\lambda \in D_\varphi$. 			
\begin{remark}\label{Ephi-Xphi} \rm
In the case that $E_\varphi$ is closed
in $\cal H$, it is an easy consequence of the closed graph
theorem that $X_\varphi$ is bounded.  Conversely, if
$X_\varphi$ is bounded then $E_\varphi$ is closed.  For
suppose $f_n \in E_\varphi$ and $f_n\to g$ in $\cal H$.  Then
the sequence $(X_\varphi f_n)$ is bounded in $\cal H$ and, by
passing to a subsequence if necessary, we can suppose that
$X_\varphi f_n$ tends weakly to some $h\in \cal{ H}$.  For any
$\lambda \in D_\varphi$ we have
$$(X_\varphi f_n) (\lambda) = \varphi(\lambda) f_n(\lambda)
\to \varphi(\lambda) g(\lambda)$$
and
$$(X_\varphi f_n) (\lambda) \to h(\lambda).$$
Thus $h\in \cal{ H}$ extends $\varphi(g|D_\varphi)$, and so
$g\in E_\varphi$.  Thus $E_\varphi$ is closed. In particular, if
$E_\varphi$ has finite codimension, then $\varphi$ is a
pseudomultiplier if and only if $X_\varphi$ is a bounded
operator.
\end{remark}

\begin{example}\label{1.2}
{\rm  Let $\Omega=\mathbb D$ and let $\cal{ H} = H^2$, the
classical Hardy space.  Let $D_\varphi = \mathbb D$ and
define $\varphi$ by $\varphi(z) = (1-z)^{-1}$.  Here
$E_\varphi=(1-z)H^2$ which is a nonclosed subspace of
$H^2$.  Thus $\varphi$ is not a pseudomultiplier. \qed}
\end{example}

\begin{example}\label{1.2.2}
{\rm  Let $\Omega=\mathbb D$ and let $\cal{ H} = H^2$. Let $\psi$ be an inner function on $\D$ which is not a finite Blascke product and let $\phi=1/\psi$.
  Let
$D_\varphi = \D\setminus \psi\inv\{0\}$.
 Here
$E_\varphi= \psi H^2$  is a  closed
subspace of
$\cal H$.  Yet $\varphi$ is not a pseudomultiplier of $\cal H$
since $E_\varphi$ has infinite codimension in $\cal H$. \qed}
\end{example}

\begin{example}\label{1.3}
{\rm  Let $\Omega=\mathbb D^2$ and let $\cal{ H} =
H^2(\mathbb D^2)$, the  Hardy space on the bidisc.  Let
$D_\varphi =
\{z\in \mathbb D^2:z_1\not= 0\}$ and define $\varphi$
by
$\varphi(z) = z^{-1}_1$.  Here
$E_\varphi= z_1H^2(\mathbb D^2)$  is a  closed
subspace of
$\cal H$.  Yet $\varphi$ is not a pseudomultiplier of $\cal H$
since $E_\varphi$ has infinite codimension. \qed}
\end{example}

There
are also two ways that a pseudomultiplier $\varphi$
can fail to be a multiplier: $D_\varphi = \Omega $ but
$E_\varphi \not= \cal{ H}$ or $E_\varphi = \cal{ H}$ but
$D_\varphi \not= \Omega$.

\begin{example}\label{1.4}
{\rm  Let $\Omega$ be a nonempty open proper subset of
$\mathbb D$ and let
$\cal{ H}$ be the space of analytic functions $f$ on $\Omega$
that have an extension $f^\sim$ to $\mathbb D$ such that
$f^\sim \in H^2$.  Let $\cal H$ have the inner product defined
by
\begin{equation}\label{2}
\ip{f}{g}_\cal{ H}\ =\ \ip{f^\sim}{g^\sim}_{H^2}.
\end{equation}
Fix $\alpha \in \mathbb D\backslash \Omega$ and define
$\varphi$ on $\Omega$ by $\varphi(z) = (z-\alpha)^{-1}$. 
Then $D_\varphi = \Omega$ and, since equation (\ref{2}) implies that
we may identify $\cal H$ with $H^2$, we have 
$$
E_\varphi = \{f\in \cal{ H}: f^\sim (\alpha) = 0\},
$$
a closed subspace of $\cal H$ of codimension 1.  Hence
$\varphi$ is an everywhere-defined pseudomultiplier of $\calh$ but is
not a multiplier of $\calh$. \qed}
\end{example}

\begin{example}\label{1.5}
{\rm  Let $\Omega=\mathbb D$ and  $\cal{ H} = H^2$. Fix a
bounded analytic function $\psi$ on $\mathbb D$ and let
$S\subset
\mathbb D$ be a set of uniqueness for $\cal H$.  If $\varphi=
\psi|S$ then  $E_\varphi= \cal{ H}$ and $\varphi$ is a
pseudomultiplier of $\cal H$.  However, $\varphi$ is not a
multiplier unless $S=\mathbb D$.}
\end{example}

\begin{example}\label{1.6}
{\rm  Let $\Omega_1=\mathbb D\times \{1\}, \Omega_2=
\mathbb D\times \{2\}$ and 
$\Omega = \Omega_1 \cup \Omega_2$. Fix a linear
transformation $L$
of $H^2$.  For $f\in H^2$ define $f^\sim$ on $\Omega$ by the
formula
$$f^\sim((\lambda, i)) = \left\{\begin{array}{ll} 
f(\lambda)\qquad&
{\rm if}\qquad i=1,\\ \\
(Lf)(\lambda)\qquad & {\rm if}  \qquad i=2.
\end{array}\right.
$$
Let $\cal{ H} = \{f^\sim: f\in H^2\}$ with inner product
\begin{equation}\label{3}
\ip{f^\sim}{g^\sim}_\cal{ H} = \ip{f}{g}_{H^2}.
\end{equation}
If $\psi$ is a bounded analytic function on $\mathbb D$ and
$\varphi$ is defined on  $\Omega_1$  by $\varphi((\lambda, 1))
= \psi(\lambda)$ then $E_\varphi = \cal{ H}$ and $\varphi$ is
a pseudomultiplier.  However, $\varphi$ is not a multiplier of
$\cal H$ since $D_\varphi = \Omega_1 \not= \Omega$. \qed }
\end{example}

\begin{example}\label{direct-sum} A direct sum of pseudomultipliers.  \rm  In this example we conjoin  pseudomultipliers on two function spaces, $\phi$ of $\calh_1$ on a set $\Omega_1$  and $\chi$ of $\calh_2$ on  a set $\Omega_2$.     Define $\Omega$ to be the set-theoretic disjoint union $\Omega_1 \sqcup \Omega_2$ of $\O_1$ and $\O_2$.  Let $\calh$ be  the direct sum $\calh_1 \oplus \calh_2$.
  Clearly elements of $\calh$ are functions on $\O$ in a natural way.

  Consider the function $\gamma:D_\phi\sqcup D_\chi\to\C$ defined by
  \[
  \gamma (x) =\left\{ \begin{array}{lcl}     \phi(x) & \text{ if } & x\in D_\phi \\
  					\chi(x)  &\text{ if } & x\in D_\chi.
  					\end{array}\right.
  \]
where  $D_\phi$ and $D_\chi$ are the domains of $\phi$ and $\chi$, and so are sets of uniqueness for $\calh_1, \calh_2$ respectively. 
 For any $\vec{f}{g} \in\calh$, the function $\gamma\vec{f}{g}$ on $D_\phi\sqcup D_\chi$ is given by
  \[
\gamma\vec{f}{g} (x) =\left\{ \begin{array}{lcl}     \phi(x)f(x) & \text{ if } & x\in D_\phi \\
  					\chi(x)g(x)  &\text{ if } & x\in D_\chi.
  					\end{array}\right. 
  \] 
  
 Now a vector $\vec{f}{g} \in \calh$ satisfies $\gamma\vec{f}{g}\in\calh$
 if and only if $\phi f$ extends to an element $h_1$ of $\calh_1$ and $\chi g$ extends to an element $h_2$ of $\calh_2$.
   
 \[
 E_{\gamma} = \left\{\vec{f}{g} \in\calh:\text{ there exists } h=\vec{h_1}{h_2}\in\calh\text{ such that }\gamma\vec{f}{g}=h|(D_\phi\sqcup D_\chi)\right\}
 \]
 Hence 
 \be\label{Dir-Ephichi}
 E_{\gamma}= \left\{\vec{f}{g} \in\calh: f\in E_{\phi}, g\in E_{\chi}\right\} = E_\phi\oplus E_\chi.
 \ee
It is easy to see that if $\dim E_{\phi}^\perp= n$ and $\dim E_{\chi}^\perp= m$, then $E_{\gamma}$ has codimension $n +m$ in $\calh$. Therefore  $\gamma$ is a pseudomultiplier of $\calh$. \qed
\end{example}

In Examples \ref{1.4}, \ref{1.5} and \ref{1.6}, $\Omega$ has
been contrived to be either too large or too small.  In other
words, $D_\varphi$ is too small or too large.  In Example
\ref{1.4} we can envisage every $f\in \cal{ H}$ extending
naturally off $\Omega$ to the larger set $\mathbb D$.  The
function $\varphi$ apparently ``sees" this larger set as well,
and though it is bounded on $\Omega$ it fails to be a
multiplier as its values are unbounded near $\alpha$.  On the
other hand, in Example \ref{1.5}, $\varphi$ becomes a
multiplier if we block from  our vision the points of
$\mathbb D\backslash S$ (that is, replace $\cal{ H}$ with
$H^2|S$); alternatively, from its very definition,
$\varphi$ has an extension which  is a multiplier of $\cal H$.

At first glance Example \ref{1.6} resembles Example
\ref{1.5}: in both we can convert $\varphi$ into a multiplier
by restricting the domain of the Hilbert function space $\cal
H$.  However, there is a qualitative difference: in the former
$\varphi$ has an extension which is a multiplier, in the latter
in general it has not.  In Example \ref{1.5} $\varphi$ can
somehow see the points outside $D_\varphi$, in Example
\ref{1.6} it cannot.

It is clear that if we are to obtain any kind of classification of
pseudomultipliers then we need to address the issue of
modification of domains. 
There are different approaches to this issue.
Cowen and McCluer in \cite{cowen} and McCarthy and Shalit in \cite{McCShalit} have explored the notion that every Hilbert function space 
(or even Banach function space $X$) has a ``natural" domain of definition. In \cite{cowen}
the authors define a  Banach function space $X$ on a set $D$ to be {\em algebraically consistent}
if every continuous linear functional $x^*$ on $X$ that is partially multiplicative (in the sense that whenever $f,g$ and $fg$ all belong to $X$, the relation $x^*(fg)=x^*(f)x^*(g)$ holds)
is a point-evaluation functional at a point of $D$.
 The set of partially multiplicative linear functionals is analogous to the maximal ideal space of a
commutative Banach algebra.  However, in the context of a {\em particular } pseudomultiplier, there is another natural notion, which we regard as the most basic and fruitful concept in the general theory --  that of a pseudomultiplier {\em seeing} a vector.

We formalized the concept of a pseudomultiplier 
seeing a vector with value $c$ in Definition \ref{1.7} in the Overview of the paper (Section \ref{overview}).
Recall that a pseudomultiplier $\varphi$ of $\calh$ sees a vector $v\in\calh$
	with value $c$ if there exists a pseudomultiplier $\psi$ of
$\cal{ H}_v$ such that $\psi$ extends $\varphi$, $p \in D_\psi, E_\varphi = E_\psi
|\Omega$ and $\psi(p) =c$.

We first probe the meaning of ``seeing a vector" in the case of a multiplier $\phi$ of $\calh$. We can regard $\phi$ as a pseudomultiplier of $\calh$, with $D_\phi=\Omega$, $E_\phi=\calh$, with singular space $\{0\}$ and $X_\phi$ the corresponding operator $M_\phi:\calh\to\calh$ 
of multiplication by $\phi$. 

\begin{proposition}\label{multipliers} 
Let $\phi$ be a multiplier of a Hilbert function space $\calh$ on a set $\Omega$. 
 The vectors which are visible to $\phi$ are precisely the eigenvectors of $M_\phi^*$. 
\end{proposition}
\begin{proof}
 Suppose that $\phi$ sees $v$ with value $c$.  That means that there is a pseudomultiplier $\psi$ of $\calh_v$ defined on $\Omega \cup\{p\}$ such that $\psi$ extends $\phi$ and $\psi(p)=c$.  Now, for any $f\in\calh$, as functions on $\Omega \cup\{p\}$, $\psi f_v$ equals $\phi f$ on $\Omega$ and equals $c\ip{f}{v}$ at $p$.  Hence 
\[
c\ip{f}{v}=\ip{\phi f}{v} = \ip{f}{M_\phi^*v}.
\]
Thus $M_\phi^* v= \bar c v$. The steps are all reversible, so $\phi$ sees $v$ with value $c$ if and only if $M_\phi^* v= \bar c v$.  
\end{proof}
In particular, for $\lam\in\Omega$, a multiplier $\phi$ sees the kernel $k_\lam$ with value $\phi(\lam)$, a fact that we generalized to pseudomultipliers in Example \ref{1.8}.	

In the next example $\calh\calh$ denotes the class of functions on $\Omega$ which can be expressed as a product $fg$ where $f,g\in\calh$.

\begin{example}\label{1.9}
{\rm The result \cite[Theorem 2.1]{AY1} implies that if $\cal{
H}$ is a nonsingular Hilbert function space (in the sense that the vectors $\{k_\lam:\lam\in\O\}$ are linearly independent in $\calh$) and $\varphi$ is a
pseudomultiplier of $\cal{ H}$ with the property that
$D_\varphi$ is a set of uniqueness for $\cal{ H}\cal{ H}$, then
$\varphi$ sees $k_\lambda$ for all $\lambda \in \Omega
\backslash F$ for some set $F \subset \Omega$ containing at
most $\dim(\cal{ H} \ominus E_\varphi)$ points.}
\end{example}

\begin{remark}\label{1.10}
{\rm Definition \ref{1.7} allows us to formalize the heuristic
remarks made following Examples \ref{1.4}, \ref{1.5}, and
\ref{1.6}.  Thus, in Example \ref{1.4}, $\varphi$ sees
$k_\lambda$ with value $(\lambda-\alpha)^{-1}$ for all
$\lambda \in \mathbb D\backslash \{\alpha\}$ while $\varphi$
does not see $k_\alpha$.  On the other hand, in Example
\ref{1.5} while strictly speaking $k_\lambda(z) =
(1-\overline \lambda z)^{-1}$ is the reproducing
kernel for $\cal{ H}$ only when $\lambda \in 
\Omega$, $\varphi$ sees
$k_\lambda$ with  value
$\psi(\lambda)$ for all $\lambda \in \mathbb D$.  In Example
\ref{1.6}, in contrast to Example \ref{1.5}, $\varphi$ does  not
necessarily always see the kernel function.  Thus, if $\omega
= (\lambda, 1) \in \Omega_1$, then $\varphi$ sees
$k_\omega$ with value $\psi(\lambda)$ but if $\omega =
(\lambda, 2) \in \Omega_2$, then the reader can easily check
that $\varphi$ sees $k_\omega$ with value $c$ if and only if
$L^* k_\lambda$ is an eigenvector for $M^{*}_\varphi$ with
corresponding eigenvalue $\overline c$, a situation that
need not occur.}
\end{remark}

\begin{example}\label{1.11}
{\rm In Example \ref{1.4}, $\varphi$ does not see $k_\alpha$. 
If we were to drop the requirement in Definition \ref{1.7} that
$E_\varphi=E_\psi|\Omega$ this would no longer be true: 
with the resulting definition, $\phi$ {\em would} see $k_\alpha$.
In accordance with Definition \ref{1.7}, to determine whether $\phi$ sees the vector $v=k_\alpha$ in $\calh=H^2|\Omega$ with value $c\in\C$
we must consider whether there is a pseudomultiplier $\psi$ of the Hilbert function space $\calh_v$ on $\Omega \cup \{p\}$ such that $\psi$ extends $\phi$, $E_\phi=E_\psi|\Omega$ and $\psi(p)=c$.
Such a function $\psi$ on $\Omega \cup \{p\}$ is necessarily given by
$$
\psi(\lambda) = \left\{
\begin{array}{cl}
\varphi(\lambda) &\qquad \mbox{ for } \lambda \in \Omega\\ \\
c&\qquad \mbox{ if } \lambda = p.
\end{array} \right.
$$
Now $E_\psi$ comprises those functions $f_v\in \calh_v$ on $\Omega \cup \{p\}$ such that $\psi f_v \in \calh_v$.
By construction of $\calh_v$, $\psi f_v \in \calh_v$ if and only if there exists $g\in H^2$ such that, as functions on 
$\Omega \cup \{p\},\ \psi f_v = (g|\Omega)_v$, which is to say that $\phi f=g$ on $\Omega$ and, corresponding to values at $p$, $c\ip{f}{v} = \ip{g|\Omega}{v}.$  Since $v=k_\alpha$, the reproducing kernel of $\alpha$ in $H^2$, this pair of equations becomes
\begin{eqnarray}
	\frac{f(\lambda)}{\lambda-\alpha}&=&g(\lambda) \mbox{ for all } \lambda\in\Omega \label{main}\\
cf^\sim(\alpha)=c\ip{f}{v} = \ip{g|\Omega}{v}&=&\ip{g}{k_\alpha}=g(\alpha) \label{atp}
\end{eqnarray}
(recall that $f^\sim$ is the function in $H^2$ whose restriction to $\Omega$ is $f$). 
From equation \eqref{main} we can infer that $f^\sim(\alpha)=0$ and $g(\alpha)=(f^\sim)'(\alpha)$ which, in conjunction with equation \eqref{atp} further implies that $g(\alpha)=0$, and therefore $f^\sim(\alpha)= (f^\sim)'(\alpha)=0$.   Thus 
if $f_v\in E_\psi$ then $f^\sim(\alpha)= (f^\sim)'(\alpha)=0$.  Conversely, if $f^\sim(\alpha)= (f^\sim)'(\alpha)=0$ then equations \eqref{main} and \eqref{atp} hold and so $E_\psi=\{f_v\in\calh_v: f^\sim(\alpha)=(f^\sim)'(\alpha)=0\}$ and so $E_\psi|\Omega$ is a proper subset of $E_\phi$.
The present example explains why
we require that $E_\varphi=E_\psi|\Omega$ in Definition \ref{1.7}. \qed}
\end{example}

\begin{lemma}\label{Epsi=Ephi}
Let $\phi$ be a pseudomultiplier of a Hilbert function space $(\calh,\O)$, let $v\in\calh$ and let $p$ be a point not in $\O$. For any pseudomultiplier $\psi$ of $\calh_v$ such that $D_\psi \supset D_\phi\cup\{p\}$ and $\psi|D_\phi=\phi$, the inclusion $E_\psi|\O \subseteq E_\phi$ holds.
\end{lemma}
\begin{proof}
Consider an arbitrary element of $E_\psi|\O$: it can be expressed in the form $f_v|\O$ for some $f\in\calh$.  By definition $f_v\in E_\psi$ if and only if  there exists $g\in\calh$ such that $g_v=\psi f_v$ on $D_\psi$.  Evaluating both sides of this equation at any point $z\in D_\phi$, we find that $g(z)=\phi(z)f(z)$, and so $f\in  E_\phi$.  Thus $E_\psi|\O \subseteq E_\phi$, as claimed.
\end{proof}

\section{Ambiguous points}\label{ambiguous}

In this section we shall repeatedly use a notation introduced in the overview of the paper.
Let $\cal{ H}$ be a Hilbert function space on $\Omega$.
For a pseudomultiplier $\phi$ of $\calh$, $c\in \mathbb C$ and $\alpha\in D_\phi$, define the pseudomultiplier
$\varphi_{c, \alpha}$ on $D_\varphi$ by
\begin{equation}\label{modify-2}
\varphi_{c, \alpha} (\lambda) = \left\{ \begin{array}{cl}
\varphi(\lambda) & \ \  {\rm if}\quad \lambda \in
D_\varphi\setminus\{\alpha\}\\ \\
c& \ \ {\rm if}\quad \lambda = \alpha.\end{array}\right.
\end{equation}

 Recall Definition \ref{2.1}: a point $\alpha \in
D_\varphi$  is an {\it ambiguous point for } a pseudomultiplier
$\varphi$ of $\calh$ if there exists $c\in \mathbb C$ with the properties
that $c\not= \varphi(\alpha), E_\varphi = E_{\varphi_{c,\alpha}}$, and
$X_{\varphi_{c,\alpha}} = X_\varphi$. 
An ambiguous point is a point at which the value of a
pseudomultiplier $\varphi$ is not determined by the
associated operator $X_\varphi$. 

In Definition \ref{d2.12} below we shall extend the notion of an ambiguous point of $\phi$ to points $\alpha\in\Omega$ not assumed to be in $D_\phi$. We then prove in Theorem \ref{2.14} that 
$\alpha \in \Omega$ is ambiguous for $\varphi$ if and
only if, for every $c\in \mathbb C,\ \varphi$ sees
$k_\alpha$ with value $c$.
This linkage between ambiguous points and the notion of seeing a vector is useful for the
organization of the wide variety of examples of
pseudomultipliers and for the analysis of their singularities.

\begin{lemma}\label{2.2}
Let $\psi$ be a  multiplier of a Hilbert function space $\cal{ H}$ on a set $\Omega$.  Suppose that $\alpha\in \Omega$ has the property
that there exists a sequence $\{\lambda_n\} \subseteq
\Omega$ with $\lambda_n \not= \alpha$ and $k_{\lambda_n}
\to k_\alpha$.   Let
$\varphi=\psi_{z,\alpha}$ where $z\in \mathbb C\setminus\{ \psi(\alpha)\}$.  Then $\phi$ is a pseudomultiplier of $\calh$ and $\alpha$ is an
ambiguous point for $\varphi$. 
\end{lemma}
\begin{proof}

 This follows immediately
from Definition \ref{2.1} and the following simple observation:
\begin{equation}\label{5}
{\rm if}\ c\not= \psi(\alpha), \ {\rm then}\ E_{\psi_{c,\alpha}}
= k^\perp_\alpha \ {\rm and}\ X_{\psi_{c, \alpha}} = X_\psi
|k^\perp_\alpha.
\end{equation}
To prove equation (\ref{5}), fix $c\not= \psi(\alpha)$ and first assume
that $f\in k^\perp_\alpha$.  Evidently,
$\psi_{c,\alpha}(\lambda) f(\lambda) = \psi(\lambda)
f(\lambda)$ for all $\lambda \in \Omega\backslash
\{\alpha\}$.  But since $f \in k^\perp_\alpha, f(\alpha) =0$, and
so $\psi_{c,\alpha}(\alpha) f(\alpha) = \psi(\alpha) f(\alpha)$. 
Thus, $f\in k^\perp_\alpha$ implies that $\psi_{c,\alpha}f =
\psi f \in \cal{ H}$ and we have shown that $k^\perp_\alpha
\subseteq E_\psi$ and $X_{\psi_{c,\alpha}}|k^\perp_\alpha =
X_\psi| k^\perp_\alpha$.  There remains to show that
$E_{\psi_{c,\alpha}} \subseteq k^\perp_\alpha$.  Accordingly,
fix $f\in E_{\psi_{c,\alpha}}$.  Thus, $\psi_{c, \alpha} f \in \cal{
H}$.  But  $\psi$ is a multiplier and so we see that
$g=\psi_{c,\alpha} f - \psi f \in \cal{ H}$.  Furthermore,
$g(\lambda) =0$ if $\lambda \not= \alpha$.   Recalling our
technical hypothesis we see  that 
$$
g(\alpha) = \ip{g}{ k_\alpha} =
\lim\limits_{n\to \infty} \ip{g}
{k_{\lambda_n}}=\lim\limits_{n\to \infty} g(\lambda_n)=0.
$$  
Hence $0=g(\alpha) = (c-\psi(\alpha)) f(\alpha).$  
But
$c\not= \psi(\alpha)$, so $f(\alpha) =0$ i.e. $f\in
k^\perp_\alpha$.  This concludes the proof of equation (\ref{5}).
\end{proof}

\begin{example}\label{2.3}
	{\rm Let $\cal{ H}$ be the space $W^{1,2}[0,1]$ of absolutely continuous
		functions on $[0, 1]$ whose weak derivatives are in $L^2$ and
		with inner product given by
		$$
		\ip{f}{g} = \int^1_0 f(t) \overline {g(t)} dt + \int^1_0
		f^\prime (t) \overline{ g^\prime(t)} dt.
		$$
		We showed in Example \ref{intro2.3}	in the introduction that $\chi(t)=\sqrt{t}$ is a pseudomultiplier of
	$\cal{ H}$ with $E_\chi = k^\perp_0$.	
			 We claim that $0$
		is an ambiguous point for $\chi$.  Indeed,
		\begin{equation}\label{6}
		{\rm if}\ c\in \mathbb C\ {\rm is\ arbitrary,\ then}\
		E_{\chi_{c,0}} = k^\perp_0 \ {\rm and}\ 
		X_{\chi_{c,0}}=X_\chi.
		\end{equation}
		The proof of the implication (\ref{6}) is very simple.  On the one hand, if
		$f\in k^\perp_0$, then $\chi_{c, 0} f=\chi f$ so that
		$f\in E_{\chi_{c,0}}$ and $X_{\chi_{c,0}} f= X_\chi
		f$.  On the other hand, if $f\in E_{\chi_{c,0}}$, then
		$\sqrt{t}f\in\cal{ H}$ which forces $f(0)=0$.
		\qed}
\end{example}

Observe that
there is a qualitative difference  between the examples of ambiguous points implied by Lemma 
\ref{2.2} and the ambiguous point in Example \ref{2.3}.  In Lemma \ref{2.2}, $\varphi$ can be
redefined  at the ambiguous point $\alpha$ in such a way
that $\alpha$ is no longer an ambiguous point.
But in Example \ref{2.3}, $0$ remains an ambiguous point for
$\chi$, regardless of how $\chi$ is redefined at $0$
(this follows from the implication (\ref{6})).  We formalize these two 
cases in the following definition.

\begin{definition}\label{2.4}
\rm Let $\varphi$ be a pseudomultiplier of $\calh$ and let  $\alpha \in D_\varphi$ be
an ambiguous point for $\varphi$.  We say $\alpha$ is a
{\em removable ambiguous point} for $\varphi$ if there exists
$c\in \mathbb C$ such that $\alpha$ is not an ambiguous
point for $\varphi_{c,\alpha}$. 
 We say $\alpha$ is an {\em essential
ambiguous point} for $\varphi$ if $\alpha$ is an ambiguous point but
not a removable ambiguous point for $\varphi$.  
\end{definition}

Thus, in Lemma \ref{2.2} $\alpha$ is a removable ambiguity while
in Example \ref{2.3}, $\alpha$ is an essential ambiguity. 
Before continuing our investigation of ambiguous points
we state and prove some useful technical lemmas.

\begin{lemma}\label{2.5}
If $\varphi$ is a pseudomultiplier, $\alpha \in D_\varphi$
and $E_\varphi \subset k^\perp_\alpha$ then, for any $c\in
\mathbb C, E_\varphi = E_{\varphi_{c,\alpha}} \cap
k^\perp_\alpha$ and
$X_\varphi = X_{\varphi_{c,\alpha}}|E_\varphi$.
\end{lemma}
\begin{proof} If $f\in E_\varphi$ then $f(\alpha)=0$
and there exists $h\in \cal{ H}$ such that $h=f\varphi$ on
$D_\varphi$.  Hence $h=\varphi f = \varphi_{c,\alpha}f$ on
$D_\varphi$, and so $f\in E_{\varphi_{c,\alpha}}$.  Thus
$E_\varphi \subseteq E_{\varphi_{c,\alpha}} \cap
k^\perp_\alpha$ and $X_\varphi =
X_{\varphi_{c,\alpha}}|E_\varphi$.  Conversely, assume that
$f\in E_{\varphi_{c,\alpha}} \cap k^\perp_\alpha$.  Then
$f(\alpha) =0$ and there exists $h\in \cal{ H}$ such that
$\varphi_{c,\alpha} f=h$ on $D_\varphi$.  Hence $\varphi f =
\varphi_{c,\alpha}f =h$ on $D_\varphi$ and  so $f\in
E_\varphi$.  Thus $E_{\varphi_{c,\alpha}}\cap
k^\perp_\alpha \subseteq E_\varphi$. \end{proof}

\begin{lemma}\label{2.6}
If $\varphi$ is a pseudomultiplier, $\alpha \in D_\varphi$
and $E_\varphi \subseteq k^\perp_\alpha$, then there can
exist at most one complex number $c$ such that
$E_{\varphi_{c,\alpha}}$ is not a subset of $k^\perp_\alpha$.
\end{lemma}

\begin{proof} Let $\varphi$ and $\alpha$ satisfy the
hypotheses of the lemma and assume that $c_1, c_2 \in
\mathbb C$ with $E_{\varphi_{c_1,\alpha}} \not\subseteq
k^\perp_\alpha$ and $E_{\varphi_{c_2,\alpha}}
\not\subseteq k^\perp_\alpha$.  We need to prove that
$c_1=c_2$.  Since
$E_{\varphi_{c_1,\alpha}} \not\subseteq k^\perp_\alpha$,
there exists $f_1 \in E_{\varphi_{c_1,\alpha}}$ with
$f_1(\alpha)=1$.  Similarly, pick $f_2\in
E_{\varphi_{c_2,\alpha}}$ with $f_2(\alpha)=1$.  Since 
$f_1\in E_{\varphi_{c_1,\alpha}}$, there exists $h_1\in \cal{
H}$ such that $\varphi_{c_1,\alpha} f_1=h_1$ on $D_\varphi$. 
Similarly, there exists $h_2\in \cal{ H}$ such that
$\varphi_{c_2,\alpha}f_2 =h_2$ on $D_\varphi$.  Note that
since $f_1(\alpha) = f_2(\alpha) =1,
\varphi_{c_1,\alpha}(\alpha)=c_1$ and
$\varphi_{c_2,\alpha}(\alpha)=c_2$, we have
$h_1(\alpha)=c_1$ and $h_2(\alpha)=c_2$.  Now define $f\in
\cal{ H}$ by the formula
$$
f=(c_2-\varphi(\alpha)) f_1 - (c_1-\varphi(\alpha)) f_2,
$$
and define $h\in \cal{ H}$ by the formula
$$h=(c_2-\varphi(\alpha)) h_1 - (c_1-\varphi(\alpha))
h_2.$$
We claim that $\varphi f=h$ on $D_\varphi$.  To see this,
first fix $\lambda \in D_\varphi\backslash\{\alpha\}$ and
note that
\begin{eqnarray*}
\varphi(\lambda) f(\lambda) &=& (c_2-\varphi(\alpha))
\varphi(\lambda) f_1(\lambda) - (c_1-\varphi(\alpha))
\varphi(\lambda) f_2(\lambda)\\ 
&=&  (c_2-\varphi(\alpha))
\varphi_{c_1,\alpha}(\lambda) f_1(\lambda) -
(c_1-\varphi(\alpha))
\varphi_{c_2,\alpha}((\lambda) f_2(\lambda)\\ 
&=& (c_2-\varphi(\alpha))
h_1(\lambda)  - (c_1-\varphi(\alpha))
h_2(\lambda)\\  
&=& h(\lambda).
\end{eqnarray*}
On the other hand,
\begin{eqnarray*}
\varphi(\alpha) f(\alpha) &=&
\varphi(\alpha)\Big((c_2-\varphi(\alpha)
 f_1(\alpha) - (c_1-\varphi(\alpha))
 f_2(\alpha)\Big)\\  
&=&  \varphi(\alpha)\Big((c_2-\varphi(\alpha))
 - (c_1-\varphi(\alpha))\Big)\\  
&=&  \varphi(\alpha)(c_2- c_1)\\  
&=&
\Big(c_2-\varphi(\alpha)\Big)c_1 -
\Big(c_1-\varphi(\alpha)\Big)c_2\\
&=&  \Big(c_2-\varphi(\alpha)\Big)h_1(\alpha)
 - \Big(c_1-\varphi(\alpha)\Big)h_2(\alpha)\\  
&=& h(\alpha).
\end{eqnarray*}
Thus, in fact, $\varphi f=h$ on $D_\varphi$.  But this means
that $f\in E_\varphi$.  Since, by assumption, $E_\varphi
\subseteq k^\perp_\alpha, f \in k^\perp_\alpha$, i.e.
$f(\alpha)=0$.  Thus
\begin{eqnarray*}
0 &=&
f(\alpha)\\ 
&=&  \Big(c_2-\varphi(\alpha)\Big)f_1(\alpha)
 - \Big(c_1-\varphi(\alpha)\Big)\\ 
&=&
\Big(c_2-\varphi(\alpha)\Big)  -
\Big(c_1-\varphi(\alpha)\Big) \\
&=& c_2-\ c_1 ,
\end{eqnarray*}
so that $c_1=c_2$ as was to be shown. 
\end{proof}
\vspace{1mm}

Lemmas \ref{2.5} and \ref{2.6} combine to give a
simple criterion for a point $\alpha$ to be an ambiguous
point for a pseudomultiplier.

\begin{proposition}\label{2.7}
Let $\varphi$ be a pseudomultiplier and let $\alpha \in
D_\varphi$.  Then $\alpha$ is an ambiguous point for
$\varphi$ if and only if $E_\varphi \subseteq
k^\perp_\alpha$.
\end{proposition}

\begin{proof}  First assume that $E_\varphi\subseteq
k^\perp_\alpha$.  By Lemma \ref{2.6}, there exists  $c\not=
\varphi(\alpha)$ such that $E_{\varphi_{c,\alpha}} \subseteq
k^\perp_\alpha$.  But then Lemma \ref{2.5}  implies that
$E_\varphi=E_{\varphi_{c,\alpha}}$ and $X_\varphi =
X_{\varphi_{c,\alpha}}$.  Thus  $\alpha$ is an ambiguous
point for $\varphi$.

Conversely, assume that $\alpha$ is an ambiguous point for
$\varphi$.  Thus  there exists $c\not= \varphi(\alpha)$ such
that $E_\varphi = E_{\varphi_{c,\alpha}}$ and $X_\varphi =
X_{\varphi_{c,\alpha}}$.  We need to show that $E_\varphi
\subseteq k^\perp_\alpha$.  Accordingly, fix $f\in
E_\varphi$.  To see that  $f\in k^\perp_\alpha$ we argue by
contradiction. Assume that $f(\alpha) \not= 0$.  Now
since $f\in E_\varphi=E_{\varphi_{c,\alpha}}$ there exist
$h_1, h_2 \in \cal{ H}$ such that $\varphi f=h_1$ and
$\varphi_{c,\alpha} f=h_2$ on $D_\varphi$.  Hence, if we set
$h=h_1-h_2$ we have that $h\in \cal{ H}$ and 
\begin{eqnarray*}
h(\lambda)&=& h_1(\lambda) - h_2(\lambda)\\ \\
&=& \varphi(\lambda) f(\lambda) -
\varphi_{c,\alpha}(\lambda) f(\lambda)\\ \\
&=& \left\{ \begin{array}{cl}
0& {\rm if}\quad \lambda \in
D_\varphi\backslash\{\alpha\}\\ \\
\Big(\varphi(\alpha) - c\Big) f(\alpha)\qquad & {\rm
if}\quad
\lambda=\alpha.
\end{array}\right.
\end{eqnarray*}
Note that since $\varphi(\alpha) \not= c$ and $f(\alpha)
\not= 0$, this formula implies that $h\not= 0$.  Also, the
formula implies that $\varphi h=\varphi(\alpha) h$ on
$D_\varphi$ so that $h\in E_\varphi$ and $X_\varphi
h=\varphi(\alpha)h$.  But the formula similarly implies
that $X_{\varphi_{c,\alpha}} h=ch$.  Thus, since
$\varphi(\alpha) \not= c$ and $h\not= 0,
X_{\varphi_{c,\alpha}} h\not= X_\varphi h$, contradicting
the fact that $X_\varphi = X_{\varphi_{c,\alpha}}$.  This
completes the proof of Proposition \ref{2.7}. \end{proof}
\vspace{1mm}

Thus $\alpha$ is an ambiguous point for $\varphi$ if and
only if $E_\varphi \subseteq k^\perp_\alpha$.  A
consequence of Lemmas \ref{2.5} and \ref{2.6} is that, for
ambiguous points $\alpha$, there is a rather rigid
dichotomy for the behavior of the spaces
$E_{\varphi_{c,\alpha}}$ as $c$ varies.

\begin{lemma}\label{2.8}
Let $\varphi$ be a pseudomultiplier and let $\alpha \in D_\phi$. If $\alpha$ is an
ambiguous point for $\varphi$ then exactly one of the following two
statements holds:
\begin{enumerate}
\item[\rm (i)] \begin{equation}
\label{2.11}
E_{\varphi_{c,\alpha}} \subseteq
k^\perp_\alpha \text{ for all }c\in \mathbb C;
\end{equation}

\item[\rm (ii)]
  There exists $\gamma \in
\mathbb C$ such that
\begin{equation}
\label{2.12} 
E_{\varphi_{\gamma, \alpha}}
\not\subset k^\perp_\alpha \text{ but } E_{\varphi_{c,\alpha}}
\subset k^\perp_\alpha
\end{equation}
for all $c\in
\mathbb C\smallsetminus\{\gamma\}$.
\end{enumerate}

\noindent Furthermore, 

if equation \eqref{2.11} holds then
$E_{\varphi_{c,\alpha}} = E_\varphi$ and
$X_{\varphi_{c,\alpha}} = X_\varphi$ for all
$c\in \mathbb C$; 

if equation \eqref{2.12} holds then
$E_{\varphi_{c,\alpha}} = E_\varphi$ and
$X_{\varphi_{c,\alpha}}=X_\varphi$ for all
$c\in \mathbb C\smallsetminus \{\gamma\}$ while
$E_\varphi$ is a subspace of $E_{\varphi_{\gamma, \alpha}}$
of codimension 1 and $X_\varphi=
X_{\varphi_{\gamma,\alpha}}|E_\varphi$.
\end{lemma}

\begin{proof}  That either of equations  (\ref{2.11}) or (\ref{2.12}) occurs is
simply a restatement of Lemma \ref{2.6}.  If (\ref{2.11}) occurs,
then Lemma \ref{2.5} implies that $E_{\varphi_{c,\alpha}} =
E_\varphi$ and $X_{\varphi_{c,\alpha}} = X_\varphi$ for all
$c\in \mathbb C$.  If (\ref{2.12}) occurs then again Lemma
\ref{2.5} implies that $E_{\varphi_{c,\alpha}}=E_\varphi$
and $X_{\varphi_{c,\alpha}} = X_\varphi$ for all $c\in
\mathbb C\backslash\{\gamma\}$.  Finally, suppose that
(\ref{2.12}) occurs, so that $E_{\varphi_{\gamma,\alpha}}
\not\subseteq k^\perp_\alpha$.  Since by Lemma \ref{2.5}, 
$E_\varphi = E_{\varphi_{\gamma,\alpha}}\cap k^\perp_\alpha$,  
necessarily $E_\varphi\subseteq
E_{\varphi_{\gamma,\alpha}}$ and $E_\varphi$ has
codimension 1 in $E_{\varphi_{\gamma, \alpha}}$.  Also, in
this case, Lemma \ref{2.5} implies that $X_\varphi=
X_{\varphi_{\gamma,\alpha}}|E_\varphi$.
\end{proof}

Lemma \ref{2.8} has a number of corollaries about the
nature of ambiguous points.  The key observation is that in
Lemma
\ref{2.8}, the inclusion (\ref{2.11}) occurs when $\alpha$ is essential, and the relation
(\ref{2.12}) occurs when $\alpha$ is removable.

\begin{proposition}\label{2.9}
Let $\varphi$ be a pseudomultiplier and assume that
$\alpha \in D_\varphi$ is an ambiguous point for $\varphi$.   Then $\alpha$
is an essential ambiguous point for $\phi$ if
and only if  $E_{\varphi_{c,\alpha}} \subset k^\perp_\alpha$
for all $c\in \mathbb C$.
\end{proposition}

\begin{proof}  By definition, $\alpha $ is essential
if and only if $\alpha$  is ambiguous for $\varphi_{c,\alpha}$
for all $c\in \mathbb C$.  Hence the proposition follows
immediately from Proposition \ref{2.7}. \end{proof}

Notice that in Definition \ref{2.1} $\alpha$ was said to be
ambiguous if $\varphi$ could be redefined at $\alpha$
without the underlying operator $X_\varphi$ changing.  In
fact, if $\alpha$ is ambiguous, $\varphi$ can be {\it
arbitrarily} redefined at $\alpha$ if $\alpha$ is
essential and arbitrarily redefined at $\alpha$ with
one exception if $\alpha$ is removable.

\begin{proposition}\label{2.10}
 Let $\varphi$ be a pseudomultiplier and let
$\alpha \in D_\varphi$ be an ambiguous point for $\varphi$. 
\begin{enumerate}
\item [\rm{(i)}] If $\alpha$ is essential, then $E_{\varphi_{c,\alpha}} =
E_\varphi$ and $X_{\varphi_{c,\alpha}} =X_\varphi$ for all
$c\in \mathbb C$.  

\item [\rm{(ii)}] If $\alpha$ is removable then there exists
a unique $\gamma \in \C$ such that $E_{\varphi_{c,\alpha}} = E_\varphi$
 if and only if $c\in \mathbb C\backslash\{\gamma\}$.  Moreover, this number $\gam$ has the property that  $X_{\varphi_{c,\alpha}} = X_\varphi$ if and only if $c\in \mathbb C\backslash\{\gamma\}$.
\end{enumerate}
\end{proposition}

\begin{proof} 

(i) Suppose that $\al$ is an essential ambiguous point for $\phi$ and $c\in\C$.  By Proposition \ref{2.7} $E_\phi \subseteq k_\alpha^\perp$.   Thus, for any function $f\in E_\phi$, $f(\al)=0$ and so $\varphi_{c,\alpha}f= \phi f$ on $D_\phi$, from which it follows that $E_{\varphi_{c,\alpha}} =
E_\varphi$ and $X_{\varphi_{c,\alpha}} =X_\varphi$.

(ii)  Suppose that $\al$ is a removable ambiguous point for $\phi$.  Then $\al$ is not an essential ambiguous point for $\phi$, and so, by Proposition \ref {2.9}, there exists $\gam\in\C$ such that $E_{\phi_{\gam,\al}} \not \subset k_\al^\perp$.  Thus, in the dichotomy described in Lemma \ref{2.8} it is alternative (ii) that holds, and so $E_{\phi_{c,\al}} \subset k_\al^\perp$ for all $c\in\C\setminus\{\gam\}$, and by the ``Furthermore" of Lemma \ref{2.8}, $E_{\varphi_{c,\alpha}} =
E_\varphi$ and $X_{\varphi_{c,\alpha}} =X_\varphi$ for all $c\in\C\setminus\{\gam\}$.
It follows that $\gam$ is unique.
\end{proof}
\vspace{1mm}

We conclude our remarks concerning ambiguous  points
$\alpha \in D_\varphi$ by noting the following fact.

\begin{proposition}\label{p2.11}
Let $\varphi$ be a pseudomultiplier on $\calh$ and let
$\alpha \in D_\varphi$. If $\alpha$ is a removable ambiguous point for
$\varphi$ then 
\begin{enumerate}
\item [\rm{(i)}]
 there exists a unique $c\in \mathbb C$ such that
$\alpha$ is not an ambiguous point for $\varphi_{c,\alpha}$,
\item [\rm{(ii)}] there exists a unique $d\in \mathbb C$  such that 
$E_{\varphi} \subset E_{\varphi_{d,\alpha}}$. 
 Finally,
$c = d$ and $X_\varphi = X_{\varphi_{c,\alpha}}|E_\varphi$.
\end{enumerate}
\end{proposition}

\begin{proof} 

 Suppose that $\alpha$ is a removable ambiguous point for $\varphi$.
 
(i) By Definition \ref{2.4}, there exists $c\in\C$ such that $\al$ is not an ambiguous point for $\phi_{c,\al}$, so that, to prove statement (i), we must only show that this $c$ is unique.  But by Proposition \ref{2.10}(ii), this value (denoted by $\gam$) is unique.  So statement (i) holds.

(ii) Since $\alpha$ is a removable ambiguous point for $\phi$, $\al$ is not an essential ambiguous point for $\phi$, and so, by Proposition \ref {2.9}, we may choose $d\in\C$ such that $E_{\phi_{d,\al}} \not \subset k_\al^\perp$. Thus, in the dichotomy described in Lemma \ref{2.8} it is alternative (ii) that holds, with $\gam=d$, and so $E_{\phi_{z,\al}} \subset k_\al^\perp$ for all $z\in\C\setminus\{d\}$, 
$E_\varphi$ is a subspace of $E_{\varphi_{d, \alpha}}$
of codimension 1 and $X_\varphi=
X_{\varphi_{d,\alpha}}|E_\varphi$. We claim that this $d$ is unique.  But by Proposition \ref{2.10}(ii), this value (denoted by $\gam$) is unique. It is clear that $\al$ is not an ambiguous point for $\phi_{d,\al}$, and so, by the uniqueness statement in part (i),  $c=d$. Thus statement (ii) of Proposition \ref{p2.11}   holds.
\end{proof}

The concept of an ambiguous point can be extended
naturally to points not assumed to be in the domain of
$\varphi$. 
Recall Example \ref{1.8}, if  $\varphi$ is a pseudomultiplier of a Hilbert function space $\cal{ H}$ 
on $\Omega$ and if $\lambda \in D_\varphi$, then $\varphi$ sees $k_\lambda$ with value
		$\varphi(\lambda)$.
Let us combine the definition of an ambiguous point $\alpha \in D_\phi$ with the notion of $\phi$ seeing $k_\alpha$ to yield a concept of an ambiguous point $\alpha\in\O\setminus D_\phi$.

\begin{definition}\label{d2.12}
Let $\varphi$ be a  pseudomultiplier and let $\alpha\in
\Omega\backslash D_\varphi$.  
We say $\alpha$ is an {\it
ambiguous point for } $\varphi$ if there exists $c\in
\mathbb C$ such that $\varphi$ sees $k_\alpha$ with value
$c$ and $\alpha$ is an ambiguous point for the
pseudomultiplier $\psi_{c,\alpha}$ defined on $D_\varphi \cup
\{\alpha\}$ by the formula
\begin{equation}\label{7}
\psi_{c,\alpha}(\lambda) = \left\{ \begin{array}{cl}
\varphi(\lambda) \qquad & \mbox{ if }\lambda\in D_\varphi\\ \\
c\qquad & \mbox{ if } \lambda = \alpha \end{array}\right. .
\end{equation}
\end{definition}

\begin{proposition}\label{ambessential} Let $\varphi$ be a  pseudomultiplier of $\calh$ on $\O$ and let $\alpha\in \Omega\setminus D_\varphi$.  If $\alpha$ is an ambiguous point of $\phi$ then $\alpha$ is an  essential ambiguous point for $\psi_{c,\alpha}$ as defined in equation \eqref{7}. 
\end{proposition}
\begin{proof}
By the definition of ambiguous point of $\phi$, there is $c\in\C$ such that $\varphi$ sees $k_\alpha$ with
value $c$ and  $E_\varphi = E_{\psi_{c,\alpha}}$. 
Also by definition, $\alpha$ is
ambiguous for $\psi_{c,\alpha},$ and so, by Proposition \ref{2.7}, $ E_{\psi_{c,\alpha}} \subseteq
k^\perp_\alpha$, and therefore $E_\varphi \subseteq
k^\perp_\alpha$. 
We claim that, for every $d\in \C$, $\alpha$ is an ambiguous point for $\phi_{d,\alpha}$.
For all $d\in \mathbb C,
\psi_{d,\alpha}$ extends $\varphi$, and so $ E_{\psi_{d,\alpha}} \subseteq E_\varphi $.  Thus, for all $d\in \mathbb C, E_{\psi_{d, \alpha}} \subseteq k_\alpha^\perp$.  
Therefore, for all $d\in \mathbb
C, E_{(\psi_{c,\alpha})_{d,\alpha}} =
E_{\psi_{d,\alpha}}\subseteq k^\perp_\alpha$. 
 By
Proposition \ref{2.9} $\alpha$ is essential for
$\psi_{c,\alpha}$. 
\end{proof}

In light of Proposition \ref{ambessential}, 
the following analog to Proposition \ref{2.10} is not a surprise.

\begin{proposition}\label{2.13}
Let $\varphi$ be a pseudomultiplier and let $\alpha \in
\Omega \smallsetminus D_\varphi$.  For $c\in \mathbb C$
let $\psi_{c,\alpha}$ be defined by equation  \eqref{7}.  Then $\alpha$
is an ambiguous point for $\varphi$ if and only if, for all
$c\in \mathbb C, E_{\psi_{c,\alpha}}=E_\varphi$ and
$X_{\psi_{c,\alpha}}=X_\varphi$.
\end{proposition}

We conclude the section with a result which gives a
succinct characterization of ambiguity applicable to both
cases, $\alpha \in D_\varphi$ and $\alpha \in \Omega
\smallsetminus D_\varphi$.

\begin{theorem}\label{2.14}
Let $\varphi$ be a pseudomultiplier and let $\alpha \in
\Omega$.  Then $\alpha$ is ambiguous for $\varphi$ if and
only if, for every $c\in \mathbb C,\ \varphi$ sees
$k_\alpha$ with value $c$.
\end{theorem}

\begin{proof}  First suppose that $\varphi$ sees
$k_\alpha$ with arbitrary value.  Thus, for each $c$ there
exists a pseudomultiplier $\omega_c$ of $\cal{
H}_{k_\alpha}$ such that $\omega_c$ extends $\varphi,
E_\varphi = E_{\omega_c}|\Omega, p \in
D_{\omega_c}$, and  $\omega_c(p)=c$.   In
particular if $f\in E_\varphi$ then $f_{k_\alpha} \in
E_{\omega_c}$ and there exists $h\in \cal{ H}$ such that 
$\omega_cf_{k_\alpha}=h_{k_\alpha}$.   Noting that $\omega_c$
extends $\varphi$ and $\omega_c(p)=c$, we
deduce that the following holds.

If $f\in E_{\varphi}$, then there exists $h\in \cal{ H}$ such
that
\begin{eqnarray}\label{8}
\varphi(\lambda) f(\lambda) &=& h(\lambda)\qquad {\rm
if}\ \lambda \in D_\varphi\quad {\rm and}\\
c\ f(\alpha) &=& h(\alpha)\nonumber 
\end{eqnarray}
There are two possibilities: either $\alpha \in
D_\varphi$ or $\alpha \in \Omega\smallsetminus
D_\varphi$.  If $\alpha \in D_\varphi$, then noting that in equations
(\ref{8}), while $h$ depends on $c, \varphi$ and $f$ do not,
we deduce that $\varphi(\alpha) f(\alpha)=c f(\alpha)$ for
all $c$.  Hence, $f(\alpha)=0$ if $f\in E_\varphi$. 
Proposition \ref{2.7} implies that $\alpha$ is an ambiguous
point for $\varphi$.

Now assume that $\alpha$ is an ambiguous point for
$\varphi$.  Fix $c$ and define $\omega_c$ on $D_\varphi
\cup \{p\}\subseteq \Omega\cup \{p\}$ by the
formula
$$\omega_c(\lambda) = \left\{
\begin{array}{cl}\varphi(\lambda)\qquad  & {\rm if}\quad
\lambda \in D_\varphi\\ \\
c\qquad & {\rm if}\quad \lambda=p
\end{array}\right.$$
Then $\omega_c$ is a pseudomultiplier,
$\omega_c$ extends $\varphi, \omega_c(p) =
c$, and $E_{\omega_c}|\Omega \subseteq E_\varphi$. Hence
$\varphi$ will see $k_\alpha$ with value $c$ if we can show
that $f\in E_\varphi$ implies that $f_{k_\alpha} \in
E_{\omega_c}$.  A moment's reflection reveals  that this
last condition is equivalent to (\ref{8}).  Because
$\alpha$ is an ambiguous point for $\varphi$, it is easy to
check that equations (\ref{8})  follow from Proposition \ref{2.7}
when $\alpha \in D_\varphi$.

Let $\alpha \in \Omega\smallsetminus D_\varphi$.
By Proposition \ref{2.13} and Definition \ref{d2.12}, $\alpha$ is an ambiguous point for $\phi$ if and only if, for every $c\in \C$, $\varphi$ sees
$k_\alpha$ with value $c$.  
This concludes the proof of Theorem \ref{2.14}.
\end{proof}

\section{Ambiguous vectors and the ambiguous space}\label{vectors}

We have shown that a point $\alpha \in
\Omega$ is an ambiguous point for a pseudomultiplier
$\varphi$ if and only if $\varphi$ sees $k_\alpha$ with
arbitrary value.  We have also abstracted this property to say that 
a vector 
$v\in \cal{ H}$ is an ambiguous vector
for a pseudomultiplier
$\varphi$ of $\calh$ if $\varphi$ sees $v$ with arbitrary value (see Definition \ref{3.1}).  
 Thus, if $\varphi$ is a pseudomultiplier and $\alpha$ is a point,
$\alpha$ is an ambiguous point for $\varphi$ if and only if
$k_\alpha$ is an ambiguous vector for $\varphi$.  An
ambiguous vector need not be a kernel. In this section we shall study general  ambiguous vectors for $\varphi$.

The following  lemma gives an
operator-theoretic expression of the concepts of
seeing a vector and ambiguity.  Recall that if $\varphi$ is a
pseudomultiplier then $X_\varphi$ is a bounded linear
transformation from $E_\varphi$ into $\cal{ H}$.  Thus
$X^*_\varphi$ is a bounded linear
transformation from $\cal{ H}$ into $E_\varphi$. 
Recall also a basic fact in the theory of multipliers that if
$\varphi$ is a multiplier of $\cal H$ and $k_\lambda$ is the
reproducing kernel for $\lambda \in \Omega$, then
$X_\varphi^* k_\lambda = \overline{\varphi(\lambda)}
k_\lambda$. We generalize this fact  to
pseudomultipliers in the following lemma.

\begin{lemma}\label{3.4}
Let $\varphi$ be a pseudomultiplier of $\calh$ and let $v\in \cal{ H}$.  
Then 

{\rm (i)} $\varphi$ sees $v$ with value $c$ if and only if there
exists $u\in E^\perp_\varphi$ such that
\begin{equation}\label{9}
X^*_\varphi v= \overline c v + u.
\end{equation}

{\rm (ii)} $\varphi$ sees $v$ with value $c$ if and only if
\begin{equation}\label{12}
X^*_\varphi v = \overline c P_{E_\varphi}v.
\end{equation}

{\rm (iii)} If $\phi$ sees $v$ with two distinct values then $v$ is an ambiguous vector for $\phi$.
\end{lemma}

\begin{proof} (i). Fix a pseudomultiplier $\varphi, c \in
\mathbb C$, and $v \in \cal{ H}$.  Define $\psi$ on $D_\varphi
\cup \{p\}\subseteq  \Omega\cup \{p\}$ by the formula
$$
\psi(\lambda) = \left\{\begin{array}{cc} \varphi(\lambda) &
\qquad \fa \lambda \in D_\varphi\\ \\ c &\qquad \text{ if }\lambda =
p,\end{array}\right.
$$
so that $\varphi$ sees $v$ with value $c$ if and only if 
\begin{equation}\label{10.0}
(E_\psi)|\O = E_\varphi.
\end{equation}
By Lemma \ref{Epsi=Ephi}, equation \eqref{10.0} is equivalent to the inclusion
\begin{equation}\label{10}
(E_\psi)|\O \supset E_\varphi.
\end{equation}
Assume that $\varphi$ sees $v$ with value $c$, so that equation \eqref{10} holds.  It means that, for any $f\in E_\phi$ there exists $g\in\calh$ such that $g_v=\psi f_v$ on $D_\psi$.
It follows that $g(z)=\phi(z)f(z)= (X_\phi f)(z)$ for all $z\in D_\phi$, and (since $g_v(p)=(\psi f)(p)$)
\begin{equation}\label{10.2}
c\ip{f}{v} =\ip{g}{v}.
\end{equation}
Thus 
$$
 c\ip{f}{v} = \ip{X_\varphi f}{v} \qquad {\rm for\ all}\quad f\in
E_\varphi,
$$
and so $X_\phi^*v-\bar c v \in E_\phi^\perp$, and equation \eqref{9} holds.

Conversely, suppose there exists $u\in E^\perp_\varphi$ such
that $X^*_\varphi v = \overline c v + u$. Thus
$$
\ip{f}{X^*_\varphi v} = \ip{f}{\overline c v} \qquad {\rm for\
all}\quad f \in E_\varphi,
$$
from which it follows that
$$
\ip{X_\varphi f}{v} = c\ip{f}{v} \qquad {\rm for\ all}\quad f\in
E_\varphi,
$$
and hence
\begin{equation}\label{11}
\left(X_\varphi f\right) _v (p) = cf_v(p) \qquad {\rm
for\ all}\quad f \in E_\varphi\ .
\end{equation}
We claim that equation \eqref{10.0} holds.  To see this, consider any $f\in E_\phi$.  To show that $f\in E_\psi|\O$ it suffices to prove that 
$f_v\in E_\psi$, since $f_v|\O =f$.  Now $f_v\in E_\psi$ if and only if there exists $g\in\calh$ such that $\psi f_v = g_v$, which is to say that $\phi f=g$ on $D_\phi$, and, corresponding to evaluation at $p$, $c\ip{f}{v}=\ip{g}{v}$.  Equation \eqref{11} shows that we may obtain such an element $g$ of $\calh$ by choosing $g=X_\phi f$.
We have shown that $\phi$ sees $v$ with value $c$.
 This establishes Lemma \ref{3.4} (i).
 
 (ii) is  an immediate consequence of Part (i) and the
fact that ${\rm ran}\ X^*_\varphi \subseteq E_\varphi$.

(iii) 
Suppose $\phi$ sees $v$ with values $c_1$ and $c_2$, where $c_1\neq c_2$.
By part (ii) the following equations hold:
\[
X^*_\varphi v = \overline c_1 P_{E_\varphi}v \; \text {and} \; 
X^*_\varphi v = \overline c_2 P_{E_\varphi}v.
\]
Hence 
$(\overline c_1-\overline c_2) P_{E_\varphi}v =0$, and therefore  $P_{E_\varphi}v =0$ and
$X^*_\varphi v =0$.
Thus $X^*_\varphi v = \overline c P_{E_\varphi}v$ for any $c\in\C$, that is, $v$ is an ambiguous vector for $\phi$. 
\end{proof}

\begin{corollary}\label{1.71}
Scaling of visible vectors. \rm If $\phi$ sees a vector $v$ with value $c$ then, for any $\lam\in\C$, $\phi$ sees $\lam v$ with value $\bar\lam c$.  

\end{corollary}
\begin{proof}
Suppose that $\phi$ sees a vector $v\in\calh$ with value $c$
 and consider any $\lam\in\C$.  By Lemma \ref{3.4} there exists $u\in E^\perp_\varphi$ such
that $X^*_\varphi v = \overline c v + u$. Multiply this equation through by $\lam$ to get
\[
X^*_\varphi \lam v = \overline c \lam v + \lam u.
\]
Since $\lam u\in E^\perp_\varphi$, it follows from Lemma \ref{3.4} that $\phi$ sees $\lam v$ with value $\bar\lam c$.
\end{proof}

The set of ambiguous vectors of a pseudomultiplier $\phi$
plays an important role:  recall that in Definition \ref{3.1} we denoted it by $\cala_\phi$.
As a further corollary of Lemma \ref{3.4} we obtain an operator-theoretic
characterization of ambiguous vectors and the ambiguous space.

\begin{proposition}\label{3.6}
Let $\varphi$ be a pseudomultiplier and let $v\in \cal{ H}$. 
Then

{\rm (i)} $v$ is an ambiguous vector for $\varphi$ if and only if $v\in
E^\perp_\varphi \cap \ker X^*_\varphi$.

{\rm (ii)}
\begin{equation}\label{amb-space}
\cal{ A}_\varphi  = E^\perp_\varphi \cap \ker X^*_\varphi
\end{equation}
is a finite-dimensional subspace of $\cal H$.
\end{proposition}
 
\begin{proof} (i). Suppose $v$ is ambiguous for $\varphi$. 
By Lemma \ref{3.4},
$$X^*_\varphi v = \overline c P_{E_\varphi}v$$
for every $c\in \mathbb C$.  Taking $c=0$ gives $v\in \ker
X^*_\varphi$ and taking $c=1$ gives $v\in
E^\perp_\varphi$.  Conversely, if $v\in E^\perp_\varphi
\cap \ker X^*_\varphi$ then both sides of equation (\ref{12}) are
zero, hence $\varphi$ sees $v$ with arbitrary value $c$.  Thus $v$ is
ambiguous for $\varphi$. 

(ii) follows from Part (i).
\end{proof}

In the next 3 examples we calculate the ambiguous spaces of several pseudomultipliers.

\begin{example} \label{ambphi} \rm We showed in Example \ref{arch1-A} that 
 the pseudomultiplier	  $\phi(z)=1/z$ on $H^2$ has no ambiguities, directly from the definition of an  ambiguous vector.
Proposition \ref{3.6} reinforces this statement. The ambiguous vectors for $\phi$ are the members of $E_\phi^\perp\cap \ker X_\phi^*$. Recall that $E_\phi=zH^2$, and so $E_\phi^\perp$ is the space $H^2\ominus zH^2$ of constant functions on $\D$.  Consider any $f\in \cala_\phi$: then, since $f\in E_\phi^\perp$, $f = c\mathbf{1}$ for some $c\in\C$, and since $f\in\ker X_\phi^*$, $X_\phi^*f=0$. 
	Thus, for every $g\in H^2$,
	\begin{eqnarray*}
		0&=&\ip{X_\phi^*c\mathbf{1}}{zg}\\
		&=&\ip{c\mathbf{1}}{X_\phi zg}=c \ip{\mathbf{1}}{g}
	\end{eqnarray*}
	and therefore $c=0$.  Hence $\cala_\phi=\{0\}$. \qed
\end{example}

\begin{example}\label{ambchi} \rm 
	Consider the pseudomultiplier $\chi$ in Example \ref{intro2.3} ($\chi(t)=\sqrt{t}$ on $W^{1,2}[0,1]) $.
	Recall that $E_\chi = k_0^\perp$. Let us show that $\cala_\chi=\C k_0$.	
	By  Proposition \ref{3.6}, the ambiguous vectors for $\chi$ are the members of $E_\chi^\perp\cap \ker X_\chi^*$.
	We need to show that every vector $v=ck_0$ belongs to $\ker X_\chi^*$.
	Note that, for every $f\in E_\chi = k_0^\perp$,
	\begin{eqnarray*}
		\ip{X_\chi^*ck_0}{f} =\ip{ck_0}{X_\chi f}=c\overline{(\chi f)(0)}=0.
	\end{eqnarray*}
	Therefore the ambiguous space $\cala_\chi = \cals_\chi = E_\chi^\perp = \C k_0$. \qed
\end{example}

\vspace{1mm}

\begin{example}\label{3.3} 
\rm Let $\cal{H} $ be the Hilbert  space of functions $f$ on  $[0, 1]$
such that both $f$ and $f'$ are absolutely continuous on $[0,1]$ and 
such that the weak derivative $f''$ is in $L^2$, equipped with the inner product
$$
\ip{f}{ g}_\cal{ H}= f(0) g\overline{(0)} + f^\prime(0)
g^\prime\overline{(0)} + \int^1_0 f^{\prime\prime}(t)
\overline {g^{\prime\prime}(t)} dt.
$$
If $\chi(t) = \sqrt{t}$, then we claim that
$$E_\chi = \{f \in \cal{ H}:f(0) = f^\prime(0) = 0\}.$$  
Also $E^\perp_\chi$ is a 2 dimensional subspace consisting of ambiguous vectors.
\begin{proof}
To prove that $E_\chi = \{f:f(0) = f^\prime(0) = 0\}$, consider any function $f\in\calh$ such that $f(0) = f^\prime(0) = 0$. We have to show that $\chi f$ and $(\chi f)'$ are absolutely continuous on $[0,1]$ and $(\chi f)''$ is in $L^2$.
It follows from Example \ref{intro2.3} that the function
\[
(\chi f)'(t) = \frac{1}{2\sqrt{t}}f(t) + \sqrt{t}f'(t)
\]
 is in $L^2(0,1) \subset L^1(0,1)$, and so $\chi f$ is absolutely continuous on $[0,1]$.
 
Now let us show that $(\chi f)''\in L^2(0,1)$.  We have, for $t\in (0,1]$,
 \be\label{2deriv}
 (\chi f)''(t) = -\tfrac 14 t^{-\tfrac 32}f(t) + \frac{1}{\sqrt{t}}f'(t) + \sqrt{t}f''(t).
 \ee
 The  third term in the sum on the right hand side of equation \eqref{2deriv} is in $L^2(0,1)$,
 because $f''$ is in $L^2(0,1)$. The second term in the sum of equation \eqref{2deriv}  is  $\frac{1}{\sqrt{t}}f'(t)$. 
  Since  $f'(0)=0$, for $t \in (0,1)$,
 \begin{eqnarray}\label{W22.01}
 t^{-\tfrac 12}f'(t) &= &t^{-\tfrac 12}\int_0^t f''(u)\, du \nonumber \\
 	&=&\sqrt{t} \left( t^{-1}\int_0^t f''(u)\, du \right).
 \end{eqnarray}
 By assumption  $f''$ is in $L^2(0,1)$, and so, by Hardy's inequality,   $\left( t^{-1}\int_0^t f''(u)\, du \right)$ is in  $L^2(0,1)$. Thus  $\frac{1}{\sqrt{t}}f'(t)$ is in  $L^2(0,1)$.

  We claim that the first term on the right hand side of equation \eqref{2deriv} is  in $L^\infty(0,1)$.  Indeed,  for $x\in (0,1)$,
 \begin{eqnarray}\label{W22.1}
 x^{-\tfrac 32}f(x) &= &x^{-\tfrac 32}\int_0^x f'(t)\, dt\qquad\qquad (\text{let } u=t/x) \nonumber \\
 	&=&x^{-\tfrac 32}\int_0^1 f'(xu) x \,du   \nonumber\\
 	&=& x^{-\tfrac 12}\int_0^1 f'(xu) \,du   \nonumber\\
 	&=& x^{-\tfrac 12}\int_0^1 \int_0^{xu} f''(t) \, dt \,du.
 \end{eqnarray}
 By H\"older's inequality,
 \be\label{W22.2}
 \left|\int_0^{xu} f''(\xi)\ d\xi \right| \leq |xu|^\half \left\{\int_0^{xu} |f''(\xi)|^2\ d\xi\right\}^\half  \leq |xu|^\half\|f''\|_{L^2}.
\ee
Therefore
 \begin{eqnarray}\label{W22.3}
 |x^{-\tfrac 32}f(x)| &= & x^{-\tfrac 12} \left|\int_0^1 \int_0^{xu} f''(t) \, dt \,du \right| \nonumber \\
 &\le & x^{-\tfrac 12} \int_0^1 |xu|^\half\|f''\|_{L^2}\,du\nonumber \\
 & = & \|f''\|_{L^2}\int _0^1 u^{\half} \, du = \tfrac 23 \|f''\|_{L^2},
 \end{eqnarray}
 so that $x^{-\tfrac 32}f(x)$ is bounded on $[0,1]$.  Hence the weak derivative $(\chi f)'' \in L^2(0,1)$, and consequently $(\chi f)'' \in L^1(0,1)$. Therefore 
$$
(\chi f)'(x) -(\chi f)'(0)  = \int_0^x (\chi f)''(t)\ dt.
$$ 
Note that, by equation \eqref{W11.1} and the fact that $f(0)= 0$,
\[
(\chi f)'(t) = \frac{1}{2\sqrt{t}}f(t) + \sqrt{t}f'(t)= \frac{1}{2}\sqrt{t} \left( t^{-1}\int_0^t f'(u)\, du \right) + \sqrt{t}f'(t) .
\]
Since  $f'$ is in $L^2(0,1)$, by Hardy's inequality, $\left( t^{-1}\int_0^t f'(u)\, du \right)$ is in  $L^2(0,1)$, and so, because $f'(0)= 0$, $(\chi f)'(0)=0$. Hence $(\chi f)'$ is absolutely continuous on $[0,1]$.
 
 Therefore, for $f\in\calh$ such that $f(0) = f^\prime(0) = 0$, $\chi f \in \calh$.
Conversely, if $f\in\calh$ is such that the condition  $f(0) = f^\prime(0) = 0$ is not satisfied, that is $f(0)\neq 0$ or  $f'(0)\neq 0$ then $(\chi f)''\notin L^2[0,1]$ 
and so $\chi f \notin \calh$.  Thus, $\chi$ is a pseudomultiplier of $\calh$ and
 $E_\chi = \{f \in \calh:f(0) = f^\prime(0) = 0\}$.
 For $f \in \calh$, we have $|f(0)| \le \|f\|_{\calh}$ and $|f'(0)| \le \|f\|_{\calh}$, and so there exist $k_0, s_0 \in \calh$ such that,  for all $f \in \calh$,
 $\ip{f}{k_0}= f(0)$ and  $\ip{f}{s_0}= f'(0)$.
 Thus $E_\chi^\perp = \span\{ k_0, s_0\}$, 
  
 Let us show that $\cala_\chi=\C k_0 +\C s_0$.	
	By  Proposition \ref{3.6}, the ambiguous vectors for $\chi$ are the members of $E_\chi^\perp\cap \ker X_\chi^*$. We need to show that every vector $v=c_1 k_0 +c_2 s_0$ belongs to $\ker X_\chi^*$.
	Note that, for every $f \in E_\chi = \{f \in \calh:f(0) = f^\prime(0) = 0\}$, since
	$(\chi f)(0) = (\chi f)'(0)=0$,
	\begin{eqnarray*}
		\ip{X_\chi^*(c_1k_0 +c_2 s_0)}{f} =\ip{c_1k_0 +c_2 s_0}{X_\chi f}=
		c_1\overline{(\chi f)(0)} + c_2\overline{(\chi f)'(0)}=0.
	\end{eqnarray*}
	Therefore the ambiguous space $\cala_\chi = \cals_\chi = E_\chi^\perp = \span\{ k_0, s_0\}$.
\end{proof}
\end{example}

\section{Definable vectors and polar vectors}\label{polar-vectors}

In this section we shall study definable  vectors, polar vectors, and, more generally, the polar space of a pseudomultiplier $\phi$.
By way of a preliminary observation relevant to the notion of a definable vector, we note that, if $\phi$ sees $v$  then  $\varphi$ is constant  on the coset $v+\cal{ A}_\varphi$. 

\begin{lemma}\label{3.8}
Let $\varphi$ be a pseudomultiplier, let $v\in \cal{ H}$, and
let $a \in \cal{ A}_\varphi$.  If $\varphi$ sees $v$ with value
$c$, then $\varphi$ sees $v+a$ with value $c$.  In particular,
if $\varphi$ sees $v$, then $\varphi$ sees $v+a$. 
\end{lemma}

\begin{proof} Suppose $\varphi$ sees $v$ with value
$c$.  Thus, by Lemma \ref{3.4}, there exists $u\in
E^\perp_\varphi$ such that 
$$X^*_\varphi v = \overline c v +u.$$
Noting that $X^*_\varphi a =0$, we see that
\begin{eqnarray*}
X^*_\varphi(v+a) &=& \overline c v + u\\  
&=& \overline c(v+a) + u-\overline c a.
\end{eqnarray*}
Since $\overline c a \in E^\perp_\varphi$, so also
$u-\overline c a \in E^\perp_\varphi$.  Hence by Lemma
\ref{3.4}, $\varphi$ sees $v+a$ with value $c$. \end{proof}

 If $\varphi$ is a pseudomultiplier of $\cal H$ then $\varphi$ naturally extends to
have values on its visible vectors.  These values are
arbitrary on the ambiguous vectors  but unique on the
unambiguous visible vectors. We now know
from Lemma \ref{3.8}, that $\varphi$ is constant  on each
coset $v+\cal{ A}_\varphi$.  Since, for any visible vector $v$, the
coset $v+\cal{ A}_\varphi$ contains a unique visible
vector that is orthogonal to $\cal{ A}_\varphi$,  to wit $P_{{\cala_\phi^\perp}} v$,
the following definition is  reasonable.

\begin{definition}\label{3.9}
{\rm Let $\varphi$ be  a pseudomultiplier.  We say a vector
$d\in \cal{ H}$ is {\it definable for} $\varphi$ if $d\not= 0,
\varphi$ sees $d$, and $d\perp \cal{ A}_\varphi$.  We let
$\mathbf{D}_\varphi$ denote the set of vectors that are
definable for
$\varphi$.  }
\end{definition}

By Lemma \ref{3.4}(iii),
 if $\phi$ sees $v$ with two distinct values then $v$ is an ambiguous vector for $\phi$.
Therefore, if $d$ is definable for $\varphi$ then  there
exists a {\em unique} $c\in \mathbb C$ such that $\varphi$ sees
$d$ with value $c$.  
  We refer to this $c$ as the {\it value of
$\varphi$ at $d$} and denote it by $\varphi(d)$.  

We remark that the exclusion of the vector $0$ from being
definable has the following happy consequences.  No
vector is both definable and ambiguous, and, by Lemma \ref{3.8}, every vector
in $\mathbf{ D}_\varphi + \cal{ A}_\varphi$ is an unambiguously
seen vector.

\begin{remark} \rm
We proved in Example \ref{definableex} 
that, for  the pseudomultiplier  $\phi(z)=1/z$ on $H^2$, the set of  definable vectors 
 $\mathbf{D}_\phi$  comprises the non-zero scalar multiples of the
kernels $k_\lam $, for $\lam \in  \D \setminus \{0\}$.	
The vector $k_0$ is a limit of the definable vectors $k_{(1/n)}$ for $\phi$
$$k_0 = \lim\limits_{n\to\infty} k_{(1/n)}.$$
However, $k_0$ is not itself definable because, for all $n\geq 1$, $\phi(k_{1/n}) = n$, and so
\[
\lim_{n\to\infty} \phi(k_{1/n}) = \infty.
\]
Thus $k_0$ is a polar vector for $\phi$. 

 $\cal{ A}_\varphi$ is both a
subspace  and a topologically closed set in $\calh$.  Not so $\mathbf{ D}_\varphi$. 
 Quite apart from the fact that $0$ is not in
$\mathbf{ D}_\phi $, in
the example above $\mathbf{ D}_\varphi$ is very far from being a
subspace but not far from being a closed set in $\cal H$.
For example, in Example \ref{definableex} one need only adjoin $\C k_0$  to $\mathbf{D}_\phi$ to obtain a closed set. 
\qed
\end{remark}

We shall repeatedly have a use for the following minor technical observation.
\begin{lemma}\label{3.91}
If $d$ is a definable vector of a pseudomultiplier $\phi$ then $ P_{E_\varphi}d \neq 0$.
\end{lemma}
\begin{proof}
Suppose $ P_{E_\varphi}d = 0$, and so
$d\in E_\phi^\perp$.  By Lemma \ref{3.4}, since $d$ is visible for $\phi$ with value $\phi(d)$,
\[
X_\phi^* d = \overline{\phi(d)} P_{E_\phi} d = 0,
\]
and thus $d\in E_\phi^\perp \cap \ker X_\phi^*$.  Hence, by Proposition \ref{3.6}, $d \in \cala_\phi$.  However, by the definition of definable vectors, $d\neq 0$ and $d\perp \cala_\phi$, which is a contradiction.  Hence 
$ P_{E_\varphi}d \neq 0$.
\end{proof}

\begin{proposition}\label{3.13}
Let $\varphi$ be a pseudomultiplier.  Then $\varphi$ is a
continuous function on $\mathbf{ D}_\varphi$ with respect to the metric of $\calh$.
\end{proposition}

\begin{proof} Since $\mathbf{D}_\varphi$, with the topology
inherited from $\cal H$, is a metric space, it suffices to
prove sequential continuity.  Let $d_n\to d$ in $\mathbf{ D}_\varphi$.  By Lemma \ref{3.4},
\begin{eqnarray*}
X^*_\varphi d_n &=& \overline{\varphi (d_n)} \
P_{E_\varphi} d_n,\\ \\
X^*_\varphi d &=& \overline{\varphi (d)} \
P_{E_\varphi} d.
\end{eqnarray*}
By Lemma \ref{3.91}, $P_{E_\varphi} d \neq 0$.  On letting $n\to \infty$ we therefore find
that $\varphi(d_n) \to \varphi(d)$.\end{proof}

Recall that we introduced the {\em polar vectors} of a pseudomultiplier in Definition \ref{3.11}:
 for a pseudomultiplier  $\varphi$  of $\calh$ and for $p\in \cal{H}$,  we say that $p$ is a {\it polar vector for} $\varphi$ if $p\not= 0$ and
there exists a sequence $\{d_n\} \subseteq \mathbf{ D}_\varphi$ such that  $d_n\to
p$ (with respect to  $\|\cdot \|_{\cal{H}}$) and $\varphi(d_n) \to \infty$ as $n\to \infty$.
We define the {\it polar space $\cal{ P}_\varphi$ of $\varphi$} to be the set of vectors $v\in \cal{ H}$ such that	either $v$ is a polar vector of $\varphi$ or $v=0$.

\begin{example}\label{3.2} \rm 
Let $\cal{ H}=H^2$, the classical Hardy space on
$\mathbb D$, let $n \ge 2$, and define $\varphi$ on $\mathbb D$ by
$$\varphi(\lambda) = \left\{
\begin{array}{cc}\frac{1}{\lambda^n} &\qquad \text{ if } \lambda \not=
0\\ \\ 1&\qquad \text{ if } \lambda=0.\end{array}\right.
$$
Here  
 $$E_\phi= \{f\in H^2: f(0)=f'(0)=\dots = f^{(n)}(0)=0\}= z^{n+1}H^2.$$

Let us describe the space of ambiguous vectors $\cala_\phi$ for $\phi$.
By Proposition \ref{3.6},
\[
\cala_\phi = E_\phi^\perp \cap \ker X_\phi^*.
\]
The operator $X_\phi: z^{n+1}H^2 \to H^2$ is given by $X_\phi= (S^*)^{n} P_{z^{n+1}H^2}^*$
where $S$ denotes the forward shift operator on $H^2$ and $P_{z^{n+1}H^2}:H^2 \to z^{n+1}H^2$ is the orthogonal projection operator (so that $P_{z^{n+1}H^2}^*$ is the injection operator $z^{n+1}H^2 \to H^2$). Thus $X_\phi^*= P_{z^{n+1}H^2} S^{n}$ and so, for $f\in H^2$,
\begin{eqnarray*} 
X_\phi^*f = 0 &\Leftrightarrow& S^n f \in H^2\ominus z^{n+1}H^2 \\
            &\Leftrightarrow& z^n f \in \span\{1,z, \dots, z^n\} \\
	   &\Leftrightarrow& f\text{ is constant.}
\end{eqnarray*}
Hence, the space of ambiguous vectors for $\phi$ is
\[
\cala_\phi = E_\phi^\perp \cap \ker X_\phi^*= (H^2\ominus z^{n+1}H^2)\cap \C\mathbf{1} = \C\mathbf{1}.
\]

Let us describe the set of definable vectors $\mathbf{ D}_\varphi$ of $\phi$. By Lemma \ref{3.4}, for 
 $f\in H^2$,  $\varphi$ sees $f$ with value $c$ if and only if there
exists $u\in E^\perp_\varphi$ such that
\begin{equation}\label{9.9}
X^*_\varphi f= \overline c f + u,
\end{equation}
equivalently,  if and only if there exist $a_0, a_1, \dots, a_n \in\C$ such that
\[
X_\phi^*f = \bar c f+a_0+a_1 z+\dots +a_n z^n,
\]
that is, if and only if there exist $a_0, a_1, \dots,  a_n \in\C$ such that
\be\label{9.10}
P_{z^{n+1}H^2} z^n f-\bar c f=a_0+a_1 z+\dots +a_n z^n.
\ee
Now $P_{z^{n+1}H^2} z^n f = z^n(f-f(0))$, and so equation \eqref{9.10} is equivalent to
\be\label{9.100}
z^n(f-f(0)) -\bar c f=a_0+a_1z+\dots +a_n z^n.
\ee
Substitution of $z=0$ shows that equation \eqref{9.100} implies that $a_0=-\bar c f(0)$,
and so equation \eqref{9.100} can be written
\[
(z^n-\bar c)(f-f(0)) = a_1 z+\dots +a_n z^n.
\]
Thus $\phi$ sees $f$ with value $c$ if and only if there exist $a_1, a_2,\dots,a_n \in\C$ such that
\[
f=f(0)+ \frac{a_1z+\dots +a_n z^n}{z^n-\bar c} \in H^2.
\]
The rational function on the right hand side is in $H^2$ if and only if $|c| > 1$, and so the visible vectors for $\phi$ consist of the rational functions of the form
\[
f(z)=a_0 + \frac{a_1z+\dots +a_n z^n}{z^n-\bar c} \text{ where } |c| > 1 \text{ and } a_0,a_1, \dots, a_n\in\C,
\]
or alternatively, if we define $b_j=-\bar c\inv a_j$ for $j=1,2, \dots, n$, of the functions
\be\label{9.11}
f(z)=a_0+ \frac{b_1z+b_2z^2+ \dots +b_nz^n }{1-\bar c\inv z^n} \text{ where } |c| > 1 \text{ and } a_0,b_1,b_2, \dots, b_n\in\C.
\ee
We showed above that $\cala_\phi$ consists of the constant functions, and so the visible function $f$ in equation \eqref{9.11} is definable if and only if $f\neq 0$ but $f(0)=0$.
Thus
\be\label{9.12}
\mathbf{D}_\phi = \left\{f(z)=\frac{b_1z+b_2z^2+ \dots +b_nz^n }{1-\bar c\inv z^n}: |c|> 1\text{ and }b_1,b_2, \dots, b_n \in\C \text{ are not all zero}    \right\}.
\ee
   
Now we can describe the polar space $\calp_\phi$ for $\phi$. For every non-zero polar vector
$p$ there exists a convergent  sequence of definable vectors 
$$
\left\{d_k=\frac{b_{1k}z+b_{2k}z^2 + \dots + b_{nk}z^n }{1-\bar{c_k}\inv z^n} \right\}_{k=1}^\infty \text{ in } \mathbf{D}_\phi,
$$ 
where $b_{1k}, b_{2k}, \dots, b_{nk} \in \C$, not all zero, are  such that $d_k \to p$ and 
$\phi(d_k) = c_k \to \infty $ as $k \to \infty$.
The non-zero limits of such convergent sequences $\{d_k\}_{k=1}^\infty$ in $H^2$ coincide with 
$\span\{z,z^2, \dots, z^n \}\setminus\{0\}$, so that $\calp_\phi = \span \{z,z^2, \dots, z^n\}$.
We may observe that 
$$
\cala_\phi \oplus \calp_\phi= \C\mathbf{1} \oplus \span\{z,z^2, \dots, z^n\} = H^2\ominus z^{n+1}H^2 = E_\phi^\perp.
$$ 
\qed
\end{example}

\begin{remark}\label{instructive} \rm The combination of Examples \ref{3.2} and \ref{ambphi} (with $n=1$) illustrates an interesting fact: by {\em un}defining the pseudomultiplier  $\phi$ of Example \ref{3.2} at the point $0$ we obtain a pseudomultiplier with a smaller singular space -- see Example \ref{ambphi}.  The polar space is unchanged, while the ambiguous space is reduced from $\C\mathbf{1}$ in Example \ref{3.2}  to $\{0\}$ in Example \ref{ambphi}.
\end{remark}

The following proposition shows how  the notion of a pseudopole of a pseudomultiplier, as introduced in \cite{AY1} (see Remark \ref{pseudopole} above), relates to  polar vectors. 
 In this paper by the closure of $D_\phi$  we mean the closure in $\O$ with respect to the natural metric on $\O$, induced by the norm on $\calh$ and the identification of a point of $\O$ with its reproducing kernel in $\calh$.

\begin{proposition} \label{thm-pseudopole} If $\al$ is a pseudopole of a pseudomultiplier $\phi$ and $\al$ is in the closure of $D_\phi$ in $\O$ then $ P_{\cala_\phi^\perp} k_\al$ is a polar vector of $\phi$.
\end{proposition}

\begin{proof}
By assumption, $\al$ is in the closure of $D_\phi$ in $\O$, and so we can choose a sequence $(\al_n)$ in $D_\phi$ such that $\al_n \to \al$ as $n \to \infty$.
Since $\al$ is a pseudopole of $\phi$ there exists $h \in E_\phi$ and $g\in\calh$ such that $g=X_\phi h$, $h(\al)=0$ and $g(\al) \neq 0$.  For $n \geq 1$ let
\[
d_n= P_{\cala_\phi^\perp} k_{\al_n} = k_{\al_n} - P_{\cala_\phi} k_{\al_n}.
\]
Since $\phi$ sees $k_{\al_n}$ with value $\phi(\al_n)$, it follows from Lemma \ref{3.8} that $\phi$ sees $d_n$ with value $\phi(\al_n)$.  Since $d_n \perp \cala_\phi$, $d_n$ is a definable vector for $\phi$ and $\phi(d_n)=\phi(\al_n)$.  Moreover
\[
0 \neq g(\al) = \ip{g}{k_\al} = \lim_n \ip{g}{k_{\al_n}},
\] 
while, for all $\lam\in D_\phi$,
\[
\ip{g}{k_\lam} = g(\lam) = \phi(\lam)h(\lam).
\]
Hence
\[
\lim_n \ip{g}{k_{\al_n}} = \lim_n \phi(\al_n) h(\al_n) = g(\al)\neq 0.
\]
Since also $h(\al_n) \to h(\al)=0$, we have $\phi(\al_n) \to \infty$ and therefore $\phi(d_n) \to \infty$ as $n \to\infty$.  Since the sequence of definable vectors 
$d_n= P_{\cala_\phi^\perp} k_{\al_n}$ tends to $P_{\cala_\phi^\perp} k_{\al}$ as $n \to \infty$ it follows that 
 $ P_{\cala_\phi^\perp} k_\al$ is a polar vector of $\phi$.
\end{proof}

The nature of polar vectors is much more rigid than one might at
first expect.  This is brought out in the following
proposition, which gives a clean relationship between $\mathbf{D}_\phi$ and the set $\mathcal{P}_\phi$ of polar vectors of $\phi$.
\begin{proposition}\label{3.12}
Let $\varphi$ be a pseudomultiplier and let $p\in \cal{ H}$
with $p\not= 0$.  The following are equivalent.
\begin{enumerate}

\item[\rm (i)] $p$ is a polar vector for $\varphi$.

\item[\rm (ii)]$p\in \mathbf{ D}_\varphi^-$ and for every
neighborhood $U$ of $p , \varphi$ is unbounded on $U\cap
\mathbf{ D}_\varphi$. 

\item[\rm (iii)] $p\in \mathbf{ D}_\varphi^- \backslash \mathbf{ D}_\varphi$.

\item[\rm (iv)] $p\in\mathbf{ D}_\varphi^-$ and
$\lim\limits_{\stackrel{\textstyle{d\to p}}{d\in \mathbf{ D}_\varphi}}\ 
\varphi (d) = \infty$.
\end{enumerate}
\end{proposition}

\begin{proof} It is clear that (iv) $\Rightarrow$ (i)
$\Rightarrow$ (ii).  Suppose that (ii) holds but (iii) is false,
so that $p\in \mathbf{ D}_\varphi$.
By Lemma \ref{3.91}, $P_{E_\varphi} d \neq 0$.  We may therefore assume
$\|P_{E_\varphi}p\|=1$.  Let $U$ be a neighborhood of $p$ in
$\cal H$ such that, for all $v\in U$,
\begin{equation}\label{13}
\|(X^*_\varphi -
\overline{\varphi(p)}\ P_{E_\varphi}) v
\| <1
\end{equation}
and $$\|P_{E_\varphi} v\| > \frac{1}{2}.$$
Consider $v\in \mathbf{ D}_\varphi \cap U$.  Since $v$ is
definable,
$$X^*_\varphi v = \overline{\varphi(v)} \ P_{E_\varphi} v.$$
Substituting in (\ref{13}) we have
$$\|(\overline{\varphi(v)} - \overline{\varphi(p)}) \
P_{E_\varphi}v \|<1$$and so
$$|\varphi(v) - \varphi(p)| <2.$$
Thus $\varphi$ is bounded on $U\cap \mathbf{ D}_\varphi$,
contradicting (ii).  Thus (ii) $\Rightarrow$ (iii).

Now suppose (iii) holds but (iv) is false: $p\in \mathbf{ D}_\varphi^- \smallsetminus \mathbf{ D}_\varphi$ but
$\varphi(d) \not\rightarrow \infty$ as $d\to p$ in $\mathbf{ D}_\varphi$.  Then there exists a  sequence $(d_n)$ in $\mathbf{ D}_\varphi$ and a positive constant $C$ such that $d_n\to
p$ and $|\varphi(d_n)|\le C$ for all $n$. By Lemma \ref{3.4},
for all $n\in \mathbb N$, 
$$X^*_\varphi d_n = \varphi\overline{(d_n)}\ P_{E_\varphi}
d_n.$$
Passing to a subsequence if necessary, we can assume
$\varphi(d_n) \to z\in \mathbb C$.  Taking limits in the
foregoing equation we find
$$X^*_\varphi p = \overline z P_{E_\varphi} p.$$
Thus $\varphi$ sees $p$ with value $z$.  Since each
$d_n\perp \cal{ A}_\varphi$, we have also $p\perp \cal{
A}_\varphi$.  By assumption $p\not= 0$, and so $p\in \mathbf{ D}_\varphi$, a contradiction.  Thus (iii) $\Rightarrow$ (iv).
\end{proof}

We continue this section with a fact about $\varphi$ as
a function on $\mathbf{ D}_\varphi$ (a corollary of Proposition \ref{3.12}).

We say that a pseudomultiplier $\phi$ is {\em locally bounded on $\mathbf{D}_\phi^-$}
if, for every point $d\in\mathbf{D}_\phi^-$ there is a neighbourhood $U$ of $d$ in $\calh$ such that $\phi$ is bounded on $U\cap\mathbf{D}_\phi$.

\begin{theorem}\label{3.14}
Let $\varphi$ be a pseudomultiplier.  Then $\varphi$ is  locally 
bounded on $\mathbf{ D}_\varphi^-$ if and only if $\varphi$ has no
polar vectors.
\end{theorem}

\begin{proof} Suppose  $\varphi$ has a polar vector $p$. By  Proposition \ref{3.12}, (i)$\Leftrightarrow$(ii), for every neighbourhood $U$ of $p$ in $\calh$, $\phi$ is unbounded on $U\cap\mathbf{D}_\phi$. Thus $\varphi$ is   not locally 
bounded on $\mathbf{ D}_\varphi^-$.

Conversely,
suppose $\varphi$ has no polar vectors.  By Proposition \ref{3.12}, (i)$\Leftrightarrow$(ii), for every point $d\in\mathbf{D}_\phi^-$ there is a neighbourhood $U$ of $d$ in $\calh$ such that $\phi$ is bounded on $U\cap\mathbf{D}_\phi$, that is, $\phi$ is locally bounded on $\mathbf{D}_\phi^-$.
\end{proof}

The definable vectors provide an alternative viewpoint on the polar space.

\begin{theorem}\label{4.3} Let $\phi$ be a pseudomultiplier of a Hilbert function space $\calh$ on a set $\O$. Suppose $\dim \cal{ H} = \infty$. Then

{\rm (i)}  $\mathbf{ D}_\varphi$ is non-empty. 

{\rm (ii)} $\cal{ P}_\varphi=\mathbf{ D}_\varphi^-\smallsetminus \mathbf{ D}_\varphi$.
\end{theorem}

\begin{proof}
(i). Let us show that $\mathbf{ D}_\varphi = \emptyset$ implies that   $\cal H$ is finite-dimensional.

 For any point $\alpha \in D_\varphi$, let $d_\alpha \stackrel{\rm def}{=} P_{\cal{ A}^\perp_\varphi} k_\alpha$.
 Let us show that  $d_\alpha$ is definable if and only if $\alpha$ is not ambiguous.  Suppose that $\alpha$ is ambiguous.  Then by Theorem  \ref{2.14}, $k_\alpha \in \cala_\phi$, and so $d_\alpha = 0$,
and therefore $d_\alpha$ is not definable.
Suppose now that $\alpha$ is not ambiguous.  
Then $\phi$ sees $k_\alpha$ with the unique value $\phi(\alpha)$. By definition  $d_\alpha \perp \cal{ A}_\varphi$.
By Lemma \ref{3.8}, $\phi$ sees $d_\alpha$ with the unique value $\phi(\alpha)$. Thus $d_\alpha$ is definable.

  By assumption,  $\mathbf{ D}_\varphi = \emptyset$, and so, for every $\alpha \in D_\varphi$, 
$d_\alpha$ is not definable, that is, every $\alpha \in D_\varphi$ is ambiguous. Hence $\cal{ A}_\varphi$
contains $\{k_\alpha: \alpha \in D_\varphi\}$, and so also
its closed linear span, which is $\cal H$.  
By Proposition \ref{3.6}, $\cala_\phi= E_\phi^\perp \cap \mathrm{ker} X_\phi^*$, 
which is a finite-dimensional subspace of $\calh$.
This implies that $\cal H$ is finite-dimensional. 
Therefore if $\cal H$ is infinite-dimensional, then  $\mathbf{ D}_\varphi$ is non-empty. 

(ii) Since $\cal H$ is infinite-dimensional, by (i),
$ \mathbf{ D}_\varphi$ is non-empty: pick $d\in  \mathbf{ D}_\varphi$. By Corollary \ref{1.71}, $n^{-1} d \in  \mathbf{ D}_\varphi$ for $n\in
\mathbb N$, and so $0\in \overline{\mathbf{ D}}_\varphi$.  By
definition, $0\not\in  \mathbf{ D}_\varphi$ and so $0\in \ov{\mathbf{ D}_\varphi}\smallsetminus  \mathbf{ D}_\varphi$.  The result is
now immediate from the equivalence of (i) and (iii) in
Proposition \ref{3.12}. \end{proof}

\section{The singular space}\label{singular_space}

Pseudomultipliers fail to act as operators on all of $\cal
H$ for one of two reasons.  Either they have an
ambiguous point (such as $\sqrt{t}$ acting on Sobolev
space) or they have a polar vector (such as $\frac{1}{z}$ acting on
$H^2$).  In this section we prove a theorem to the effect that that every singularity
of a pseudomultiplier can be written in a unique way as the sum of a polar vector
and an ambiguous vector.  As a corollary of our basic theorem one
obtains the interesting fact that if one adjoins the zero
vector to the set of polar vectors of a pseudomultiplier  then one
obtains a finite-dimensional subspace  of $\cal H$. 
Furthermore, this subspace is precisely the  set of
vectors that can be approximated by definable vectors
but that are not themselves definable.

\begin{definition}\label{4.1} 
{\rm  Let $\varphi$ be a
pseudomultiplier.   We define the {\it order of $\varphi $}
to be $\dim E^\perp_\varphi$.  If the order of $\varphi$ is
$m$, then we say that $\varphi$ is an {\it
$m$-pseudomultiplier}. }
\end{definition}

\begin{theorem}\label{4.2} 
Let $\varphi$ be a pseudomultiplier.  $\cal{ P}_\varphi$ is a closed
subspace of $\cal{ H}$ and 
$$\cal{ S}_\varphi=\cal{ P}_\varphi \oplus \cal{
A}_\varphi.$$
\end{theorem}

\begin{proof} First recall that, by  Proposition \ref{3.6},  
$\cal{ A}_\varphi= E^\perp_\varphi \cap \ker X^*_\varphi$ is a closed subspace of $\cal{ S}_\varphi= E^\perp_\varphi$. Hence we need to show that
\begin{equation}\label{17}
\cal{ P}_\varphi=\cal{ S}_\varphi \ominus \cal{ A}_\varphi.
\end{equation}
We first show that $\cal{ P}_\varphi \subseteq \cal{
S}_\varphi\ominus \cal{ A}_\varphi$.  Consider $p\in \cal{
P}_\varphi$.  We can suppose $p\not= 0$, so that $p$ is a
polar vector for
$\varphi$.  Choose $\{d_n\}\subseteq \mathbf{ D}_\varphi$ with
$d_n\to p$ and $\varphi(d_n)\to \infty$.
We claim that $p\in \cal{ S}_\varphi$.  By Lemma \ref{3.4},
\begin{equation}\label{18}
X^*_\varphi d_n = \overline {\varphi(d_n)} \big(d_n-P_{\cal{
S}_\varphi} d_n\big).
\end{equation}
Since $X^*_\varphi d_n \to X^*_\varphi p$ and
$\varphi(d_n)\to \infty$ we conclude from equation \eqref{18} that
$d_n-P_{\cal{ S}_\varphi} d_n \to 0$.  Hence $p-P_{\cal{
S}_\varphi}p=0$, that is, $p\in \cal{ S}_\varphi$.  Next, since
$d_n\perp \cal{ A}_\varphi$, so also $p=\lim d_n \perp \cal{
A}_\varphi$.  Thus $p\in \cal{ S}_\varphi \cap \cal{
A}_\varphi^\perp = \cal{ S}_\varphi \ominus \cal{
A}_\varphi$, and we have proved that $\cal{P}_\varphi \subseteq \cal{ S}_\varphi \ominus \cal{A}_\varphi$.

We now show that $\cal{ S}_\varphi \ominus \cal{
A}_\varphi
\subseteq \cal{ P}_\varphi$.  Consider $v\in \cal{ S}_\varphi
\ominus \cal{ A}_\varphi$.  If $v=0$, then obviously $v\in \cal{
P}_\varphi$.  Assume that $v\not= 0$.  We claim that $v$ is a
polar vector for $\phi$.  Choose a sequence $\{c_n\} \subseteq \mathbb C$ with
$c_n\to \infty$ and $\overline c_n \not\in
\sigma(X^*_\varphi)$ for all $n$.  Define $v_n \in \cal{ H}$ by
the formula
\begin{equation}\label{19}
v_n= \overline c_n \left(\overline c_n I -
X^*_\varphi\right)^{-1} v.
\end{equation}
We claim that $\varphi$ sees $v_n$ with value $c_n$.  We have
\begin{eqnarray}\label{X*vn}
X^*_\varphi v_n &=& X^*_\varphi \overline c_n
\left(\overline c_n I - X^*_\varphi\right)^{-1} v \nonumber\\
&=&\overline c_n\left(\left(X^*_\varphi - \overline
c_n I\right) \left(\overline c_n I - X^*_\varphi\right)^{-1}
v + \overline c_n\left(\overline c_n I -
X^*_\varphi\right)^{-1} v\right) \nonumber\\
 &=&
\overline c_n v_n +  \overline c_n v
\end{eqnarray}
and $v\in \cal{ S}_\varphi$.  Hence by Lemma \ref{3.4},
$\varphi$ sees $v_n$ with value $c_n$.  
Now define $w_n$ by
$$
w_n= P_{\cal{ A}_\varphi^\perp}v_n=v_n - P_{\cal{ A}_\varphi}v_n.
$$
Since $\varphi$ sees  $v_n$ with value $ c_n$, it follows from Lemma \ref{3.8} that $\phi$ sees $w_n$ with value $c_n$, and so $w_n$ is a definable vector for $\phi$.  Note that 
$w_n\perp \cal{ A}_\varphi$ by definition of $w_n$. 
Furthermore, from equation (\ref{19}) we see that $v_n\to v$ as $ n \to \infty$. 
By assumption  $v\in \cal{ S}_\varphi \ominus \cal{ A}_\varphi$, and so $v\perp \cal{ A}_\varphi$.  Thus
$w_n \to P_{\cal{ A}_\varphi^\perp}v =v$, and for sufficiently large $n, w_n \not= 0$. 
Consequently, for sufficiently large $n, w_n\in \mathbf{ D}_\varphi$ and
$\varphi(w_n) = c_n$.  Since $w_n\to v$ and $ c_n \to \infty$
we conclude that $v$ is a polar vector.  Thus $\cal{ S}_\varphi \ominus
\cal{ A}_\varphi \subseteq \cal{ P}_\varphi$.
Therefore $\cal{ P}_\varphi$ is a closed
subspace of $\cal{ H}$ and 
$$\cal{ S}_\varphi=\cal{ P}_\varphi \oplus \cal{A}_\varphi.$$
\end{proof}

Theorem \ref{4.2} gives a concrete expression of the polar
space in terms of the singular and ambiguous spaces. 

\begin{example} \label{2kernels-phi} \rm
Let $\cal{ H}=H^2$, the classical Hardy space on
$\mathbb D$. Consider the function $\phi(z)=\frac{1}{z(z-\half)}$ defined on the set 
$\D\setminus\{0, \half\}$. For the pseudomultiplier $\phi(z)=\frac{1}{z(z-\half)}$ on $H^2$, we have $ D_\phi = \D \setminus \{0, \half \}$,  $E_\phi$ is the closed subspace $ z(z-\half) H^2$ of $H^2$, which has codimension $2$ in $H^2$,and
  $E_\phi^\perp= H^2\ominus z (z -\half)H^2$.
The operator $X_\phi:z(z-\half)H^2 \to H^2$ is given by
$X_\phi z(z-\half)f=f$ for all $f\in H^2$. 

Let us describe the space of ambiguous vectors $\cala_\phi$ for $\phi$.
By Proposition \ref{3.6},
\[
\cala_\phi = E_\phi^\perp \cap \ker X_\phi^*.
\]
Consider any $f\in \cala_\phi$: then, since $f\in E_\phi^\perp$, $f = c_1k_{0}+c_2 k_{\half}$ for some $c_1, c_2 \in\C$, and since $f\in\ker X_\phi^*$, $X_\phi^*f=0$. Thus, for every $g\in H^2$,
	\begin{eqnarray}\label{X*=0}
		0&=&\ip{X_\phi^* (c_1k_{0}+c_2 k_{\half})}{z(z -\half)g} \nonumber\\
		&=&\ip{c_1k_{0}}{X_\phi z(z-\half)g} + \ip{c_2 k_{\half}}{X_\phi z(z-\half)g}
		\nonumber\\
		&=& c_1 \ip{k_{0}}{g} + c_2 \ip{k_{\half}}{g}\nonumber\\
		&=& c_1 g(0) + c_2 g(\half).
	\end{eqnarray}
For $g(z)=z$, the equation \ref{X*=0} implies that $ c_2=0$ and, for $g(z)=z - \half$, the equation \ref{X*=0} implies that $ c_1=0$, thus $f=0$. Therefore 
 $\cala_\phi=\{0\}$.

Let us describe the set of definable vectors $\mathbf{ D}_\varphi$ of $\phi$. By Lemma \ref{3.4}, for 
 $f\in H^2$,  $\varphi$ sees $f$ with value $c$ if and only if there
exists $u\in E^\perp_\varphi$ such that
\begin{equation}\label{Definable-2}
X^*_\varphi f= \overline c f + u,
\end{equation}
equivalently,  if and only if there exist $a_0, a_1 \in\C$ such that
\[
X_\phi^*f = \bar c f+a_0 k_{0}+a_1 k_{\half}.
\]
Since  $\cala_\phi=\{0\}$, it is clear that for each $\overline c \not\in \sigma(X^*_\varphi)$,
\[
f = \left(X_\phi^* - \bar c I \right)^{-1} (a_0 k_{0}+a_1 k_{\half}),
\]
where $a_0, a_1 \in\C$ are not all zero, is a definable vector of  $\phi$.
To show that  the polar space $\cal{ P}_\varphi$ of $\varphi$ is $\span\{k_{0}, k_{\half}\}$, choose a sequence $\{c_n\} \subseteq \mathbb C$ with $c_n\to \infty$ and 
$\overline c_n \not\in \sigma(X^*_\varphi)$ for all $n$. Then,  for any $g \in  E_\phi^\perp=\span\{k_{0}, k_{\half}\}$, the sequence 
$$v_n=  \overline c_n \left(\overline c_n I - X^*_\varphi\right)^{-1} g$$
 is a sequence of definable vectors in $H^2$ for $\phi$, such that  $v_n\to g$ and $\phi(v_n)  \to \infty$ as $n  \to \infty$, and so  $g$ is a polar vector of $\varphi$. 
Therefore, 
$$\cal{ P}_\varphi = E_\phi^\perp = \span\{k_{0}, k_{\half}\}.$$
\qed
\end{example}

\begin{example}\label{ambandpole} A hybrid 2-dimensional singular space.  \rm  In this example we conjoin two previous examples of pseudomultipliers, $\phi(z)=1/z$ on $H^2(\D)$ and $\chi(t)=\sqrt{t}$ on $W^{1,2}[0,1]$ (see examples \ref{arch1} and \ref{intro2.3}) to exhibit a singular space that contains both polar vectors and ambiguous vectors. 
  Recall that $D_\phi= \D \setminus \{0 \}$ and $D_\chi =I$, where $I$ denotes the unit interval $I= [0,1]$.  Let $\D \sqcup I$ denote the disjoint union of the topological spaces $\D$ and $I$, and let 
  $\O$ be the topological quotient space of $\D \sqcup I$ obtained by the identification of the two points $0_\D$ and $0_I$ in $\D \sqcup I$ (we distinguish points of $\D \sqcup I$ which belong both to $\D$ and to $I$ by attaching an appropriate suffix).
  We shall denote by $\mathbf{0}$ the point of $\O$ obtained by the identification of the points $0_\D$ and $0_I$.
  Let $\calh$ be the Hilbert space of functions 
  $\vec{f}{g}:\O   \to\C$ defined by
  \[
 \vec{f}{g} (\omega) =\left\{ \begin{array}{lcl}   f(\omega) & \text{ if } & \omega=z\in \D \setminus \{0 \}  \\
  					 g(\omega)  &\text{ if } & \omega=t\in I,
  					\end{array}\right.
  \]
 where $f\in H^2(\D)$ and $g\in  W^{1,2}(I)$, equipped with its natural inner product and operations as a subspace of the Hilbert space direct sum $ H^2(\D) \oplus W^{1,2}(I)$. It is easy to see that $\calh$ is the closed subspace of $\calk=H^2(\D) \oplus W^{1,2}(I)$ comprising the elements $\vec{f}{g}\in\calk$ such that  $f(0_\D)=g(0_I)$. 
 
  Consider the function $\gamma:\O   \to\C$ defined by
  \[
  \gamma (\omega) =\left\{ \begin{array}{lcl}     \frac{1}{z} & \text{ if } & \omega=z\in \D \setminus \{0 \}  \\
  					\sqrt{t}  &\text{ if } & \omega=t\in I.
  					\end{array}\right.
  \]
\em We claim that $\gamma$ is a pseudomultiplier of $\calh$ on the set $\O$.  Furthermore the singular set $\cals_\gamma$ is $2$-dimensional and $\cals_\gamma$ is the orthogonal direct sum of the one-dimensional spaces $\calp_\gamma$ and $\cala_\gamma$.  

\rm To prove this statement let us first show that $E_\gamma$ has codimension $2$ in $\calh$.  
  For any $\vec{f}{g} \in\calh$, the function $\gamma\vec{f}{g}$ 
  is given by
  \[
\gamma\vec{f}{g} (\omega) =\left\{ \begin{array}{lcl}   \frac{1}{z} f(z) & \text{ if } & \omega=z\in \D \setminus \{0 \} \\
  					\sqrt{t} g(t)  &\text{ if } &  \omega=t\in I.
  					\end{array}\right. 
  \] 
 In particular, $\mathbf {0} \in D_\gamma$ and $\gamma(\mathbf {0})=0$.
 A vector $\vec{f}{g} \in \calh$ satisfies $\gamma\vec{f}{g}\in\calh$
 if and only if $\phi f$ extends to an element $h_1$ of $H^2$, $\chi g$ extends to an element $h_2$ of $W^{1,2}[0,1]$ and $h_1(0_\D)=h_2(0_I)$.
  
 \[
 E_{\gamma} = \left\{\vec{f}{g} \in\calh:\text{ there exists }\; h=\vec{h_1}{h_2}\in\calh \;\text{ such that }\;\gamma\vec{f}{g}=h|_{D_\gamma}\right\}.
 \]
 Hence, by examples \ref{arch1} and \ref{intro2.3}, $f =z \psi$ for some $\psi \in H^2$, $g\in k_0^\perp$ (here $k_0$ is the reproducing kernel of $0 \in I$ in $W^{1,2}(I)$) and
 \be\label{h1h2}
  \vec{h_1}{h_2} = \vec{\phi f}{\chi g} =  
 \vec{\psi}{\chi g}.
 \ee
 Note that, for $z \in \D$,  $f'(z)=(z\psi(z))' = \psi(z) + z \psi'(z)$, and so, by equation \eqref{h1h2}, since  $\vec{h_1}{h_2}\in\calh$,
 $$ f'(0)=\psi(0)= h_1(0)= h_2(0) = (\chi g)(0) =  0.$$
 Therefore
 \be\label{Ephichi}
 E_{\gamma}= \left\{\vec{f}{g} \in\calh: f\in zH^2, g\in k_0^\perp \text{ and } f'(0)=0\right\}.
 \ee
 Let us present $E_{\gamma}$ using reproducing kernels of $\calh$.
 We shall denote by $s_z$ the reproducing kernel of the point $z\in\D$ in the space $H^2$, and by $k_x$ the reproducing kernel of $x\in I$ in $W^{1,2}(I)$. We also need the kernel $s^1_z\in H^2$ which reproduces the derivative at $z$, in the sense that 
  \[
  \ip{f}{s^1_z}_{H^2} = f'(z) \text{ for all } f\in H^2.
  \]
  Explicitly,
  \[
  s^1_z(\lam)= \frac{\lam}{(1-\bar z\lam)^2}.\] 
  Let $P_\calh$ denote the orthogonal projection operator from $\calk$
  onto its subspace $\calh$. Observe that $\calh$ is the orthogonal complement in 
  $\calk$ of the vector $v=\vec{s_0}{-k_0}$, and therefore, for every $x \in \calh$,
   \be \label{projH}
   P_\calh (x) = x-  \frac{ \ip{x}{v}v}{ \norm{v}^2} = x-  \frac{ \ip{x}{v}v}{ 1+k_0(0)}.
  \ee
  It is easy to see that the reproducing kernel of the point $z\in\D$ in $\calh$ is $P_\calh \vec{s_z}{0}$.  Similarly, the reproducing kernel in $\calh$ of the point $x\in I$ is $P_\calh\vec{0}{k_x}$. 
  Formula \eqref{projH} for $P_{\calh}$ shows that the reproducing kernel $k_{\mathbf 0}$ of the point ${\mathbf 0} \in\O$ in the space $\calh$ is given by the formula
 \be\label{twoequal}
 k_{\mathbf 0} = P_{\calh} \vec{s_0}{0} =P_{\calh} \vec{0}{k_0} = \frac{1}{1+k_0(0)} \vec{k_0(0) s_0}{k_0} \; \text{and} \; P_{\calh} \vec{s^1_0}{0} =\vec{z}{0}.
 \ee
 
 Re-write equation \eqref{Ephichi} in the form
 \be\label{Ephichi-2}
  E_{\gamma}=\left\{\vec{f}{g}\in\calh: \vec{f}{g} \perp \span \left\{P_{\calh}\vec{s_0}{0},P_{\calh}\vec{s^1_0}{0},P_{\calh}\vec{0}{k_0} \right\}\right\}.
  \ee
 Thus, in view of equations \eqref{twoequal}, 
 \[
 E_\gamma=\left\{\vec{k_0(0)s_0}{k_0}, \vec{z}{0}  \right\}^\perp,
 \]
 and so $E_{\gamma}$ has codimension $2$ in $\calh$. Therefore  $\gamma$ is a pseudomultiplier of $\calh$.
By equations \eqref{Ephichi-2} and \eqref{twoequal}
\be\label{Ephichi-3}
\cal{ S}_\gamma=E_\gamma^\perp= \span \left\{P_{\calh}\vec{s_0}{0},P_{\calh}\vec{s^1_0}{0},P_{\calh}\vec{0}{k_0} \right\} = \span \left\{ \vec{k_0(0) s_0}{k_0}, \vec{z}{0}    \right\}.
\ee 

Note that 
\[
\ip{ \vec{k_0(0) s_0}{k_0}}{ \vec{z}{0}}_\calk = 
\ip{k_0(0) s_0}{z}_{H^2} +  \ip{k_0}{0}_{W^{1,2}} = 0,
\]
and so equation \eqref{Ephichi-3} gives an {\em orthogonal} basis for $\cals_\gamma$.  By Proposition \ref{3.6}, 
the ambiguous vectors for $\gamma$ are the members of $E_\gamma^\perp\cap \ker X_\gamma^*$. 
Consider any $\vec{f}{g}\in \cala_\gamma$: then $\vec{f}{g} \in E_\gamma^\perp$, and so 
$\vec{f}{g}=c_1  \vec{k_0(0) s_0}{k_0} +c_2 \vec{z}{0} $ for some $c_1, c_2 \in\C$, and $X_\gamma^*\vec{f}{g}=0$. Recall that $X_\gamma^*$ acts from $\calh$ to $E_\gamma$.
	Thus, for every element $h=\vec{h_1}{h_2} \in E_\gamma$

	\begin{eqnarray*}
		0&=&\ip{X_\gamma^*\vec{f}{g}}{h}\\
		&=& \ip{\vec{f}{g}}{X_\gamma h}.
	\end{eqnarray*}
In particular, by equation \eqref{Ephichi} we may choose $h_1(z)=z^2$ and $h_2\in k_0^\perp$, and so 
	\begin{eqnarray*}
	0&=&\ip{c_1  \vec{k_0(0) s_0}{k_0} +c_2 \vec{z}{0}}
		{X_\gamma \vec{z^2 }{h_2} }\\
    &=&\ip{c_1  \vec{k_0(0) s_0}{k_0} +c_2 \vec{z}{0}}
		{\vec{z }{\chi h_2} }\\
	&=& c_1 \ip{  \vec{k_0(0) s_0}{k_0}}{\vec{z }{\chi h_2}} +c_2 \ip{ \vec{z}{0}}{\vec{z }{\chi h_2} }\\	
	&=& c_1 \left( \ip{k_0(0) s_0}{z } +\ip{k_0}{\chi h_2}  \right) +
	c_2 \left( \ip{z}{z } + \ip{0}{\chi h_2}  \right) \\
&=& c_1 k_0(0)\overline{z(0)} +c_1 \overline{\chi(0) h_2(0)} +
	c_2 \ip{z}{z } \\	
	&=& c_2 \ip{z}{z} = c_2,
	\end{eqnarray*}
and so $c_2=0$.  Hence we have proved that
\[
\cala_\gamma \subseteq \C \vec{k_0(0) s_0}{k_0} = \C P_\calh \vec{0}{k_0}.
\]
Let us show that the vector $v=\vec{k_0(0) s_0}{k_0} $ belongs to $\ker X_\gam^*$.
By equation \ref{Ephichi},
\[
 E_{\gamma}= \left\{\vec{f}{g} \in\calh: f\in zH^2, g\in k_0^\perp \text{ and } f'(0)=0\right\}.
 \]
Thus, for any $u\in E_\gam$, say $u=\vec{f}{g}$, where $f=z\psi$ for some $\psi\in H^2$,
\[
\ip{X_\gam^*v}{u}=\ip{v}{\gam u}=\ip{\vec{k_0(0)s_0}{k_0}}{\vec{\psi}{\chi g}} = k_0(0)\ip{s_0}{\psi} + \ip{k_0}{\chi g}=k_0(0)\overline{\psi(0)} + 0.
\]
Since $f=z\psi$ where $\psi\in H^2$, $f'(z)=\psi(z)+z\psi'(z)$ for all $z\in \D $. Thus the equation $f'(0)=0$ implies  $\psi(0)=0$, and therefore $\ip{X_\gam^*v}{u} = 0$.  Hence $v\in \ker X_\gam^* \cap E_\gam^\perp = \cala_\gam $.
We conclude that \[
\cala_\gamma = \C \vec{k_0(0) s_0}{k_0}.
\]
The singular space of $\gamma$ contains both polar vectors and ambiguous vectors.
$\mathbf{0} \in \O$ is an ambiguous point of $\gamma$.
By Theorem \ref{4.2}, $\cal{ P}_\gamma$ is a closed
subspace of $\cal{ H}$ and 
$$
\cal{ P}_\gamma  = \cal{S}_\gamma \ominus \cal{A}_\gamma = \span \left\{ \vec{k_0(0) s_0}{k_0}, \vec{z}{0}\right\} \ominus \C \vec{k_0(0) s_0}{k_0} = \C\vec{z}{0}.
$$
\qed
\end{example}

\section{Visible subspaces of a pseudomultiplier}\label{localness}

The notion of a pseudomultiplier seeing a vector has proved
fruitful.  We shall now take a slightly different approach to
that notion which will lead  to a generalization, the idea
of a pseudomultiplier {\it seeing a subspace}.

An obvious way to generalize Definition \ref{1.7}, of $\phi$ seeing a vector $v$ with value $c$, would be
to formalize the situation where knowledge of $n$
moments of $f, \{\ip{f}{v_1},\ \ip{f}{v_2}, \dots,  \ip{f}{v_n}\}$,
implies that the corresponding moments of $X_\varphi f,
\{\ip{X_\varphi f}{v_1}, \dots, \ip{X_\varphi f} {v_n}\}$, are
determined.  Knowing the moments $\{\ip{f}{v_1}, \dots, \ip{f}
{v_n}\}$ is equivalent to  knowing $P_\cal{ V}f$ where
$P_\cal{ V}$ denotes orthogonal projection on the space $\cal{V}=\span \{v_1, \dots , v_n\}$.  Further, since
$X_\varphi$ is linear, the only way  the moments of
$X_\varphi f$ can be determined by the moments of $f$ is
linearly.  The following definition is therefore a natural
generalization of Definition \ref{1.7}.

\begin{definition}\label{5.3}
{\rm Let $\varphi$ be a pseudomultiplier and let $\cal V$
be a closed subspace of $\cal H$.  We say that $\phi$ {\em sees} $\cal V$ if there is a bounded linear operator 
$C:\cal{ V}\to \cal{ V}$ such that, for all $f\in E_\varphi$,
\begin{equation}\label{22}
P_\cal{ V} X_\varphi f = C P_\cal{ V} f,
\end{equation}
where $P_\calv$ is the orthogonal projection of $\calh$ onto $\calv$.
We say that $\cal V$ is a {\it regular space for } $\varphi$
if $\varphi$ sees $\cal V$ and there is a {\it unique} $C$
such that equation (\ref{22}) holds for  all $f\in E_{\varphi}$; when
this is so we denote $C$ by $\varphi(\cal{ V})$ and call it
the {\it value of} $\varphi$ at $\cal V$.  Alternatively we 
say that $\varphi$ {\it sees} $\cal V$ {\it with value}
$C$.} 
\end{definition}

\begin{proposition}\label{5.5-mult}  Let $\varphi$ be a
multiplier of $\cal{ H}$ and let $\cal{ V} \subseteq \cal{ H}$ be a closed
subspace.  The following two statements are equivalent.
\begin{enumerate}
\item $\phi$ sees $\cal V$ 
\item  $X^*_\varphi \cal{ V} \subseteq \cal{ V}$, that is, 
$\calv$ is a closed invariant subspace for $X^*_\varphi$.
\end{enumerate}
Moreover, every closed invariant subspace for $X^*_\varphi$ is regular for $\varphi$
with value $C= P_\calv X_\phi|_\calv$.
\end{proposition}
\begin{proof} 
Since $\phi$ is a multiplier of $\calh$, $E_\phi=\calh$.\\
(1) $\Longrightarrow$ (2).
  Let $\phi$ see $\cal V$ with value $C:\cal V \to \cal V$.  By definition $C$ is bounded and
$P_{\cal V} X_\phi =CP_{\cal V}$ on $\cal{ H}$. Thus $ X_\phi^* P_{\cal V}^* =P_{\cal V}^* C^*$.
Hence $X_\phi^* \cal V \subseteq \cal V$, as claimed.

(2) $\Longrightarrow$ (1).
Suppose that $X_\phi^* \cal V \subseteq \cal V$.  We must find a bounded linear operator $C:\cal V\to \cal V$ such that $P_{\cal V} X_\phi =CP_{\cal V}$ on $\cal{ H}$.
Let $C= P_\calv X_\phi|_\calv$.  We need to check that 
\[
P_\calv X_\phi f= P_\calv X_\phi P_\calv^* P_\calv f
\]
for all $f\in\calh$, that is, that, for all $f\in\calh$,
\be\label{xphiv}
X_\phi(1- P_\calv^* P_\calv) f \in \calv^\perp.
\ee
The last equation is equivalent to saying that, for all $v\in\calv$ and $f\in\calh$,
\be\label{xphiv2}
\ip{(1- P_\calv^* P_\calv) f}{X_\phi^*v} =0.
\ee
Since, by assumption $X_\phi^*v\in\calv$ and since $1- P_\calv^* P_\calv$ is the orthogonal projection on $\calv^\perp$, equation \eqref{xphiv2} is true.
Hence $P_{\cal V} X_\phi =CP_{\cal V}$ on $\cal{ H}$, as required.
\end{proof}

\begin{remark} \rm 
By Proposition \ref{multipliers}, for any multiplier $\phi$ of a Hilbert function space $\calh$, the vectors which are visible to $\phi$ are precisely the eigenvectors of $X_\phi^*$.  Proposition \ref{5.5-mult} shows that a closed subspace $\calv$ of $\calh$ is visible to $\phi$ if and only if $\calv$ is an invariant subspace for $X_\phi^*$.  From these facts it is apparent that even if $\phi$ sees a closed subspace $\calv$ of $\calh$, it does not follow that $\phi$ sees the individual vectors in $\calv$. 

For example, for $\varphi \in H^\infty$,  $\varphi$ is a multiplier of $H^2$. 		 The only vectors in $H^2$ seen by
		$\varphi$ are the eigenvectors of $X^*_\varphi$.  In
		particular, if $\varphi(z)=a_0+a_1z$ for some $a_0,a_1\in\C$, $a_1\neq 0$, then $X_\phi^*$ is the Toeplitz operator $T_{\bar\phi}$, and if
	$v(z) = z, \text{ for } z	\in \mathbb T$ then, for $z\in\T$,
	\[
	X_\phi^*v(z)= P_+ \bar\phi(z) z = P_+((\ov{a_0+a_1z})z)=P_+(\ov{a_0} z +\ov{a_1} |z|^2) = \ov{a_0} z +\ov{a_1}= \ov{a_0}v+\ov{a_1},
	\]
	which is not a scalar multiple of $v$.  Thus $v$ is not an eigenvector of $X_\phi^*$,
	so that 
		$\varphi$ does not see the  vector		$v$.  However, $\calv=\span \{1, v\}$ is an invariant subspace for $X_\phi^*$, since $X_\phi^* 1= \ov{a_0},$ and so 
		$X_\phi^* \span \{1, v\} = \span \{1, v\}$. Therefore	
		$\varphi$ does see the space		$\cal{ V} \stackrel{\rm def}{=} \span \{1, v\}$.\qed
\end{remark}

\begin{example}\label{shift} \rm
Consider the multiplier $\phi(z)=z$ on $H^2$.  By Proposition \ref{5.5-mult},
the closed subspaces $\calv$ of $H^2$ visible to $\phi$ are precisely the closed invariant subspaces for $X_\phi^*$.  Here $X_\phi$ is the unilateral shift $S$ on $H^2$. By Beurling's Theorem, the closed invariant subspaces for $S^*$ are the spaces of the form $H^2 \ominus \theta H^2$, where $\theta$ is an inner function.
\end{example}

\begin{example}\label{5.4}
{\rm Let $\varphi \in H^\infty$, so that $\varphi$ is a
multiplier of $H^2$.  Let
$\alpha \in \mathbb D$ and let $k^{(r)}_\alpha$ denote the
kernel which  reproduces $f^{(r)}(\alpha)$, that is,
$$
\ip{f} {k^{(r)}_\alpha}\ =\ f^{(r)} (\alpha)
$$ 
for all $f\in
H^2$. Then $\varphi$ sees $\cal{ V}\stackrel{\rm def}{=} \span
\{k_\alpha, k^{(1)}_\alpha, \dots , k^{(n)}_\alpha\}$ for any
$n\in \mathbb N$.  This is because $P_\cal{ V} f$ depends
only on the values of $f(\alpha), f'(\alpha), \dots, f^{(n)} 
(\alpha)$, and Leibniz' rule shows that $P_\cal{V}(\varphi f)$ can be calculated once these values are
known for $f$.  
By Proposition \ref{5.5-mult}, there is a {\it unique} bounded linear operator
$C:  \cal{ V} \to  \cal{ V}$, given by $C= P_\calv X_\phi P_\calv^*$,
such that $P_\cal{ V}(\varphi f) =C P_\cal{ V} f$, and so $\cal V$ is a regular space for $\varphi$.\qed}
\end{example}

\begin{lemma} \label{prep5.5} Let $\varphi$ be a
pseudomultiplier of $\calh$ and let $\cal{ V}$ be a closed subspace of $\cal{ H}$. 
If $\phi$ sees $\cal V$  then
$X^*_\varphi \cal{ V} \subseteq \cal{ V} + E^\perp_\varphi$.
\end{lemma}
\begin{proof}
 Let $\phi$ see $\cal V$ with value $C:\cal V \to \cal V$.  By definition $C$ is bounded and
$P_{\cal V} X_\phi =CP_{\cal V}$ on $E_\phi$.  For all $v\in \cal V$ and $f\in E_\phi$ we have
\[
\ip{X_\phi f}{v}=\ip{P_{\cal V}X_\phi f}{v}=\ip{CP_{\cal V}f}{v} =\ip{f}{C^*v},
\]
and so
\[
\ip{f}{(X_\phi^*-C^*)v}=0 \text{ for all } f\in E_\phi.
\]
That is,
\[
(X_\phi^* - C^*)v \in E_\phi^\perp \text{ for all }v\in \calv.
\]
Hence $X_\phi^* \cal V \subseteq \cal V + E_\phi^\perp$, as claimed.
\end{proof}

\begin{lemma}\label{5.5PVclosed}  Let $\calv$ be a closed subspace of a Hilbert space
 $\calh$ and let $E$ be a closed finite-codimensional subspace of $\cal{ H}$.  The orthogonal projection of $E$ onto $\calv$ is a closed subspace of $\calh$.
\end{lemma}
\begin{proof}
Let $P_\calv:\calh\to\calv$ be the operation of orthogonal projection, so that $P_\calv^*:\calv\to\calh$ is the inclusion operator and $P$, defined to be $P_\calv^*P_\calv$, is the Hermitian operator in $\calb(\calh)$ with range $\calv$.
Let $M=E^\perp + PE^\perp$.  $M$ is a finite-dimensional subspace of $\calh$ and is invariant under $P$.  Therefore, since $P$ is Hermitian, $M$ is a reducing subspace for $P$, and the restrictions $P|M$ and $P|M^\perp$ are also Hermitian projections.

Since $E^\perp\subseteq M, M^\perp\subseteq E$.  Let $M_0= E\cap M$.  We have $E=M^\perp + \tilde M$, where $\tilde M$ is the orthogonal complement of $M^\perp$ in $E$, so that $\tilde M =\{x\in E:x\perp M^\perp\}=E\cap M =M_0$.  Hence $E=M_0+M^\perp$, and therefore $PE=PM_0+ PM^\perp$.  Now $PM^\perp$ is the range of $P|M^\perp$, which is a Hermitian projection, and therefore $PM^\perp$ is closed in $\calh$.  $PM_0$, on the other hand, has finite dimension, and so, by \cite[Proposition III.4.3]{Con},  $PE=PM_0+ PM^\perp$ is a closed set in $\calh$
as the sum of the closed subspace $PM^\perp$ of $\calh$ and the finite-dimensional subspace $PM_0$ of $\calh$.
\end{proof}

The referee pointed out that Lemma \ref{5.5PVclosed} is a special case of the general result that the range of a bounded linear
operator is closed if it is finite co-dimensional, a fact well known to experts in the theory of Fredholm operators. Indeed, apply this fact to the operator $P_\calv : E \to \calv$.

The following proposition generalizes Proposition \ref{5.5-mult} to a pseudomultiplier.

\begin{proposition}\label{5.5}  Let $\varphi$ be a
pseudomultiplier of $\calh$ and let $\cal{ V}$ be a closed subspace of $\cal{ H}$.
\begin{enumerate}

\item[\rm (i)] $\phi$ sees $\cal V$  if and only if
$X^*_\varphi \cal{ V} \subseteq \cal{ V} + E^\perp_\varphi$.

\item[\rm (ii)]  $\cal V$ is regular if and only if $ X_\phi^*\calv \subseteq \calv + E_\phi^\perp$ and $\cal{ V} \cap
E^\perp_\varphi = \{0\}$.
\end{enumerate}
\end{proposition}

\begin{proof} 
(i) By Lemma \ref{prep5.5},  $\phi$ sees $\cal V$  implies that
$X^*_\varphi \cal{ V} \subseteq \cal{ V} + E^\perp_\varphi$.

Conversely, suppose that $X_\phi^* \cal V \subseteq \cal V + E_\phi^\perp$. 
By Lemma \ref{5.5PVclosed},  the projection $P_\calv E_\phi$ of $E_\phi$ on $\calv$ is closed in $\calv$. 
We must find a bounded linear operator $C:\cal V\to \cal V$ such that $P_{\cal V} X_\phi =CP_{\cal V}$ on $E_\phi$, equivalently, we have to show that there exists 
 a bounded linear operator $C^*:\cal V\to \cal V$ such that 
 $$P_{E_\phi}  X_\phi^* P_{\cal V}^*=P_{E_\phi} P_{\cal V}^*C^*.$$
Since $X^*_\varphi \cal{ V} \subseteq \cal{ V} + E^\perp_\varphi$, we obtain
$$
P_{E_\phi}  X_\phi^* P_{\cal V}^* (\cal{ V}) \subseteq P_{E_\phi} ( \cal V+ E^\perp_\varphi) =  P_{E_\phi} P_{\cal V}^* (\cal{ V}).
$$
Hence, for the bounded linear operators
$$P_{E_\phi}  X_\phi^* P_{\cal V}^*: \cal V \to E_\phi \; \text{and} \; P_{E_\phi} P_{\cal V}^**: \cal V\to E_\phi,$$
the inclusion  $\text{ran} (P_{E_\phi}  X_\phi^* P_{\cal V}^*) \subseteq \text{ran} (P_{E_\phi} P_{\cal V}^*)$ holds.
By the generalized Douglas factorization lemma \cite[Theorem 1.1]{arias} (\cite{douglas}), there exists a bounded linear operator $C^*:\cal V\to \cal V$ such that $P_{E_\phi}  X_\phi^* P_{\cal V}^*=P_{E_\phi} P_{\cal V}^*C^*.$
Therefore there exists a bounded linear operator $C:\cal V\to \cal V$ such that
 $ CP_\calv f = P_\calv X_\phi f$ for all $f \in E_\phi$, and so $\phi$ sees $\calv$, as claimed.

(ii)  Recall that $\calv$ is regular for $\phi$ if $\phi$ sees $\calv$ with the {\em unique} value
$C:\cal V\to \cal V$, which is a bounded linear operator such that
 $ CP_\calv f = P_\calv X_\phi f$ for all $f \in E_\phi$.

Suppose that $\calv$ is regular for $\phi$. By Part (i) $ X_\phi^*\calv \subseteq \calv + E_\phi^\perp$. We need to show that  $\cal{ V} \cap
E^\perp_\varphi = \{0\}$.
 
 By Lemma \ref{5.5PVclosed},  the projection $P_\calv E_\phi$ of $E_\phi$ on $\calv$ is closed in $\calv$. Consider the unique
bounded linear operator $C:\cal V\to \cal V$ such that
 $ CP_\calv f = P_\calv X_\phi f$ for all $f \in E_\phi$.  Let $C_0: P_\calv E_\phi \to \calv$ be the restriction of $C$ to the closed subspace $P_\calv E_\phi$ of $\calv$. Therefore $C_0$ is a bounded operator on $P_\calv E_\phi$.

For any bounded linear operator $C$ on $\calv$, $\phi$ sees $\calv$ with value $C$ if and only if $C$ is an extension of the operator $C_0: P_\calv E_\phi\to\calv$.  Since $C_0$ is a bounded operator on the closed subspace $P_\calv E_\phi$ of $\calv$, $C_0$ has a unique extension to $\calv$ if and only if $P_\calv E_\phi = \calv$.  We assert that $P_\calv E_\phi = \calv$ if and only if $\calv \cap E_\phi^\perp = \{0\}$.
Indeed, suppose that $P_\calv E_\phi = \calv$ but $\calv \cap E_\phi^\perp \neq \{0\}$.
Then there exists $v\in \calv \cap E_\phi^\perp$ such that $v\neq 0$.  By assumption there exists $f\in E_\phi$ such that $P_\calv f= v$. Then
\[
\|v\|^2= \ip{v}{v}=\ip{v}{P_\calv f}=\ip{v}{f}=0,
\]
and so $v=0$, a contradiction.  Thus $\calv \cap E_\phi^\perp = \{0\}$.

Conversely, suppose $\calv \cap E_\phi^\perp = \{0\}$ but $\calv\ominus P_\calv E_\phi \neq \{0\}$.  Pick any non-zero $v \in \calv\ominus P_\calv E_\phi$.
For every $f\in E_\phi$ we have $0= \ip{v}{P_\calv f}=\ip{v}{f}$, and so $v\in \calv\cap E_\phi^\perp$, which contradicts the hypothesis $\calv \cap E_\phi^\perp = \{0\}$.
Thus $P_\calv E_\phi = \calv$ if and only if $\calv \cap E_\phi^\perp = \{0\}$, as asserted, and so $\calv$ is regular for $\phi$ if and only if $X_\phi^*\calv \subseteq \calv + E_\phi^\perp$ and $\calv \cap E_\phi^\perp = \{0\}$, which is statement (ii).
\end{proof}

The next corollary shows the connection between a pseudomultiplier $\phi$ seeing a vector, and $\phi$ seeing a closed subspace.

\begin{corollary}\label{phi-sees-vector-space} \rm
Let $\phi$ be a pseudomultiplier of a Hilbert function space $\calh$ and let $v \in \calh$. Then 
 $\phi $ sees $v$ if and only if $\phi $ sees the closed subspace $\calv= \C v$ in $\calh$.
\end{corollary} 
 \begin{proof} Suppose $\phi $ sees $v$ with value $c$. 
By Lemma \ref{3.4} (i), there
exists $u \in E^\perp_\varphi$ such that
\begin{equation}\label{9.2}
X^*_\varphi v= \overline c v + u.
\end{equation}
Then,  for any $\lambda \in \C$,
 $\varphi$ sees $\lambda v$ with value $\overline{\lambda} c$ and there
exists $\lambda u \in E^\perp_\varphi$ such that
\begin{equation}\label{9.2a}
X^*_\varphi \lambda v= \overline c \lambda v + \lambda u.
\end{equation}
Therefore 
\[
X^*_\varphi \cal{V} \subseteq \cal{V} + E^\perp_\varphi.
\]
Hence, by  Proposition \ref{5.5} (i), $\phi$ sees $\cal{V}$.

Conversely, suppose $\phi$ sees the closed subspace $\cal{V}= \C v$ in $\calh$.
By Proposition  \ref{5.5} (i), 
$ X^*_\varphi \cal{V} \subseteq \cal{V} + E^\perp_\varphi $.
Therefore there exist $d \in \C$ and $u \in E^\perp_\varphi $, such that
 $X^*_\varphi v = d v + u$. Hence, by Lemma \ref{3.4} (i),  $\varphi$ sees $v$ with value $\bar{d}$. \qed
\end{proof}

\begin{corollary}\label{v1+v2}
Let $\varphi$ be a
pseudomultiplier of $\calh$ and let $\calv_1, \calv_2$ be closed subspaces of $\cal{ H}$ which are visible to $\phi$.  Then $\phi$ sees the closed span of $\calv_1$ and $\calv_2$.
\end{corollary}
\begin{proof}
By Proposition \ref{5.5}, 
 $\phi$ sees $\cal V$  if and only if
$X^*_\varphi \cal{ V} \subseteq \cal{ V} + E^\perp_\varphi$.
Thus $X_\phi^*\calv_i \subseteq \calv_i+E_\phi^\perp$ for $i=1,2$, and therefore
$X_\phi^*(\calv_1+\calv_2) \subseteq \calv_1+\calv_2 +E_\phi^\perp$. 
Let us denote the closure of a set $\calv$ in $\calh$ by $\ov{\calv}$.
Thus, since $E_\phi^\perp$ is finite dimensional and therefore closed, $X_\phi^*(\ov{\calv_1+\calv_2}) \subseteq \ov{\calv_1+\calv_2} +E_\phi^\perp$.
Hence $\phi$ sees $\ov{\calv_1+\calv_2}$.
\end{proof}

If a pseudomultiplier $\phi$ sees vectors $v_1$ and $v_2$, it does not follow that $\phi$ sees their sum, as we showed in Example \ref{1.9int} for the pseudomultiplier $1/z$ of $H^2$.  However, if $\phi$ sees closed subspaces $\calv_1$ and $\calv_2$, then $\phi$ does see
 $\ov{\calv_1+\calv_2}$.

\begin{example}\label{viorthogonal} \rm
Let $\phi$ be a pseudomultiplier of a Hilbert function space $\calh$ and suppose $\calv=\span\{v_1, \dots,v_n\}$ where each $v_i \in\calh$ and  $\phi $ sees $v_i$ with value $c_i$ for each $i\geq 1$.  Then $\phi$ sees $\calv$. To see this assertion, we note that since $\calv$ is closed, we may apply Proposition \ref{5.5}, and so
$\phi$ sees $\cal V$  if and only if
$X^*_\varphi \cal{ V} \subseteq \cal{ V} + E^\perp_\varphi$.
By Lemma \ref{3.4}, for all $i$, $X_\phi^*v_i=\ov{c_i}v_i + u_i$ for some $u_i \in E_\phi^\perp$.
Thus  $X^*_\varphi \cal{ V} \subseteq \cal{ V} + E^\perp_\varphi$.
By Proposition \ref{5.5}, $\phi$ sees $\calv$. \qed
\end{example}

\begin{example}\label{shift-pseudo} \rm
Consider the pseudomultiplier $\phi(z)=1/z$ on $H^2$.  Let $\calv$ be a closed subspace of $H^2$.
By Lemma \ref{prep5.5},
$\phi$ sees $\cal V$  implies that
$X_\phi^* \cal V \subseteq \cal V + E_\phi^\perp$.
 Here $E_\phi=zH^2$ and so $E_\phi^\perp$ is the space $\C k_0$ of constant functions on $\D$. One can check that
$$X_\varphi = S^*P_{zH^2}^*, \; \text{and so} \;
X^*_\varphi= P_{zH^2} S  =S: H^2\to zH^2,$$
where $S$ is the unilateral shift on $H^2$ and $P_{zH^2}:H^2 \to zH^2$ acts by orthogonal projection. By Beurling's Theorem, the closed
invariant subspaces for $S$ are the spaces of the form $\theta H^2$, where $\theta$ is an inner function.
Therefore, for all inner functions $\theta$, if $\calv =\theta H^2$, then
\[
X_\phi^* (\cal V) = S (\calv) \subseteq \calv \subseteq \cal V +\C k_0.
\]
Note that $\calv\cap E_\phi^\perp = \theta H^2\cap \C k_0$, which is $\{0\}$ if $\theta$ is non-constant and so, in this case $P_\calv E_\phi = \calv$.  If $\theta = 1$ then  $\calv=H^2$, and so $P_\calv E_\phi = E_\phi$. 
By Proposition \ref{5.5}, $\phi$ sees all spaces of the form $\calv=\theta H^2$, where $\theta$ is an inner function. 

Another way to construct closed
visible subspaces for $\phi$ is the following.  We have shown that $\phi$ sees $\calv$ if and only if $SV \subseteq \calv + \C k_0$, which is equivalent to the inclusion $S^*(\calv+\C k_0)^\perp \subseteq \calv^\perp$.  Let us show that $\calv = (\theta H^2)^\perp$, where $\theta$ is an inner function such that $\theta(0)\neq 0$, is visible to $\phi$.  Note that $(\calv + \C k_0)^\perp = \calv^\perp \cap k_0^\perp = \calv^\perp\cap zH^2$.  Consider any vector $g \in \calv^\perp\cap zH^2$.  We can write $g(z) = z \theta(z)h(z)$ for some $h\in H^2$.  Thus $S^*g=\theta h \in \calv^\perp$.  Therefore $S^*(\calv+\C k_0)^\perp \subseteq \calv^\perp$, and so $\phi$ sees $\calv = (\theta H^2)^\perp$, whenever $\theta$ is an inner function such that $\theta(0)\neq 0$.

We have described two different kinds of visible subspaces for the pseudomultiplier $\phi$ of $H^2$. The first kind of visible subspaces comprises  the invariant subpaces  of the shift operator $S$. The second consists of the invariant subpaces  of the adjoint of $S$ which are contained in $z H^2$.  By Corollary \ref{v1+v2}, $\phi$ also sees the closed span of a pair of visible subspaces of the two different types.

It would be interesting to describe {\em all} closed visible subspaces of the pseudomultiplier $\phi$ on $H^2$.\qed
\end{example}

\section{Local subspaces over a subset of $\Omega$}\label{singularities}

Let $\cal H$ be a Hilbert function space on a set $\Omega$. 
For $\alpha, \beta \in \Omega$ define $d(\alpha, \beta) \in
\mathbb R^+$ by 
\begin{equation}\label{23}
d(\alpha, \beta) = \|k_\alpha - k_\beta\|.
\end{equation}
Then $d$ is a pseudometric on $\Omega$, that is, it
satisfies the axioms of a metric except for the strict
positivity of $d$ when $\alpha \not= \beta$.  Clearly $d$
is a metric on
$\Omega$ if and only if $\alpha, \beta \in \Omega$ and $\alpha
\not= \beta$ imply $k_\alpha \not= k_\beta$, alternatively,
if and only if the elements of $\cal H$ separate the points of
$\Omega$.

\begin{definition}\label{6.1}
{\rm Let $\cal H$ be a Hilbert function space on a set
$\Omega$.   We shall say that $\Omega$ is {\it complete}
(for $\cal H$) if
$\Omega$ is complete with respect to the pseudometric $d$
defined by (\ref{23}) (that is, if every sequence in
$\Omega$ which is Cauchy  with respect to $d$ has a limit
point in
$\Omega$ with respect to $d$).}
\end{definition}

Thus, $\Omega$ is complete for $\cal H$ if and only if the set
of kernels $\{k_\lambda: \lambda \in \Omega\}$ is a closed 
subset of $\cal H$.  It is reasonable to restrict attention to
pseudomultipliers of Hilbert function spaces on complete
sets.  The ``natural" Hilbert function spaces do have complete
$\Omega$'s, and in any case there is an obvious process of
completion -- regard $\cal H$ as a function space on the
closure of $\{k_\lambda: \lambda \in \Omega\}$ in $\cal H$:
pseudomultipliers extend in a natural way.

For any finite set $F\subset \Omega$ we shall denote by
$\#F$ the number of elements of $F$ and by $\cal{ V}_F$ the
span of $\{k_\alpha: \alpha \in F\}$ in $\cal H$.  Notice that,
for any finite $F\subset \Omega$, there exists a subset
$F_0$ of $F$ such that $\cal{ V}_{F_0}=\cal{ V}_F$ and
$\#F_0=\dim \cal{ V}_F$.

\begin{definition}\label{6.2}
{\rm Let $F$ be a finite subset of
$\Omega$.  A subspace $\cal{ M}$ of $\cal H$ {\em lies over} $F$
if $\cal{M} \subset \cal{V}_F$.  We say that a subspace $\cal{ M}$ of $\cal H$ 
{\em lies over} $F$
{\em tautly} if in addition $\cal{ M} = \cal{ V}_{F_0}$ for
some subset $F_0 $ of $F$.}
\end{definition}

Thus, in $H^2$, if $\alpha \in \mathbb D\smallsetminus\{0\},
\span\{k_0+k_\alpha\}$ lies over $\{0, \alpha\}$, but does
not lie tautly over $\{0, \alpha\}$.  Note that $\cal{ M}$ lies
over $F$ tautly if and only if there exists $F_0 \subset F$
with $\#F_0=\dim \cal{ M}$ and $\cal{ M}= \cal{ V}_{F_0}$. 
The property of tautness  connects with the visibility of
subspaces, as the following proposition shows.

\begin{proposition}\label{6.3}
Let $\varphi$ be a pseudomultiplier and let $F$ be a finite
subset of $D_\varphi$.  If $\cal{ M}$ is a subspace of
$\cal{ H}$ which lies over $F$ tautly then $\varphi$ sees
$\cal{ M}$.
\end{proposition}

\begin{proof} This is immediate from the facts that
$\varphi$ sees $k_\alpha$ for $\alpha \in D_\varphi$,
and hence sees $\cal{ V}_{F_0}$ for $F_0 \subset
D_\varphi$, while $\cal{ M}=\cal{ V}_{F_0}$ for some
$F_0 \subset F$. \end{proof}
\vspace{2mm}

There  is no converse: if one knows that a pseudomultiplier
$\varphi$ sees a space $\cal{ M}$ which lies over $F\subset
\Omega$, one cannot conclude that $\cal{ M}$ is taut over
$F$.  Indeed, if $\varphi$ happens to be a constant function,
then $\varphi$ sees {\it every} subspace.  However, the
above example, $\cal{ M}=\span\{k_0+k_\alpha\} \subset
H^2$ is likely to be typical.  If $\varphi \in H^\infty$ and
$\varphi(0)\not= \varphi(\alpha)$, and if $\cal{ M}=\mathbb
Cu$ is a one-dimensional space lying over $\{0, \alpha\}$
then $\varphi$ sees $\cal{ M}$ if and only if $\cal{ M}=\mathbb
C k_0$ or $\cal{ M}=\mathbb Ck_\alpha$, that is, if and only if
$\cal{ M}$ is taut over $\{0, \alpha\}$.

For any closed subspace $\cal{ M}$ of $\cal{ H}$ we shall
denote by $P_\cal{ M}$ the orthogonal projection operator
from $\cal H$ to $\cal M$.

\begin{definition}\label{6.4}
{\rm Let $\cal{ M}, \cal{ M}_1, \cal{ M}_2, \dots$ be closed
subspaces of $\cal H$.  We say that $\cal{ M}_j$ {\it
converges to } $\cal{ M}$ (written $\cal{ M}_j \to \cal{ M}$) if
$\|P_\cal{ M} - P_{\cal{ M}_j}\|\to 0$ as $j\to \infty$.}
\end{definition}

The notion of a subspace lying over a finite subset, which we
introduced above,  and the observation about the visibility
of taut subspaces are in themselves of limited interest, but
they serve as intermediaries to the much richer concepts
appropriate to {\it infinite} subsets of $\Omega$.

Further we will use the well-known result that if $P$ and $Q$ are orthogonal projections 
with $\| P -Q \| < 1$, then $P$ and $Q$ have the same rank.

\begin{definition}\label{6.5}
{\rm Let $D\subset \Omega$.  A closed subspace $\cal{ M}$ of
$\cal{ H}$  {\it lies over} $D$ if there exists an integer
$n$ and a sequence $(F_j)$ of subsets of $D$ such that
$\#F_j \le n$ for each $j$ and a sequence $(\cal{ M}_j)$ of
subspaces of
$\cal H$ such that
\begin{enumerate}
\item[\rm (i)] $\cal{ M}_j$ lies over $F_j$ for all $j$, and
\item[\rm (ii)] $\cal{ M}_j \to \cal{ M}$.
\end{enumerate}
}
\end{definition}

Some remarks are in order.  Firstly, we must have $\dim
\cal{ M}_j\le n$ for each $j$, and so, if $\cal{ M}$ lies over any
set then $\cal{ M}$ is finite-dimensional.  Secondly, there is
a subtlety in the requirement that $\#F_j \le n$ for all $j$.  If
we were to replace this requirement in the foregoing
definition by the condition that $\dim \cal{ M}_j \le n$ for all
$j$, then we should find (provided $D$ is a set of  uniqueness
for $\cal H$) that {\it every} finite-dimensional subspace
of $\cal H$ lies over $D$.  The condition chosen is much
more restrictive: it entails a certain local character on
$\cal{ M}$, as will become clearer in Theorem \ref{6.17}
below.

\begin{example}\label{6.6} \rm
Let $\cal{ H}=H^2$, let $D$ be a compact subset of
$\mathbb D$ and let $\alpha \in \mathbb D\smallsetminus
D$.  It is trivial that the space $\mathbb C k_\alpha$
lies over $\mathbb D$.  It is not true, however, that
$\mathbb Ck_\alpha$ lies over $D$.  To see this, let
$$
\inf\limits_{\lambda \in D}
\left|\frac{\alpha-\lambda}{1-\overline
\lambda \alpha}\right| = c >0.
$$
Consider the set $F=\{\lambda_1, \dots, \lambda_n\}\subset
D$, and write
$$b(z) = \prod\limits^n_{j=1}
\frac{z-\lambda_j}{1-\overline \lambda_jz}.$$
We take the $\lambda_j$ to be distinct, so that 
$$\cal{ V}_F = \span\{k_{\lambda_1}, \dots,
k_{\lambda_n}\}=H^2 \ominus bH^2.$$
Thus, if $S$ is the shift operator, we have 
\begin{eqnarray*}
{\rm dist} (k_\alpha, \cal{ V}_F) &=& \|P_{\cal{ V}_F}^\perp
k_\alpha\|= \|P_{bH^2}k_\alpha\|\\ \\
&=& \|b(S) b(S)^* k_\alpha\|\\  \\
&=& \|\overline{b(\alpha)} b k_\alpha\| = |b(\alpha)|\
\|k_\alpha\|\\ \\
&=& \frac{|b(\alpha)|}{(1-|\alpha|^2)^{1/2}} \ge
\frac{c^n}{(1-|\alpha|^2)^{1/2}} >0.
\end{eqnarray*}
Thus it cannot be the case that $\cal{M}_j\to \mathbb
Ck_\alpha$ for any sequence $(\cal{ M}_j)$ such that $\cal{M}_j$ 
lies over an $n$-element subset of $D$.  \qed
\end{example}

Here is an alternative presentation of the notion introduced
in Definition \ref{6.5}.  Say that a vector $v\in \cal{ H}$ 
{\it lies over } $D$ if there exists an integer $n$, a
sequence of subsets $(F_j)$ of $D$ and a sequence of
vectors $(v_j)$ in $\cal H$ such that $\#F_j \le n$ for all $j,
v_j\in \cal{ V}_{F_j}$ for all $j$ and $v_j\to v$.  We leave to
the reader the verification of the following.

\begin{proposition}\label{6.7}
Let $D\subset\Omega$ and let $\cal{ M}$ be a
finite-dimensional subspace of $\cal H$.  The following are
equivalent.
\begin{enumerate}

\item[\rm (i)] $\cal{ M}$ lies over $D$;

\item[\rm (ii)] $\cal{ M}$ has a basis of vectors that lie over
$D$;

\item[\rm (iii)] every vector in $\cal{ M}$ lies over $D$.
\end{enumerate}
\end{proposition}
Not every vector in $\calh$ lies over a set $D\subseteq \Omega$.  Below we give some examples in $H^2$ on $\D$.  To this end we introduce a notion of {\em valency}.

\begin{definition} \label{valency}
A holomorphic function $f$ on $\D$ is said to be {\em $n$-valent} if, for all $w\in\C$, the equation $f(z)=w$ has at most $n$ solutions in the disc $\D$, counting multiplicities.
\end{definition}
It is easy to see that an equivalent condition is: a holomorphic function $f$ on $\D$ is $n$-valent if, for every positive $r < 1$ and for all $w\in\C$, the equation $f(z)=w$ has at most $n$ solutions in the disc $r\D$, counting multiplicities.

\begin{proposition} \label{fisnvalent}
If $(f_j)$ is a convergent sequence in $H^2$ with limit $f$ and, for some positive integer $n$, $f_j$ is $n$-valent for every $j \geq 1$, then $f$ is $n$-valent.
\end{proposition}
\begin{proof}
For all $w\in\C$ and $r\in (0,1)$, by the Argument Principle,
\begin{eqnarray}
\frac{1}{2\pi \mathrm{i}} \int_{r\T} \frac {f_j'(z)}{f_j(z)-w} \, dz &=& \text{the number of solutions in $r\D$ of } f_j(z)=w  \notag \\
	&\leq& n. \label{fjvalentn}
\end{eqnarray}
Since $f_j \to f$ in $H^2$, we have $f_j \to f$ and $f'_j \to f'$ {\em uniformly} on $r\T$ as $j \to \infty$.
 It follows that
\[
\frac {f_j'(z)}{f_j(z)-w} \to \frac {f'(z)}{f(z)-w}
\]
uniformly on $r\T$ as $j\to \infty$.  Hence, if we let $j\to \infty $ in the inequality \eqref{fjvalentn}, we obtain 
\[
\frac{1}{2\pi \mathrm{i}} \int_{r\T} \frac {f'(z)}{f(z)-w} \, dz \leq n. \notag
\]
Thus, again by the Argument Principle, the number of solutions of the equation $f(z)=w$ in $r\D$ is at most $n$, that is, $f$ is $n$-valent.
\end{proof}
\begin{corollary} \label{finitelyvalent}
If $f \in H^2$ lies over $\D$ then there exists a positive integer $n$ such that $f$ is $n$-valent.
\end{corollary}
\begin{proof}
Let $f \in H^2$ lie over $\D$.  Then there exists a positive integer $n$ and a sequence $(F_j)$ of subsets of $\D$ such that  $\#F_j \leq n$ and a function $f_j\in \calv_{F_j}$ for each $j$ such that
$f_j \to f$.  Since $f_j$ is the span of at most $n$ Szeg\H{o} kernels, $f_j$ is a rational function of degree at most $n$, and so is $n$-valent.  Hence, by Proposition \ref{fisnvalent}, $f$ is $n$-valent.
\end{proof}
Corollary \ref{finitelyvalent} provides an easy way to see that not all functions in $H^2$ lie over $\D$.  For example, it is immediate that any infinite Blaschke product does not lie over $\D$.
However, as we showed in Example \ref{loc-sub-intro}, the vector $v(z) =z$ does lie over $\mathbb D$.

What is the appropriate notion of lying tautly over
$D$?  By analogy with Proposition \ref{6.3}, it
should guarantee visibility to pseudomultipliers
defined on $D$.  The right notion is  the following,
but we vary the nomenclature: ``local" seems more
fitting than ``taut" here.

\begin{definition}\label{6.8}
{\rm Let $D \subset \Omega$. A closed subspace $\cal M$ of
$\cal H$ is {\it local over} $D$ if there exists an
integer $n$, a sequence of subsets $(F_j)$ of $D$
and a sequence of subspaces $(\cal{ M}_j)$ of $\cal H$
such that $\#F_j\le n$ for all $j$ and 
\begin{enumerate}
\item[\rm (i)] $\cal{ M}_j$ lies tautly over $F_j$ for each
$j$, and
\item[\rm (ii)] $\cal{ M}_j \to \cal{ M}$.
\end{enumerate}}
\end{definition}

We have emphasized the similarity to Definition
\ref{6.5}, but we could also have described local
subspaces slightly more economically, as the
following proposition shows.

\begin{proposition}\label{6.9}
Let $\cal H$ be a Hilbert function space on a set $\Omega$,
let $D\subset \Omega$ and let $\cal M$ be a closed 
subspace of $\cal H$.  The following are equivalent.
\begin{enumerate}
\item[\rm (i)] $\cal M$ is local over $D$;
\item[\rm (ii)] there exist an integer $n$ and a sequence
$(E_j)$ of subsets of $D$ such that $\#E_j\le n$ and
$\cal{ V}_{E_j}\to \cal{ M}$;
\item[\rm (iii)] $\cal M$ is finite-dimensional and there
exists a sequence $(E_j)$ of subsets of $D$ such that
$\cal{ V}_{E_j}\to \cal{ M}$.
\end{enumerate}
\end{proposition}

\begin{proof}  Condition (ii) is little more than a
rephrasing of the definition of a local space (see
Definition \ref{6.2} and \ref{6.8}).  Suppose (ii) holds. 
Then $\dim \cal{ V}_{E_j} \le n$ for all $j$, and since
$\cal{ V}_{E_j}\to \cal{ M}$ it follows that $\dim \cal{
M}\le n$, and so (iii) holds.  Conversely, suppose (iii)
holds.  It is not hard to see that $\cal{ V}_{E_j}$ must
have the same dimension as $\cal M$ for all
sufficiently large $j$.  For such $j$ we may pick $F_j
\subset E_j \subset D$ such that $\cal{ V}_{F_j} = \cal{
V}_{E_j}$ and $\#F_j = \dim \cal{ M}$.  Then $\cal{
V}_{F_j}\to \cal{ M}$ and so (ii) holds.\end{proof}

\vspace{2mm}
Thus the local subspaces of $\cal H$ over $D$ are
precisely the finite-dimensional  limits of the spaces
of the form $\cal{ V}_F, F\subset D$.

\begin{example}\label{6.10} \rm
In $H^2$, a subspace which lies over $\mathbb D$ but
is not local over $\mathbb D$ is $\mathbb Cv$, where
$v(z) =z$.  We have already seen that $v$ lies over
$\mathbb D$.  Suppose there is a sequence $(E_j)$ of
subsets of $\mathbb D$ such that $\cal{ V}_{E_j} \to
\mathbb Cv$.  We must have $\#E_j=1$ for
sufficiently large $j$, hence $\cal{ V}_{E_j} = \mathbb C
k_{\alpha_j}$ for some $\alpha_j\in \mathbb D$.   If
$P_{E_j} v=c_jk_{\alpha_j}$, we must have
$c_jk_{\alpha_j}\to v$ in $H^2$.  On considering the
first two Taylor coefficients we get the contradictory
conclusions $c_j\to 0$ and $c_j \overline \alpha_j \to
1$.  Thus $\mathbb Cv$ is not local over $\mathbb D$.
\end{example}

\begin{theorem}\label{6.11}
Let $\varphi$ be a pseudomultiplier.  If $\cal M$ is a
subspace of $\cal H$ which is local over $D_\varphi$ 
then $\varphi$ sees $\cal M$.
\end{theorem}

\begin{proof}  By Proposition \ref{6.9}, there is a sequence $(F_j)$ of subsets
of $D_\varphi$ and a sequence $(\calv_{F_j})$ of
subspaces of $\cal H$ such that $\#F_j \le n \in \mathbb
N$ and $\calv_{F_j} \to \cal{ M}$. 
By Example \ref{1.8}, for every $\lam\in D_\phi$, $k_\lam$ is visible for $\phi$. 
Since the $F_j$ are finite sets, and by Example  \ref{viorthogonal}, the span of a finite collection of visible vectors is a visible space,  $\varphi$ sees $\calv_{F_j}$. 
Hence, by Proposition \ref{5.5},
$$
X^*_\varphi \calv_{F_j} \subset \calv_{F_j} +
E^\perp_\varphi,\quad j=1, 2, \dots
$$
Since $\calv_{F_j} \to \cal{ M}$, it follows that
$$X^*_\varphi \cal{ M} \subset \cal{ M} + E^\perp_\varphi.
$$
By Proposition \ref{6.9}, $\calm$ is finite-dimensional, and is therefore closed in $\calh$.
By Proposition \ref{5.5}, $\varphi $ sees $\cal{ M}$.
\end{proof}

\section{Projectively complete Hilbert function spaces and local subspaces}\label{local_spaces}

Our next task is to clarify the sense in which a subspace
which is local over a set $D$ can genuinely be attached to a
set of points in $D$.  A further assumption on $D$ will be
required, for consider the example of a proper
dense subset $D$ of $\mathbb D$.  If $\alpha \in \mathbb D
\smallsetminus D$ then it is clear that the subspace
$\mathbb Ck_\alpha$ of $H^2$ is local over $D$, yet
$\mathbb Ck_\alpha$ as a subspace of $H^2$ (as a function
space on $D$) is not associated with a point of the domain
of the function space.  This example suggests we confine
attention to domains $\Omega$ which are complete in
some sense.  In fact the appropriate notion of
completeness turns out to be a little stronger than the
natural one given in Definition \ref{6.1}.  Subject to this
completeness assumption, it transpires that the structure
of local subspaces is remarkably transparent (Theorem
\ref{6.18} below).

\begin{lemma}\label{6.12}
Let $\cal{ M}, \cal{ M}_1, \cal{ M}_2, \dots$ be closed subspaces of
$\cal H$ such that
$\dim \cal{ M}<\infty$ and $\cal{ M}_j \to \cal{ M}$.  If
$(u_j)$ is a bounded sequence with $u_j\in \cal{ M}_j$ for
each $j$ then there exists
$u\in \cal{ M}$ and a sub-sequence $(u_{j_k})$ of $(u_j)$ such
that $u_{j_k}\to u$ as $k\to \infty$.
\end{lemma}

\begin{proof}  Pick a sub-sequence $(u_{j_k})$ of
$(u_j)$ such that $u_{j_k}$ converges weakly to an element
$u\in \cal{ H}$.  We have
$$u-u_j = u - P_\cal{ M}u+ P_\cal{ M}(u-u_j) + (P_\cal{
M}-P_{\cal{ M}_j})u_j.$$
Let $j\to \infty$ along the sequence $(j_k)$.  Since
$u-u_j\to 0$ weakly and $\cal{ M}$ is finite-dimensional,
$P_\cal{ M}(u-u_j)\to 0$ in norm.  Since $(u_j)$ is bounded
and $P_{\cal{ M}_j}\to P_\cal{ M}$, the final term on the
right hand side also tends to $0$ in norm.  Thus
$(u-u_{j_k})$ converges in norm, and its limit must be its
weak limit, which is $0$.  Thus $u=P_\cal{ M}u$, that is,  $u\in
\cal{ M}$ and $u_{j_k}\to u$. \end{proof}

To obtain our structure theorem for local subspaces we
need to introduce {\it projective Hilbert space} $P\cal{
H}$. This is the set of one-dimensional subspaces of the
Hilbert space $\cal{ H}$, with a natural metric $p$ defined
as follows.  For each one-dimensional subspace $\mathbb
Cx$ of $\cal H$, where $x\in \cal{ H}$ and $\|x\|=1$, we may
consider its intersection $\mathbb Tx=\{\lambda x:
|\lambda|=1\}$ with the unit sphere $\cal{ H}_1$ of $\cal H$. 
This gives a one-one correspondence between the
elements of $P\cal{ H}$ and a subset of the set $\cal{ C}$ of
all closed subsets of $\cal{ H}_1$.  We may give $\cal C$ the
Hausdorff metric $\rho$ induced by the standard metric of
$\cal H$ restricted to $\cal{ H}_1$.  This induces a metric
$p$ on $P\cal{ H}$.  Note that, for $x, y\in \cal{ H}_1$,
\begin{eqnarray*}
\rho(\mathbb Tx, \mathbb Ty)&=&
\inf\limits_{\lambda
\in
\mathbb T} \|x-\lambda y\|\\
&=& 
\inf\limits_{\lambda \in \mathbb T} \{\|x\|^2 + \|\lambda
y\|^2 - 2 {\rm Re} \overline \lambda \ip{x} {y}\}^{1/2} \\
&=& \{1-|\ip{x} {y}|^2\}^{1/2}. 
\end{eqnarray*}
Hence the metric $p$ on $P\cal{ H}$ is given explicitly by
\begin{equation}\label{24}
p(\mathbb C x, \mathbb C y) =  \{1-|\ip{x}{y}|^2\}^{1/2},\quad x, y \in \cal{ H}_1.
\end{equation}
We note that there are other natural ways of defining a metric on $P\cal H$.
For example, we could identify the one-dimensional subspace $\C x$, for $x
\in \cal{H}_1$, with the orthogonal projection operator on $\C x$.  The
operator norm would then induce a metric on $P\cal H$.  A little calculation
then gives the distance between $\C x$ and $\C y$ as  $\sqrt{1 - |\ip{x}{y}|^2}$. 
This is identical to $p$.

 $P\cal{ H}$ is complete with respect to
$p$ (if $(\mathbb C x_j)$ is a Cauchy sequence we can
construct a sequence $(y_j)$ in $\cal{ H}_1$ with $y_j \in
\mathbb C x_{n_j}$ and $\|y_{j+1} - y_j\| <2^{-j}$; then
$(y_j)$ is Cauchy in $\cal{ H}_1$, and if $y$ is its limit,
$\mathbb C x_j \to \mathbb C  y)$.  The canonical mapping
$\cal{ H}\smallsetminus\{0\}\to P\cal{ H}: x\mapsto
\mathbb C x$ is easily shown to be continuous, so if $(x_j),
(c_j)$ are sequences in $\cal{ H}\smallsetminus\{0\},
\mathbb C$ respectively and $c_j x_j \to x$ in $\cal H$ then
$\mathbb C x_j\to \mathbb C x$.  Conversely, if $\mathbb C
x_j \to \mathbb C x$ in $P\cal{ H}$, then the calculation
above shows that
$$\left\|\lambda_j \frac{x_j}{\|x_j\|} - \frac{x}{\|x\|}
\right\|\to 0$$
where $\lambda_j$ in $\mathbb T$ is chosen so that
$\lambda_j \ip{x_j}{x} >0$.  We infer that $\mathbb C x_j \to
\mathbb C x$ if and only if there is a sequence $(c_j)$ in
$\mathbb C\smallsetminus\{0\}$ such that $c_j x_j \to x$. 

\begin{definition}\label{H-proj-comp} \rm
A Hilbert function space
 $(\cal{H},\Omega)$ is {\it projectively complete}
for $\cal H$ if the set $\{{\mathbb C}{k_\lambda}:\lambda
\in \Omega, k_\lambda \neq 0\}$ is closed in $P\cal{ H}$.
In this case we shall also say that $\O$ is projectively complete for $\cal
H$. 
\end{definition}
Since $P \cal H$ is a complete metric space, so too is 
$\{{\mathbb C}{k_\lambda}:\lambda
\in \Omega, k_\lambda \neq 0\}$
 with respect to the restriction of
$p$ whenever $(\cal{H},\Omega)$ is projectively complete. 
 
\begin{proposition}\label{6.13}
Let $\cal H$ be a Hilbert function space on $\Omega$.  Then
$\Omega$ is projectively complete for $\cal H$ if and only if
whenever $(c_j), (\alpha_j)$ are sequences in $\mathbb C,
\Omega$ respectively and
$$
c_j k_{\alpha_j} \to u \in \cal{ H} \setminus \{0\},
$$
there exist $c\in \mathbb C$ and $\alpha \in \Omega$ such
that 
$$u=ck_\alpha.$$
\end{proposition}

\begin{proof}
Suppose $\O$ is projectively complete for $\cal H$ and $c_j k_{\a_j}
\to u \neq 0$.
Then almost every $k_{\a_j}\neq 0$ and, by
 continuity of the canonical mapping
$\cal{H}\setminus \{0\} \to P\cal H$, we have $\C k_{\a_j} \to \C u$.
Thus $\C u$ lies in the closure of $\{\C k_\lambda: \lambda \in \O,
k_\lambda \neq 0 \}$
and so, by the definition of projective completeness, $\C u\in\{\C
k_\lambda: \lambda\in \O,k_\lambda \neq 0 \}$.  
Hence $u=ck_\a$ for some scalar $c$ and some $\a\in\O.$

Conversely, suppose $ (\cal{H},\O)$ is such that $c_j k_{\a_j} \to u
\neq 0$ implies that $ u=ck_{\a}$ for some $c\in \C $ and $\a \in \O$.
Consider any element $\C x, x\in \cal{H}\setminus \{0\},$ of the closure of
$\{\C k_\lambda:\lambda \in\O, k_\lambda \neq 0\}$ in $ P\cal H$.  
Pick a sequence $
(\a_j)$ in $\O$ such that $\C k_{\a_j} \to \C x$.  As noted above, there
is a sequence $ (c_j)$ of non-zero scalars such that $c_j k_ {\a_j}
\to x$.  By hypothesis, $x=ck_\a$ for some $c\neq 0$ and $\a\in \O$.
Hence $\C x \in \{\C k_ \lambda: \lambda \in\O ,k_\lambda \neq 0 \}$.  
Thus $\O$ is projectively complete for $\cal H$. 
\end{proof} 

If the mapping $\alpha \mapsto  \C k_\alpha$ maps $\O$ injectively into
$\cal H \setminus \{0\} $
then $p$ induces a metric $d$ on $\O$, given by
$$
d(\lambda,\mu) =p(\mathbb Ck_\lambda, \mathbb Ck_\mu) = 
\left\{ 1 - \frac{|\ip{k_\lambda}{k_\mu}|} {\|k_\lambda\|\
\|k_\mu\|}\right\}^{1/2}.
$$

\begin{example}\label{6.14} \rm
$\mathbb D$ is projectively complete for $H^2$.  By
(\ref{23}), for $\lambda, \mu \in \mathbb D$ we have
\begin{eqnarray*}
p(\mathbb Ck_\lambda, \mathbb Ck_\mu) &=& 
\left\{ 1 - \frac{|\ip{k_\lambda}{k_\mu}|} {\|k_\lambda\|\
\|k_\mu\|}\right\}^{1/2}\\
&=&
\left\{ 1 - \frac{(1-|\lambda|^2)^{1/2} 
(1-|\mu|^2)^{1/2}}
{|1-\overline \lambda \mu|}\right\}^{1/2}\\
&=&
\left\{ 1 - \frac{(1-|\lambda|^2)(1-|\mu|^2)}
{|1-\overline \lambda \mu|^2}\right\}^{1/2}  h(\lambda, \mu) \\
&=&
\left|  \frac{\lambda-\mu}
{1-\overline \lambda \mu}\right| h(\lambda, \mu)
\end{eqnarray*}
where
$$
h(\lambda, \mu)=\left\{ 1 +
\frac{(1-|\lambda|^2)^{1/2}(1-|\mu|^2)^{1/2}}
{|1-\overline \lambda \mu |}\right\}^{-1/2} .
$$ 

We have $1/\sqrt{2} \le |h| \le 1$,  and so the metric induced
on $\mathbb D$ by $p$ is uniformly equivalent to the standard
pseudo-hyperbolic distance \cite[Sec. 1.4.4]{K}, for which $\mathbb D$ is
complete. \qed
\end{example}

\begin{example}\label{zH2} \rm
$\mathbb D$ is not projectively complete for $zH^2$, the subspace of $H^2$
consisting of functions which vanish at $0$.  The reproducing kernel here is
$$
k(\lambda, \mu) = \frac{\bar{\mu}\lambda}{1 -\bar{\mu}\lambda}.
$$
A similar calculation to that in the preceding example shows that
$\{{\mathbb C}{k_\lambda}:\lambda
\in \Omega, k_\lambda \neq 0\}$ is uniformly equivalent to 
$\D \setminus \{0\}$ with
the pseudo-hyperbolic metric, and this is clearly an incomplete metric space.
Alternatively, one may apply Proposition \ref{6.13} with $c_j = j,\a_j= 1/j, u(\lambda) = \lambda$.
\end{example}

\begin{example}\label{sob} 
The closed unit interval is projectively complete for the Sobolev space
$W^{1,2}[0,1]$ of Example \ref{2.3}. \rm To see this fact note that the reproducing kernel of 
$W^{1,2}[0,1]$ is given by equation \eqref{kerW},
from which it is clear that the kernel  does not vanish anywhere on $[0,1]^2$.
Since $[0,1]$ is a compact metric space for the natural topology and
elements of $W^{1,2}[0,1]$ are continuous functions, we can appeal to the
following observation to conclude that $[0,1]$ is projectively complete 
for $W^{1,2}[0,1]$.
\end{example}

\begin{proposition} \label{compactness}
Let $\O$ be a compact metric space and let $\cal H$ be a Hilbert space
of continuous functions on $\O$.  Suppose that the reproducing kernel of
$\cal H$ has no zero on $\O \times \O$. Then $\O$ is projectively complete 
for $\cal H$.
\end{proposition}
\begin{proof} Apply Proposition \ref{6.13}. Suppose that $c_j k_{\a_j} \to
u \neq 0$ in $\cal H$ for some scalars $c_j$ and points $\a_j \in \O$.
Since $\O$ is compact metric we may assume (passing to a 
subsequence if necessary) that $(\a_j)$ converges to some point $\beta \in
\O$.  Consider any point $\lambda \in \O$.  We have
\begin{eqnarray*}
u(\lambda) &=& \ip{u}{k_\lambda} = \lim_j \ip{c_j k_{\a_j}}{k_\lambda} \\
	   &=& \lim_j c_j k(\lambda,\a_j).
\end{eqnarray*}
Since $u \neq 0$ there is some $\lambda$ for which $u(\lambda) \neq 0$.
In view of the fact that $k(\lambda,\a_j) \to k(\lambda,\beta) \neq 0$,
it must be that the
sequence $(c_j)$ converges to a non-zero limit $c$ in $\C$.  Hence we have,
for any $\lambda \in \O$,
$$
u(\lambda) = c k(\lambda, \beta) = c k_\beta(\lambda).
$$
Thus $u = c k_\beta$. By Proposition \ref{6.13}, 
$\O$ is projectively complete for $\cal H$.
\end{proof}

\begin{definition}\label{6.15} \rm
Let $\cal H$ be a Hilbert function space on $\Omega$ and
let $\cal{ M}$ be a local subspace of $\cal{ H}$ over
$\Omega$ of dimension $n$.  We say that $\alpha \in
\Omega$ is a {\it support point } of $\cal{ M}$ if there exist
a sequence $(\alpha_j)$ in $\Omega$, a sequence $(c_j)$ in
$\mathbb C$ and a sequence $(E_j)$ of subsets of $\Omega$
such that $\#E_j = n, \alpha_j \in E_j, \cal{ V}_{E_j}\to \cal{
M}$ and $c_jk_{\alpha_j}\to k_\alpha$.  The {\it support} of
$\cal{ M}$ is the set of support points of $\cal{ M}$.  It will
be denoted by $\supp \cal{ M}$.
\end{definition}

Note that if $\alpha \in \supp \cal{ M}$ then $k_\alpha \in \cal{ M}$.
For if $c_j, k_{\alpha_j} $ are as in the definition then, by Lemma
\ref{6.12}, the sequence $(c_j k_{\alpha_j}) $ has a subsequence which
converges to an element $u$ of $\cal M$.  We must have $u = k_\alpha$, and
so $k_\alpha \in \cal M$.

\begin{example}\label{6.16} \rm
Let $\cal{ M}=\span\{1, v \}$ in $H^2$ on $\mathbb D$, where
$v(\lambda)=\lambda$.  We have already seen in Example \ref{loc-sub-intro} and in the lines following it that $\cal{ M}$ is local over $\mathbb D$ and that $0$ is a
support point of $\cal{ M}$.  No other point of $\mathbb D$
is a support point of $\cal{ M}$, for if $\alpha \in \D$ and $\alpha \neq 0$
then $k_\alpha \not \in \cal M$, and hence $\alpha \not \in \supp \cal M$. 
Hence $\supp \cal{ M}=\{0\}$.
\end{example}

\begin{definition}\label{6.17} \rm
A local subspace $\cal{ M}$ of $\cal{ H}$ is {\it punctual} if
$\supp \cal{ M}$ consists of a single point.  
\end{definition}

\begin{definition}\label{6.171} \rm
Let $\calh$ be a Hilbert function space over a set $\O$.
For a subset $D$ of $\Omega$ we define the {\it projective
closure} of $D$ in $\Omega$ to be the set 
\[ 
pc(D) = \{\lambda \in \Omega: \mathbb Ck_\lambda \; \text{is in the closure of} \;
\{ \mathbb C k_\alpha: \alpha \in D\} \; \text{
in} \; P\cal{ H}\}.
\]
\end{definition}

\begin{lemma}\label{exist-subspace}
Let $\calh$ be a Hilbert space and let $\calm$ be a subspace of $\calh$ of finite dimension $m$.  
Let $(\calv_i)_{i\geq 1}$ be a sequence of subspaces of $\calh$ which converges to $\calm$ as $i\to\infty$. 
Let $k$ be a positive integer no greater than $m$. 
For every positive integer $i$ let $\calw_i$ be a $k$-dimensional subspace of $\calv_i$.  Then there exists a sequence $(i_q)_{q\geq 1}$ of positive integers and a $k$-dimensional subspace $\calw$ of $\calm$ such that $P_{\calw_i} \to P_\calw$ in norm as $i\to\infty$ along the sequence $(i_q)_{q\geq 1}$.
\end{lemma}
\begin{proof}
For each $i$ let $e_i^1, \dots,e_i^k$ be an orthonormal basis of $\calw_i$. Since $\calm$ is finite-dimensional and each $e_i^1$ is a unit vector, there is a sequence $(i_q)_{q\geq 1}$ of positive integers and a vector $u^1\in\calm$ such that $P_\calm e_{i_q}^1\to u^1$ as $q\to \infty$.  
 Passing to a sub-sequence of $(i_q)_{q\geq 1}$, we may assume that $P_\calm e_{i_q}^2\to u^2$ as $q\to \infty$ for some vector $u^2\in\calm$ of norm at most $1$.  Continuing in this way, we arrive at vectors $u^1,u^2,\dots,u^k\in\calm$ such that $P_\calm e_{i_q}^\ell\to u^\ell$ as $q\to \infty$ for $\ell=1,2,\dots,k$.   
   Since $e_i^\ell \in\calw_i$ and $\calw_i\subseteq\calv_i$, for each positive integer $i$ and $\ell=1,2,\dots,k$,
   \[
   e_i^\ell = P_{\calv_i}e_i^\ell = (P_{\calv_i}-P_\calm)e_i^\ell + P_\calm e_i^\ell.
   \]
   In this equation let $i \to \infty$ along the sequence $(i_q)$ and note that $\|P_{\calv_i}-P_\calm\| \to 0$ to conclude that $e_{i_q}^\ell \to u^\ell$ as $q \to\infty$.
 
 Let $\calw =\text{span}\{u^\ell:1\leq\ell\leq k\}\subseteq \calm$, and notice that $u^1,u^2,\dots,u^k$ constitutes an orthonormal basis of $\calw$.
Hence $\dim \calw = k$ and
\begin{eqnarray}
\|P_{\calw_i} - P_\calw \| &=& \| \sum_{\ell=1}^k e_i^\ell \otimes  e_i^\ell - \sum_{\ell=1}^k u^\ell \otimes u^\ell \| \notag \\
	&=& \|\sum_{\ell=1}^k [(e_i^\ell-u^\ell)\otimes e_i^\ell + u^\ell\otimes(e_i^\ell - u^\ell)]\| \notag \\
	&\leq& 2\sum_{\ell=1}^k \|e_i^\ell - u^\ell\| \notag \\
	&\to& 0 \quad\text{as}\; \; i \to \infty \;\; \text{along the sequence} \; (i_q)_{q\geq 1} \notag
\end{eqnarray}
as claimed.
\end{proof}

The final result of the paper states that, in a projectively complete Hilbert function
space, every local space is a finite sum of punctual spaces.

\begin{theorem}\label{6.18}
Let $\cal{ H}$ be a projectively complete Hilbert function
space on $\Omega$, let $D\subset \Omega$ and let $\cal{
M}$ be a local space over $D$.  There exist subspaces $\cal{
M}_1, \dots, \cal{ M}_N$ of $\cal{ H}$ such that
\begin{itemize}
\item[\rm (i)] $\cal{ M}_i$ is punctual for $i=1, \dots, N$;

\item[\rm (ii)] $\supp \cal{ M}_i= \{\alpha_i\} \subset {\rm pc}(D) \mbox{ for } i=1, \dots, N$;

\item [(iii)]$\cal{ M} = \cal{ M}_1+\cal{ M}_2+\dots + \cal{
M}_N$;

\item[\rm (iv)] $\sum\limits^N_{i=1} \dim \cal{ M}_i = \dim
\cal{ M}$, and

\item[\rm (v)] $\supp \cal{ M}=\{\alpha_1, \dots, \alpha_N\}$.
\end{itemize}
\end{theorem}

\begin{proof} 
By Proposition \ref{6.9}, since  $\cal M$ is local over $D$,  the dimension ($m$, say) of $\cal M$ is finite and  there exist a positive integer $n$ such that $m \le n$ and a sequence
$(E_j)$ of subsets of $D$ such that $\#E_j\le n$ and
$\cal{ V}_{E_j}\to \cal{ M}$. Evidently, for all sufficiently large $j$,
$ \dim \cal{ V}_{E_j} =\dim \cal M = m$, and hence there is a subset $F_j$ of $E_j$ having $m$ elements such that $\cal{V}_{F_j}=\cal{V}_{E_j}$.  Clearly $\{k_\lam: \lam \in F_j\}$ is a linearly independent set.
Let $F_j = \{\lam^\ell_j:\ell=1, \dots,m\}$\footnote{In the commonly studied Hilbert function spaces the kernels $\{k_\lam: \lam \in \O\}$ are linearly independent, in which case $E_j=F_j$.  However, we do not need to exclude spaces for which some kernels are linearly dependent.}.  Then
$$
\cal{ V}_{E_j} = \span \{k_{\lambda_j^\ell}:  \ell =1, \dots, m\}.
$$ 

By Lemma \ref{6.12}, since $\dim \cal M < \infty$ and $\cal{ V}_{E_j} \to \cal{ M}$,
for any bounded sequence
$(u_j)$, with $u_j \in \cal{ V}_{E_j}$ for
each $j$, there exists
$u\in \cal{ M}$ and a strictly increasing sequence $({j_t})_{t\geq 1}$ of positive integers such
that $u_{j_t}\to u$ as $t\to \infty$.

For $\ell=1,2,\dots,k$ and all $j\geq 1$ let consider the bounded sequence $(u^\ell_j)$, where
$u_j^\ell = \frac{k_{\lambda_j^\ell}}{\|k_{\lambda_j^\ell}\|}$.
It is clear that  $u_j^1 \in  \cal{ V}_{E_j}$ for each $j$. By Lemma \ref{6.12},  there exists
$u^1 \in \cal{M}$ and a strictly increasing sequence $(j_{1t})_{t\geq 1}$ of positive integers  such
that $u^1_{j_{1t}}\to u^1$ as $t\to \infty$.

For $\ell=2$, consider the bounded sequence $(u^2_{j_{1t}})$, so that
 $u_{j}^2 \in  \cal{ V}_{E_{j}}$ for each $j\geq 1$. By Lemma \ref{6.12},  there exists
$u^2 \in \cal{M}$ and a strictly increasing sub-sequence $(j_{2t})$ of $(j_{1t})$ such
that $u^2_{j_{2t}}\to u^2$ as $t \to \infty$.

We repeat this argument for  $\ell= 1, \dots, m$.  We obtain strictly increasing sequences $(j_{\ell t})_{t\geq 1}$ of positive integers for  $\ell = 1, \dots m$ such that $(j_{(\ell+1) t})_{t\geq 1}$ is a sub-sequence of $(j_{\ell t})_{t\geq 1}$ for $ \ell = 1, \dots m-1$, and elements $u^1, \dots u^m$ of $\cal{M}$ such that 
\[
u^\ell_{j_{\ell t}} \to u^\ell \mbox{ as } t \to \infty
\]
Since $(j_{mt})_{t\geq 1}$ is a sub-sequence of $(j_{\ell t})_{t\geq 1}$ for $\ell= 1, \dots m$, it follows that 
\[
u^\ell_{j_{m t}} \to u^\ell \mbox{ as } t \to \infty \mbox{ for } \ell= 1, \dots m.
\]

By assumption, $\cal{ H}$ is a projectively complete Hilbert function
space on $\Omega$ and hence $\Omega$ is projectively complete. 
According to Proposition \ref{6.13}, for a projectively complete $\Omega$,  
whenever $(c_{j})_{j=1}^\infty, (\lambda_j)_{j=1}^\infty$ are sequences in $\mathbb C$ and $\Omega$ respectively and
$$
c_{j} k_{\lambda_j} \to u \in \cal{ H} \setminus \{0\}, \; \text{as} \; j \to \infty,
$$
there exist $c \in \mathbb C$ and $\alpha \in \Omega$ such
that 
$$u=c k_{\alpha}.$$
Therefore, for $\ell=1, \dots, m$, there exist $c^\ell \in \mathbb C$ and $\beta_\ell \in pc(D)$ such
that $u^\ell=c^\ell k_{\beta_\ell}.$ Hence
$$
\frac{k_{\lambda_{j_{mt}^\ell}}}{\|k_{\lambda_{j_{mt}}^\ell}\|} \to c^\ell k_{\beta_\ell} \; \text{as} \; t \to \infty \; \mbox{ for } \ell= 1,\dots,m.
$$

By Definition \ref{6.4},
$\cal{V}_{E_j} \to \cal{M}$ means that 
$$ 
\| P_{\cal{M}} - P_{\cal{V}_{E_j}} \| \to 0 \text { as } j \to \infty.
$$
Since the $k_{\lam_j^\ell}$ span $\cal{V}_{E_j}$, for every positive integer $j$ and each $v \in \cal{H}$, there exist scalars $c_j^\ell, \ell= 1, \dots, m$, such that
\[
 P_{\cal{V}_{E_j}} v = \sum_{\ell=1}^m c_{j}^\ell \frac{k_{\lambda_{j}^\ell}}{\|k_{\lambda_j^\ell}\|}.
 \]
 Hence, for every $v \in \cal{M}$,  
$$
 v = P_{\cal{M}} v = \lim_{t \to \infty} P_{\cal{V}_{E_{j_{mt}}}} v = \lim_{t \to \infty} 
\sum_{\ell=1}^m c_{j_{mt}}^\ell  \frac{k_{\lambda_{j_{mt}}^\ell}}{\|k_{\lambda_{j_{mt}}^\ell}\|},
$$
Since 
\begin{equation} \label{kernels}
\frac{k_{ \lambda_{j_{mt}}^\ell} }{\|k_{\lambda_{j_{mt}}^\ell}\|} \to c^\ell k_{\beta_\ell} \; \text{as} \;\; t \to \infty,
\end{equation}
the support of $\cal{M}$ is the set $ \{\beta_1, \dots, \beta_m\}$ and $k_{\beta_1}, \dots , k_{\beta_m} \in \calm$.  Here we can have $\beta_i =\beta_j$ for some $i \neq j$. As we have shown in Example \ref{6.16}, the dimension of a local subspace $\cal{V}$ with $\supp \cal{V}$ consisting of a single point $\{\beta\}$ can be greater than one. Let $N$ be the number of  distinct points in $ \{\beta_1, \dots, \beta_m\}$, so that we may write $\mathrm{supp}~\calm = \{\beta_1, \dots, \beta_m\} = \{\alpha_1, \dots, \alpha_N\}$ for some points $\alpha_i \in D$.
For each integer $i$ such that $1\leq i\leq N$,  let $J_i$ be the set of indices $j\in[1,m]$ such that $\beta_j=\alpha_i$ and, for each positive integer $q$, let $F^i_q=\{\lambda^\ell_q:\ell\in J_i\}$, so that
$\cal{V}_{F^i_q} = \text{span}\{k_{\lambda^\ell_q}:\ell\in J_i\}$. Let  $\dim \cal{ V}_{F^i_{j_q}} = m_i=\# J_i$,  so that $\sum\limits^N_{i=1} m_i= m = \dim \cal M$.

By Lemma \ref{exist-subspace}, there exists  a subspace $\calm_i$ of $\calm$ which is  the limit of  $\cal{V}_{F^i_q}$  as $q\to\infty$ along some increasing sequence $(j_{t})_{t\geq 1}$ of positive integers. Note that
$\dim {\cal M}_i = \dim \cal{ V}_{F^i_{j_q}} = m_i$. By Proposition \ref {6.9},  $\cal{ M}_i$ is a local subspace over $D$.
By equation \eqref{kernels}, 
$
\frac{k_{ \lambda_{j_{mt}}^\ell} }{\|k_{\lambda_{j_{mt}}^\ell}\|} \to c^\ell k_{\alpha_i} \; \text{as} \;\; t \to \infty \text{ for } \ell \in J_i.
$
Thus the support of $\calm_i$ is $\{\alpha_i\}$, and $\calm_i$ is punctual for $i=1,\dots,N$.
Clearly $\sum\limits^N_{i=1} \dim \cal{ M}_i = \dim \cal{ M}$, and therefore
$\cal{ M} = \cal{ M}_1+\cal{ M}_2+\dots + \cal{M}_N$. 
\end{proof}

\end{document}